\documentclass[a4paper]{article}

\usepackage[english]{babel}

\usepackage[utf8]{inputenc}
\usepackage[utf8]{inputenc}
\usepackage{amsmath}
\usepackage{graphicx}
\usepackage{amssymb}
\usepackage{amsthm}
\usepackage{tikz-cd}
\usepackage{mathrsfs}
\usepackage[colorinlistoftodos]{todonotes}
\usepackage{enumitem}
\usepackage{yfonts}
\usepackage{ dsfont }
\usepackage{MnSymbol}
\usepackage{slashed}

\title{Mukai duality on adiabatic coassociative fibrations }

\author{Yang Li 	\thanks{Y.L. is supported by the Engineering and Physical Sciences Research Council [EP/L015234/1], the EPSRC Centre for Doctoral Training in Geometry and Number Theory (The London School of Geometry and Number Theory), University College London. The author is also funded by Imperial College London for his PhD studies. }}

\date{\today}
\newtheorem{thm}{Theorem}[section]
\newtheorem{lem}[thm]{Lemma}

\theoremstyle{definition}
\newtheorem{eg}[thm]{Example}

\newtheorem{cor}[thm]{Corollary}

\newtheorem*{rmk}{Remark}
\newtheorem{prop}[thm]{Proposition}
\newtheorem{Def}[thm]{Definition}
\newtheorem*{Question}{Question}

\newtheorem*{Acknowledgement}{Acknowledgment}

\newcommand{\ie}{\emph{i.e.} }
\newcommand{\cf}{\emph{cf.} }

\newcommand{\R}{\mathbb{R}}

\newcommand{\Lap}{\Delta}

\DeclareMathOperator{\Hom}{Hom}
\DeclareMathOperator{\End}{End}
\DeclareMathOperator{\Lie}{Lie}

\DeclareMathOperator{\Tr}{Tr}

\begin{document}
	\maketitle
	
\begin{abstract}
This paper studies the formal adiabatic limit of coassociative K3 fibred torsion free $G_2$ manifolds fibred over a contractible base, shows how to put this structure on a different fibration obtained by fibrewise performing Mukai duality of K3 surfaces, and furthermore relates the gauge theories on both fibrations by a Nahm transform. This gives a mathematical interpretation to the physical speculations of Gukov, Yau and Zaslow.
\end{abstract}

\section{Introduction}

This paper studies three circles of interrelated questions about torsion-free $G_2$ manifolds with a coassociative K3 fibration $\pi: M\to B$ over a contractible local base $B$. These are \textbf{formal adiabatic structures}, \textbf{duality} and \textbf{gauge theory}.

In \cite{Donaldson}, Donaldson proposed the study of  formal   adiabatic  structures for such K3 fibrations, meaning that there is a parameter $\epsilon$ controlling the size of the fibre in this problem, and one takes the formal limit as $\epsilon\to 0$ of the structural equations. \textbf{Donaldson's adiabatic fibration}, which shall be surveyed more fully in the text, involves 3-forms $\underline{\omega}$, $\underline{\lambda}$ and 4-forms $\underline{\Theta}$ and $\underline{\mu}$ on $M$ described pointwise by a linear algebraic model, and an Ehresmann connection $H$ on $\pi: M\to B$, satisfying the adiabatic equations
\begin{equation}
\begin{cases}
d_f \underline{\omega}=0, \quad d_H \underline{\omega}=0, \quad d_f \underline{\lambda}=0, \\
d_H\underline{\mu}=0, \quad d_f \underline{\Theta}=0, \quad d_H\underline{\Theta}=0,
\end{cases}
\end{equation}
In particular, the K3 fibres are endowed with \textbf{hyperK\"ahler} structures.
Furthermore, these data are encoded into a positive section satisfying the \textbf{maximal submanifold equation} (\cf reviews in Section \ref{Donaldsonadiabaticfibration}).

This propels one to ask further

\begin{Question}
Can we give an adiabatic description of various geometric objects over $M$, such as $G_2$ instantons, $G_2$ monopoles, associative sections, the Levi-Civita connection, and the spin structure?
\end{Question}

As a preliminary discussion, 
we offer a unified formal treatment via a basic linear algebraic model, and write down the limiting equations of these objects (\cf Chapter \ref{Adiabaticlimitingstructures}). Solutions of the limiting equations are refered to as `adiabatic' objects. We show, among many other results, that on Donaldson's adiabatic fibrations, there is an adiabatic analogue of the well known characterisation of $G_2$ manifolds in terms of the existence of parallel spinors (\cf Chapter \ref{Adiabaticspinstructures}); the adiabatic associative sections are governed by the Fueter equation; 
the adiabatic $G_2$ instantons are essentially equivalent to adiabatic $G_2$ monopoles.

\begin{rmk}
This is formal in the sense that no analytic result is yet proven concerning when an adiabatic solution can be perturbed into a genuine finite $\epsilon$ solution, or whether a sequence of finite $\epsilon$ solutions would necessarily converge in any analytic sense to some adiabatic solution.
\end{rmk}

\textbf{Duality} has been studied extensively in the past decades in the context of mirror symmetry, notably in relation to the SYZ conjecture for Calabi-Yau 3-folds \cite{Joyce}\cite{Chan}\cite{StromingerYauZaslow}, and Fourier-Mukai transforms in algebraic geometry \cite{BraamBaal}\cite{Huy2}\cite{Mukai}. In the $G_2$ setting, Gukov, Yau and Zaslow speculate in \cite{GukovYauZaslow} that $G_2$ manifolds with calibrated fibration structures may arise in dual pairs, from motivations in physics. The rough idea is that the mirror space should also have some calibrated fibration structure, where the base of the mirror is the same as the original base, and the fibres of the mirror are moduli spaces parametrising `branes' on the original fibres. A more mathematical perspective is discussed by  Leung and Lee \cite{Conan}, who attempt to find canonical special geometric structures on moduli spaces of submanifold theoretic and gauge theoretic objects on a given $G_2$ manifold.
The recent paper of Braun and Del Zotto \cite{Braun} suggests that mirror families could arise in dual pairs for Kovalev's twisted connected sum construction of $G_2$ manifolds, by some combinatorial study of examples produced using polytopes, and gives some topological evidence by Betti number computations.

In our setting of Donaldson's adiabatic fibrations, the fibres are hyperk\"ahler K3 surfaces, and the natural source to look for the mirror spaces is to perform Mukai duality on each K3 fibre. This requires the input of a principal $U(r)$ bundle $P\to M$, sometimes treated as a Hermitian vector bundle, with certain topological conditions. (A technical variant with $PU(r)$ structure group is also important in this paper, since Hermitian Yang-Mills connections on $P$ can be equivalently thought as ASD connections on the associated $PU(r)$ bundle, and this formulation does not favour any particular complex structure). Restricted to each K3 fibre $X=M_b$ for $b\in B$, we can take the moduli space of irreducible HYM connections on $P|_{X}\to X$, which in general gives a hyperk\"ahler manifold, and for some special topological numbers are known to be K3 surfaces (assuming compactness and non-emptiness), called the Mukai dual K3 surface $X^\vee$.  At the level of differential topology, these fit into a K3 fibration $\pi^\vee: M^\vee \to B$. We refer to $\pi^\vee$ as the \textbf{Mukai dual fibration}, and its analogue with higher dimensional hyperk\"ahler fibres are called \textbf{moduli bundles}.
One naturally asks

\begin{Question}
Can we assign canonical geometric structures to $\pi^\vee: M^\vee \to B$, to satisfy the conditions of Donaldson's adiabatic fibration in their own right? 
\end{Question}

This question can be answered at two levels. The first viewpoint is to dualise the positive section $h: B\to H^2(X)$ satisfying the maximal submanifold equation. The dual positive section $h^\vee: B\to H^2(X^\vee)$, is simply obtained by composing $h$ with Donaldson's $\mu$-map which arises in the context of 4-manifold polynomial invariants. It follows rather formally from facts in the theory of ASD instantons, that 

\begin{prop}(\cf Section \ref{Dualityofmaximalsubmanifoldequations})
The positive section $h^\vee$ also satisfies the maximal submanifold equation, so encodes another Donaldson's adiabatic fibration $M^\vee\to B$, and is compatible with the hyperk\"ahler periods on the Mukai dual K3 fibres.
\end{prop}

The second and more geometric answer turns out to be intricately tied to a circle of questions in higher dimensional \textbf{gauge theory} \cite{Thomas}. 
To provide some general context, it is a common phenomenon for adiabatic equations in gauge theory to be encoded by information on some moduli bundle. To take a well known example related to the Atiyah-Floer conjecture \cite{SalamonDostoglou}, ASD instantons on the product of two Riemann surfaces are encoded into a holomorphic mapping equation when one Riemann surface collapses to zero size. In a similar vein, Haydys \cite{Haydys} describes the formal adiabatic limits of Spin(7) instantons on spin bundles over a Riemannian 4-fold, in terms of the Fueter equation on some appropriate moduli bundle. Analogously, we show (\cf Section \ref{Fueterequation$G_2$instantons})

\begin{thm}(\cf Chapter \ref{Fueterequationsadiabatic$G_2$instantonsmodulibundles})
Under appropriate smoothness assumptions,
Adiabatic $G_2$ instantons on $M$ are  equivalent to solutions of the Fueter equation on the moduli bundle, up to gauge equivalence and twisting by central $u(1)$-valued 1-forms pulled back from $B$.
\end{thm}

The starting point of this theorem is the simple fact that adiabatic $G_2$ instantons restrict to ASD connections on the K3 fibres.
The main challenge here is to assign the canonical structures to the moduli bundle to make sense of the Fueter equation: they are the fibrewise hyperk\"ahler forms, and a horizontal distribution. These \emph{gauge theoretic} structures turn out to be exactly what is needed to put a Donaldson's adiabatic fibration structure $\underline{\omega}^\vee, \underline{\Theta}^\vee, \underline{\mu}^\vee$  on the \emph{Mukai dual fibration} $\pi^\vee: M^\vee\to B$.

We can then provide a deeper explanation of the above \textbf{instanton-Fueter correspondence}. The adiabatic $G_2$ instantons over $M$ are critical points of a gauge theoretic Chern-Simons functional, and the solutions of the Fueter equation on the moduli bundle are critical points of a submanifold theoretic Chern-Simons functional. Now a connection over $M$ which is ASD on every K3 fibre represents a section of the moduli bundle. Under this natural identification,

\begin{thm}(\cf Chapter \ref{Fueterequationsadiabatic$G_2$instantonsmodulibundles})
The two Chern Simons functionals are  equal.
\end{thm}

\begin{rmk}
This suggests that there may be an underlying equivalence of quantum theories, for which the instanton-Fueter correspondence is the classical manifestation.
\end{rmk}

Quite remarkably, the well definition of these Chern-Simons functionals, which are integral expressions, is precisely based on the differential relations for the data in Donaldson's adiabatic fibration. To make a familiar analogy, this resembles how charge conservation in electrodynamics is related to the continuity equation. This observation, with some additional arguments, enables us to show

\begin{thm}(\cf Chapter \ref{Mukaidualfibrations})
The Mukai dual fibration $\pi^\vee: M^\vee\to B$ admits canonical geometric structures which satisfy all the requirements of Donaldson's adiabatic fibration. The relation between hyperk\"ahler structures on K3 fibres  $X=M_b$  and $X^\vee=M^\vee_b$ agrees with the Mukai duality of hyperk\"ahler K3 surfaces.
\end{thm}

\begin{rmk}
More generally, when the fibrewise moduli spaces are higher dimensional hyperK\"ahler manifolds rather than K3 surfaces, the moduli bundle still admits canonical geometric structures which satisfy the natural generalisation of conditions in Donaldson's adiabatic fibration (\cf Chapter \ref{Fueterequationsadiabatic$G_2$instantonsmodulibundles}). 
\end{rmk}

From now on we assume further  for each $b\in B$ the existence of a universal family $\mathcal{E}|_b\to X\times X^\vee$
of HYM connections on $P|_X\to X$ parametrised by $X^\vee$. This leads naturally to the concept of  (twisted) triholomorphic connection, whose existence question is treated in the companion paper \cite{Mukaidualitypaper}. Here `triholomorphic' means that the curvature is of Dolbeault type $(1,1)$ with respect to all the 3 complex structures on $X\times X^\vee$, and twisting deals with the subtlety that the central $U(1)$ part of the curvature may obstruct the triholomorphic condition via a nontrivial first Chern class.   Viewed differently, this provides
a universal family of irreducible HYM connections  on the fibres $X^\vee$ of $M^\vee\to B$, parametrised by the corresponding fibre $X$ of $M\to B$. 
The Mukai dual fibration of $M^\vee \to B$  is well defined if for all $b\in B$ these families of HYM connections are irreducible. Then

\begin{thm}(\cf Chapter \ref{Mukaidualfibrations})
The Mukai dual fibration of $M^\vee\to B$ is isomorphic to $\pi: M\to B$ as Donaldson's adiabtic fibrations.
\end{thm}

Continuing with the gauge theory thread, a reformulation of the instanton-Fueter correspondence is 
\begin{thm}(\cf Chapter \ref{Mukaidualfibrations})
Adiabatic associative sections on $\pi^\vee: M^\vee\to B$ are equivalent to (twisted) adiabatic $G_2$ instantons on the Hermitian vector bundles $E\to M$ associated to $P\to M$.
\end{thm}

When we restrict to the K3 fibres, this is merely the tautological correspondence of points on $X^\vee$ with HYM connections on $E|_X\to X$. This can be understood algebraically as a special case of Fourier-Mukai transforms of coherent sheaves, where points are viewed as skyscraper sheaves. Thinking more generally, if we take the Fourier-Mukai transform of a topologically different stable holomorphic bundle $\mathcal{F}|_X\not \simeq E|_X\to X$, we can sometimes get another bundle over $X^\vee$ instead of a skyscraper sheaf.

The  Hitchin-Kobayashi correspondence translates stable bundles into HYM connections. In the differential geometric context the Fourier-Mukai transform is known as the \textbf{Nahm transform} on K3 surfaces, treated in the companion paper \cite{Mukaidualitypaper} taylored to the need of the present paper. This Nahm transform is built out of solutions to the coupled Dirac equation. The basic picture is that under the condition that the slope of $\mathcal{F}|_X$ is  equal to the slope of $E$, irreducible HYM connections $\alpha|_X$ on  $\mathcal{F}|_X\to X$ will be transformed to HYM connections $\hat{\alpha}|_{X^\vee}$ on another vector bundle $\hat{\mathcal{F}}\to X^\vee$ with a certain slope.  Morever, if $\hat{\alpha}|_X$ is also irreducible, then we can define the inverse Nahm transform of $(\hat{\mathcal{F}}|_X, \hat{\alpha}|_X)$, which is  canonically isomorphic to $(\mathcal{F}|_X, \alpha|_X)$.

There turns out to be an analogous picture for Donaldson's adiabatic fibrations. Starting from a (twisted) adiabatic $G_2$ instanton $\alpha$ on a Hermitian bundle $\mathcal{F}\to M$ which is fibrewise topologically distinct from $E\to M$. Restricted to every K3 fibre, we have a HYM connection. Assuming the slope condition and irreducibility condition on every fibre, we can build a Hermitian bundle $\hat{\mathcal{F}}\to M^\vee$ by fibrewise performing the Nahm transform on K3 surfaces. This is analogous to relative Fourier-Mukai transform in algebraic geometry, but the absence of a global complex structure makes the problem much harder.

The main subtlety is to put a canonical connection $\hat{\alpha}$ on $\hat{\mathcal{F}}\to M^\vee$, which fibrewise agrees with the Nahm transform on K3 surfaces. The pair $(\hat{\mathcal{F}}, \hat{\alpha})$ is called the \textbf{Nahm transform} of $(\mathcal{F}, \alpha)$. This construction requires two major inputs. First, we need 
a universal bundle $\mathcal{E}\to M\times_B M^\vee$ over the fibred product $M\times_B M^\vee$, carrying a very special connection $\nabla^{univ}$ which we call a \textbf{twisted generalised adiabatic $G_2$ instanton}, to be defined in Chapter \ref{Mukaidualfibrations}. Second, we need a canonical connection on the spinor bundle of $M$ adapted to our adiabatic problem, introduced in Chapter \ref{Adiabaticspinstructures}.
Then by performing some very delicate calculations mainly involving Dirac operators on coupled spinor fields, we show

\begin{thm}
(\cf Chapter \ref{Dualityforgaugetheories})
In this setup, 
the Nahm transform $(\hat{\mathcal{F}}, \hat{\alpha})$ is a twisted adiabatic $G_2$ instanton over $M^\vee$. 
\end{thm}

On a more technical level, there are two main ways the conditions of Donaldson's adiabatic fibrations enter into the story of the Nahm transform. The first is a characterisation of these conditions in terms of a \textbf{curvature operator} acting on negative spinors, which is developed in the later half of Chapter \ref{Adiabaticspinstructures}, and in turn relies on a precise understanding of the deformation theory of the fibrewise hyperk\"ahler structures on K3 surfaces. The second is that the twisted generalised adiabatic $G_2$ instanton equation is highly overdetermined, and its existence theory relies substantially on the \textbf{integrability conditions} inherent in Donaldson's adiabatic fibrations.

Furthermore, there is an analogue of the \textbf{Fourier inversion} theorem. Assuming that the Nahm transform $\hat{\alpha}$ is irreducible on every K3 fibre, and $M$ is identified as the Mukai dual fibration of $M^\vee$, then we can define the inverse Nahm transform $(\hat{\hat{\mathcal{F}}}, \hat{\hat{\alpha}})$ by globalising the fibrewise construction on K3 surfaces.

\begin{thm}
(\cf Chapter \ref{Dualityforgaugetheories}) The inverse Nahm transform is gauge equivalent to the original twisted adiabatic $G_2$ instanton $(\mathcal{F}, \alpha)$ up to possibly twisting by a $u(1)$-valued 1-form pulled back from $B$.
\end{thm}

This means the gauge theory over Donaldson's adiabatic fibration $M$ is in some sense dual to the gauge theory over the Mukai dual fibration $M^\vee$.  The technical input is the difficult Fourier inversion theorem in the K3 context \cite{Mukaidualitypaper}, and the instanton-Fueter correspondence in Section \ref{Fueterequation$G_2$instantons}.

We mention some directions of \textbf{open problems} following naturally from this work, which we leave for future investigation.

\begin{itemize}
\item There is  likely an analogous picture involving $Spin(7)$ manifolds with Cayley fibrations.

\item  In this paper we work over a local base to avoid singular fibres and reducible connections. What happens when we include these singularities? (The analogous problem for Calabi-Yau 3-folds has been studied in algebraic geometry, e.g.\cite{Thomas1}.) Calculations by the author suggest that even if $M\to B$ is smooth submersion, its Mukai dual fibration can acquire singularity if we allow for reducible connections.

\item  How can we produce global examples of Mukai dual fibrations, especially non-algebraic examples?

\item One of the primary motivation for studying adiabatic structures is to understand highly collapsed bona fide torsion free $G_2$ manifolds. When can we perturb Donaldson's adiabatic fibrations into genuine coassociative fibrations, and what happens to the gauge theory thereon?

\end{itemize}

\begin{Acknowledgement}
	The author thanks his PhD supervisor Simon Donaldosn and cosupervisor Mark Haskins for inspirations, and Simon Center for hospitality.
	
\end{Acknowledgement}

\section{Adiabatic limiting structures}\label{Adiabaticlimitingstructures}

We discuss in a unified way using a simple linear algebraic model, how to formally write down the adiabatic limiting conditions for a number of interesting geometric objects, such as K3 fibred torsion free $G_2$ structures, associative sections, $G_2$ instantons, and $G_2$ monopoles. The discussion is formal in the sense that we do not pursue the question of perturbing an adiabatic solution to an actual solution for small parameters.

The torsion free $G_2$ structure case is treated by Donaldson \cite{Donaldson}, which we outline in Section \ref{Donaldsonadiabaticfibration}. The associative section case is treated in Section \ref{adiabaticlimitofassociative}, where we show the formal adiabatic limiting condition can be expressed in terms of solutions of the Fueter equation, which is a nonlinear version of the Dirac equation. We treat adiabatic $G_2$ instantons and adiabatic $G_2$ monopoles in Section \ref{adiabaticlimitof$G_2$instantons} and \ref{adiabaticlimitof$G_2$monopoles}, where we show they are essentially equivalent, just like $G_2$ instantons are equivalent to $G_2$ monopoles on compact manifolds.

A common paradigm to these equations, is that they fit into some variational framework. This interacts nicely with formal limits, and we shall define a number of Chern-Simons type functionals whose critical points characterise the adiabatic solutions.

\subsection{The basic linear algebraic model}\label{linearalgebraicmodelsection}

The basic linear algebra of coassociative fibration in a $G_2$ manifold is given by the following simple model. One takes an orthogonal splitting $\R^7=\R^3\oplus \R^4$, and thinks of $\R^3$ as isomorphic to the space of self dual 2-forms on $\R^4$. The $G_2$ structure on $\R^7$ is given by the 3-form
\[
\phi=\underline{\omega}+\underline{\lambda}=\sum \omega_i dt_i-dt_1dt_2dt_3, 
\]
where $\omega_i$ for $i=1,2,3$ are the standard basis of self dual 2-forms on $\R^4$, and $t_1, t_2, t_3$ are the standard coordinates on $\R^3$. Alternatively, one can think of the 4-form
\[
*\phi=\underline{\Theta}+\underline{\mu}= -\sum_{cyc} \omega_i dt_j dt_k + \frac{1}{2}\omega_1^2.
\]
The induced metric is just the standard Euclidean metric
\[
g=\sum dx_i^2+ \sum dt_j^2,
\]
where $x_i$ are the standard coordinates on $\R^4$. The oriented volume form is $-dt_1dt_2dt_3dx_1dx_2dx_3dx_4$. From these structures, one can write down a cross product $\times: \R^7\times \R^7\to \R^7$,
\[
g( x\times y, z)=\phi(x,y,z),
\]
and an alternating trilinear map $\chi: \R^7\times \R^7\times \R^7 \to \R^7$,
\[
g(\chi(x,y,z),w)=*\phi(x,y,z,w).
\]

Now we scale the $\R^4$ factor by a small number $\epsilon$. The defining 3-form is changed to 
\[
\phi_\epsilon= \epsilon \underline{\omega}+\underline{\lambda}= \epsilon \sum \omega_i dt_i-dt_1dt_2dt_3.
\]
Accordingly, 
\[
*_\epsilon \phi_\epsilon= \epsilon \underline{\Theta}+\epsilon^2 \underline{\mu}= -\epsilon\sum_{cyc} \omega_i dt_j dt_k + \frac{\epsilon^2}{2}\omega_1^2, \quad g_\epsilon=\epsilon \sum dx_i^2+ \sum dt_j^2,
\]
\[
g_\epsilon( x\times_\epsilon y, z)=\phi_\epsilon(x,y,z), \quad g_\epsilon(\chi_\epsilon(x,y,z),w)=*\phi_\epsilon(x,y,z,w).
\]
It is easy to describe the effect of scaling on $\chi_\epsilon$.
\[
\chi_\epsilon (x,y,z)= \begin{cases}
\epsilon \chi(x,y,z)  & x, y \in \R^4 \\
\chi(x,y,z) & x\in \R^4, y, z\in \R^3 \\
0 & x,y,z\in \R^3
\end{cases}
\]
Other possibilities are clear from alternating property. Thus in the formal limit, only the combinations of $x,y,z$ with one factor in $\R^4$ and two factors in $\R^3$ can contribute.
For example, for $y=\frac{\partial}{\partial t_2}$, $z=\frac{\partial}{\partial t_3}$, $x\in \R^4$, the limit of $\chi_\epsilon$ gives $-I_1 x$, where $I_1$ is the complex structure corresponding to $\omega_1$. We remark in passing that complex structures $I_i$ act on 1-forms on $\R^4$ by the negative of  precomposition:
\[
I_i a =-a\circ I_i, \quad a\in (\R^4)^*.
\]
This convention is to ensure compatibility with quaternionic ring structure.

Similarly, contributions to the formal limit of the cross product require at least one factor from $\R^3$.

It is also interesting to see the effect of scaling on the Hodge star $*_7$ on $(\R^7)^*$. Given a 1-form $a\in (\R^4)^*$, then $*_7 a= *_4 a\wedge \underline{\lambda}$. If instead $a\in (\R^3)^*$, then $*_7 a=\underline{\mu}\wedge *_3 a$. Here $*_3$, $*_4$ are the Hodge stars on the duals of $\R^3$ and $\R^4$. Now when we scale the 3-form, 
\[
*_{7\epsilon} a= \begin{cases}
\epsilon *_4 a\wedge \underline{\lambda}
& a\in (\R^4)^* \\
\epsilon^2\underline{\mu}\wedge *_3 a
& a\in (\R^3)^*
\end{cases}
\]

\subsection{Donaldson's adiabatic fibration}\label{Donaldsonadiabaticfibration}

Following Donaldson \cite{Donaldson}, we consider a very collapsed $G_2$ manifold $M$ with a coassociative submersive K3 fibration $\pi: M\to B$ over a contractible base $B$, with fibres diffeomorphic to $X=K3$. The natural data to describe such a situation involve the 3-form $\phi_\epsilon= \epsilon \underline{\omega}+ \underline{\lambda}$, the dual 4-form $*_\epsilon\phi_\epsilon$ and the metric $g_\epsilon$ as in the linear algebraic model. The metric induces an orthogonal splitting of $TM$ into vertical and horizonal directions, \ie the data of a connection $H$. This induces a decomposition of forms \[
\Lambda^nT^*M=\bigoplus_{p+q=n} \Lambda^{p,q} T^*M,
\]
with $p$ terms from the base and $q$ terms from the fibre. The exterior differentiation operator $d: \Omega^{p,q}\to \Omega^{p, q+1}+ \Omega^{p+1, q} +\Omega^{p+2, q-1}$ splits into types
\[
d= d_f+ d_H + F_H,
\]
where $F_H$ is an algebraic operator. If we impose the torsion free condition on $\phi_\epsilon$, and take the formal limit as $\epsilon\to 0$, we get the adiabatic limiting structure
\begin{equation}\label{adiabaticclosed}
d_f \underline{\omega}=0, \quad  d_H \underline{\omega}=0, \quad d_f \underline{\lambda}=0,
\end{equation}
\begin{equation}\label{adiabaticcoclosed}
d_H\underline{\mu}=0, \quad d_f \underline{\Theta}=0, 
\end{equation}
and
\begin{equation}\label{maximalsubmanifold}
d_H \underline{\Theta}=0.
\end{equation}
The equations (\ref{adiabaticclosed}) come from the closed condition, and the equations (\ref{adiabaticcoclosed}), (\ref{maximalsubmanifold}) come from the coclosed condition. They imply that $\underline{\omega}$ fibrewise defines a hyperk\"ahler triple. At the linear algebraic level, the data $\underline{\omega}$, $\underline{\Theta}$, $\underline{\lambda}$, $\underline{\mu}$, $H$ specifies a Riemannian metric on the base $B$, which one heuristically imagines to be the Riemannian collapsing limit of a sequence of torsion free $G_2$ metrics. We shall refer to the data solving (\ref{adiabaticclosed}), (\ref{adiabaticcoclosed}), (\ref{maximalsubmanifold}) as \textbf{Donaldson's adiabatic fibration}. Some of these equations have simple interpretations. For example, $d_f\underline{\lambda}$ means $\underline{\lambda}$ is pulled back from the base, and $d_H\underline{\mu}=0$ means the horizontal distribution $H$ preserves the fibrewise volume form. The fibrewise $\underline{\mu}$ volume is thus a constant, which Donaldson fixes to be 1, and then the compatibility of the base metric $g^{base}$ with the other data is given by
\begin{equation}
g^{base}(\frac{\partial }{\partial t_i},\frac{\partial }{\partial t_j} )
=\frac{1}{2}\int_{M_b} \omega_i \wedge \omega_j.
\end{equation}

Donaldson shows how to compress these data into a maximal submanifold equation. We briefly outline this theory. The $\omega_i$'s give 3 class maps $B\to H^2(X)=H^2(K3)$, so define an $H^2(X)$-valued 1-form on $B$. The integrability conditions on $\omega_i$ implies this 1-form is closed, hence on a local base $B$ can be written as the derivative of a function $h:B\to H^2(X)$, which is well defined up to an additive constant.  
The derivatives $\frac{\partial h}{\partial t_i}$ give the class of $\omega_i$ in $H^2(X)$, so the tangent spaces of the image of $h$ are maximal positive subspaces in $H^2(X, \R)$, whose natural nondegenerate cup product form has signature (3,19). Donaldson calls such maps $h$ \textbf{positive sections}.

Using the Torelli theorem for hyperk\"ahler structures on K3 surfaces, this map $h$ recovers the fibrewise hyperk\"ahler metric up to diffeomorphism of fibres. The quantity $\underline{\lambda}$ is a volume form on $B$, chosen appropriately to ensure the corresponding fibrewise $\underline{\mu}$ volumes are 1. Donaldson shows there is a unique choice of $H$ to satisfy $d_H \underline{\omega}=0$ and $d_H \underline{\mu}=0$. The rest of the data can be determined from here using linear algebra, and they solve the equations (\ref{adiabaticclosed}), (\ref{adiabaticcoclosed}).

\begin{prop}
(\cf \cite{Donaldson}) The data set determined above by the positive section $h$ satisfies (\ref{maximalsubmanifold}) if and only if $h$ defines a maximal submanifold of $H^2(X)$.
\end{prop}

One can take a variational viewpoint on this. The area functional on the positive section $h$ induced from the immersion into $H^2(X)$,
\begin{equation}\label{areafunctional}
Area(h)=\int_B det ^{1/3}(\int_X[\omega_i]\cup [\omega_j]) dt_1dt_2dt_3, \quad [\omega_i]=\frac{\partial h}{\partial t_i},
\end{equation}
is the limiting manifestation in the adiabatic situation of the Hitchin volume functional on closed positive 3-forms. Thus the critical points correspond unsurprisingly to Donaldson's adiabatic fibrations, which are supposed to be collapsed limits of torsion free $G_2$ structures.

\begin{rmk}
At present there is no known existence result in the direction of obtaining an actual $G_2$ structure from perturbing the adiabatic fibration data. Our discussions therefore will have a formal character, and we shall be interested in the limiting equations in their own right.
\end{rmk}

\subsection{Adiabatic limit of associative sections}\label{adiabaticlimitofassociative}

An associative section is a section $s$ of $\pi: M\to B$ which is an associative submanifold. The associative condition can be characterised as $\chi_\epsilon(x,y,z)=0$ where $x,y,z$ is a basis for the tangent space of the submanifold $s(B)$. Let $t_i$ be a set of coordinates on $B$, such that $\frac{\partial}{\partial t_i}$ is orthonormal at a given point of interest.
Using the horizontal distribution $H$, one can decompose $ds(\frac{\partial}{\partial t_i}) \in TM$ into the horizontal lift of $\frac{\partial}{\partial t_i}$, still denoted $\frac{\partial}{\partial t_i}$, and the vertical part, denoted $\nabla_{ \frac{\partial}{\partial t_i} } s$:
\[
ds(\frac{\partial}{\partial t_i})=\frac{\partial}{\partial t_i}+ \nabla_{ \frac{\partial}{\partial t_i} } s.
\]
This notation is compatible with viewing $H$ as an Ehresmann connection. We warn the reader that writing $\frac{\partial}{\partial t_i}$ for horizontal vectors must not be taken to imply $H$ is an integrable distribution.
Using the basic linear algebraic model in Section \ref{linearalgebraicmodelsection}, we can take the formal limit as $\epsilon\to 0$ of the equation
\[
\chi_\epsilon( ds(\frac{\partial}{\partial t_1}), ds(\frac{\partial}{\partial t_2}), ds(\frac{\partial}{\partial t_3})         )=0,
\]
to obtain
\begin{equation}\label{adiabaticassociativesection}
\sum_{i}I_i \nabla_{ \frac{\partial}{\partial t_i}  } s =0.
\end{equation}
Here $I_1, I_2, I_3$ are the 3 complex structures on the K3 fibres, corresponding to the symplectic forms $\omega_i$ obtained by contracting $\frac{\partial}{\partial t_i}$ with $\underline{\omega}$.

\begin{rmk}
Here by the formal limit we mean that we na\"ively pretend as if the section $s$ converges smoothly, and there exists  some actual torsion free $G_2$ structure $\phi_\epsilon$ very closely approximated by our linear algebraic model. There is no claim that actual associative sections would have this sort of convergence behaviour. Similar remarks will apply to the analogous settings elsewhere in this paper.
\end{rmk}

This limiting equation has an invariant meaning: it is the Fueter equation for the section of the hyperk\"ahler K3 bundle $\pi: M\to B$. More discussions on the Fueter equation will be given later in Section \ref{TheFueterequationrecap}. Thus

\begin{prop}\label{associativesectionisFueter}
The formal adiabatic limiting condition for an associative section is the Fueter equation.
\end{prop}
We sometimes refer to the solutions of (\ref{adiabaticassociativesection}) as \textbf{adiabatic associative sections}.

It is interesting to understand this in terms of variational formulations. Associative submanifolds are calibrated by the 3-form, so in particular are critical points of the area functional for compactly supported variations. Small perturbations of $s(B)$ will remain being a section. This suggests evaluating the formal limit of the area functional as $\epsilon\to 0$ on the space of all sections $s: B\to M$.
One finds 
\[
\text{Area}(s(B))=\int_B \underline{\lambda},
\]
\ie the area functional is constant.

Alternatively, one can think of the \textbf{Chern-Simons type functional}. Fix an arbitrary section $s_0$ as a base point, so $s(B)$ and $s_0$ are viewed as the boundary of a 4-fold $N: [0,1]\times B\to M$. We integrate the 4-form over the 4-fold $N$ to define a functional, which is independent of $N$ under boundary fixing homotopies, and the critical points are associative sections. We can take a scaled formal limit of this functional, which is 
\[
CS^{associative}(s)=\int_{[0,1]\times B} N^*\underline{\Theta}.
\]
The critical points are just adiabatic  associative sections.

\subsection{Adiabatic limit of $G_2$ instantons}\label{adiabaticlimitof$G_2$instantons}

The $G_2$ instanton equation for a connection $A$ on a principal bundle $P$ over $M$ (usually with structure group $U(r)$ or $PU(r)$; the reader should bear in mind both cases), can be written as
\[
F_A \wedge *_\epsilon \phi_\epsilon=0.
\]
The formal limit of this equation after scaling, as $\epsilon\to 0$, is
\begin{equation}\label{limiting$G_2$instanton}
F_A\wedge \underline\Theta=0.
\end{equation}
Solutions of (\ref{limiting$G_2$instanton}) will be called \textbf{adiabatic $G_2$ instantons}.
We understand this limiting equation by decomposing the curvature $F_A$ into horizontal and vertical $(p,q)$ types using the Ehresmann connection $H$ (the reader should take care not to confuse this with Dolbeault type decomposition):
\[
F_A=F_A^{(2,0)}+ F_A^{(1,1)}+F_A^{(0,2)}.
\]
Then (\ref{limiting$G_2$instanton}) can be written as the ASD instanton condition on each fibre
\begin{equation}\label{ASDfibre}
F_A^{(0,2)} \wedge \omega_i=0,
\end{equation}
and some condition specifying variation with the fibres
\begin{equation}\label{adiabatic$G_2$instantonhorizontal}
F_A^{(1,1)} \wedge \underline\Theta=0.
\end{equation}

We compare this to the variational formulations. $G_2$ instantons are critical points of the Yang-Mills functional. If one formally take the $\epsilon\to 0$ limit of the YM functional
\[
YM(A)=\lim_{\epsilon\to 0} \frac{1}{8\pi^2}\int_M |F|^2 d\text{Vol},
\]
the scaling behaviour on the norm of 2-forms will pick out vertical components, so one obtains
\[
YM(A)= \frac{1}{8\pi^2} \int_B \underline{\lambda}(y) \int_{M_y}|F^{(0,2)}|^2 \underline{\mu}.
\]
The minima of this functional are just the connections $A$ which fibrewise restrict to an ASD connection. The minimum value is up to a factor of topological charge, just $\int_B \underline{\lambda}$. We don't see any information about the variation with fibres.

Alternatively, one can think of the \textbf{Chern-Simons type functional}. One fixes an arbitrary connection $A_0$ as a base point, a path joining $A_0$ and $A$, and consider the integral 
\[
\frac{-1}{4\pi^2}\int_{M\times [0,1]} \Tr (F_A(t)\wedge \frac{\partial {A(t)} }{\partial t}) \wedge *_\epsilon \phi_\epsilon \wedge  dt . 
\]
Taking the scaled formal limit, one obtains
\begin{equation}\label{ChernSimons}
CS^{instanton}(A)=\frac{-1}{4\pi^2}\int_{M\times [0,1]} \Tr (F_A(t)\wedge \frac{\partial {A(t)} }{\partial t}) \wedge \underline\Theta \wedge dt .
\end{equation}
The critical points are just the adiabatic $G_2$ instantons. The invariance of the Chern-Simons functional under boundary fixing homotopies uses the fact that $\underline\Theta$ is closed, much in the same way the closedness of $*\phi$ is the basic integrability condition underpinning the usual $G_2$ instanton equation.

When the structure group is $U(r)$, then the existence of adiabatic $G_2$ instantons, or more specifically the fibrewise ASD condition, necessarily implies the first Chern class $c_1(P)$ of the associated vector bundle is orthogonal to the fibrewise hyperk\"ahler triple. This is quite restrictive because it forces the underlying Donaldson's adiabatic fibration to be \textbf{non-generic}. As a remedy to broaden the applicability, we introduce

\begin{Def}\label{slopepotentialdef}
A $U(r)$ connection $A$ over a K3 fibre is called \textbf{Hermitian Yang-Mills} if it satisfies \begin{equation}\label{twistedadiabatic$G_2$instantonfibrewiseHYM}
\frac{\sqrt{-1}}{2\pi }F_A\wedge \omega_i= (\mathcal{B}\wedge \omega_i) I= (\frac{1}{r}\int_X c_1(P|_{X})\cup [\omega_i])\underline{\mu}\otimes I      ,
\end{equation}
where $\mathcal{B}$ is the harmonic 2-form representing $\frac{1}{r}c_1(P)|_{K3}$, and $I$ is the identity matrix. A $U(r)$ connection over $M$ is a \textbf{twisted adiabatic $G_2$ instanton} if it is  HYM  on every fibre, and morever it satisfies (\ref{adiabatic$G_2$instantonhorizontal}). The \textbf{slope potential function} of the $U(r)$ bundle is \[
B\to \R, \quad b\mapsto \frac{1}{r} \int_{M_b} c_1(P)\cup h,
\]
where $h$ is the positive section of the Donaldson's adiabatic fibration.

\end{Def}

\begin{rmk}
The class $c_1$ is orthogonal to some choice of $\omega_2$ and $\omega_3$, this is the usual notion of Hermitian Yang-Mills connection with respect to the preferred complex structure defined by $\omega_2+i\omega_3$. When $\mathcal{B}=0$ we recover ASD connections. The slope potential is equal to $\int_{M_b} \mathcal{B}\wedge h$, and its gradient encodes the slope of the HYM connection in the holomorphic case.
\end{rmk}

\begin{rmk}
It follows easily by taking the trace of (\ref{twistedadiabatic$G_2$instantonfibrewiseHYM}) that 
$
\frac{\sqrt{-1}}{2\pi r}\Tr F_A= \mathcal{B}
$ on K3 fibres. It is useful to think of twisted adiabatic $G_2$ instantons as given by a $PU(r)$  adiabatic $G_2$ instanton, and a prescription on the central part of the curvature.
\end{rmk}

\subsection{Adiabatic limit of $G_2$ monopoles}\label{adiabaticlimitof$G_2$monopoles}

The $G_2$ monopole equation involves a connection $A$ on a principal bundle $P$ over $M$, and a Higgs field $\Phi$, \ie a section of $ad(P)$. The equation is
\[
F_A \wedge *_\epsilon \phi + *_\epsilon d_A \Phi=0.
\]
This equation arises from dimensional reduction of the Spin(7) instanton equation. When the Higgs field is parallel, this reduces to the $G_2$ instanton equation; this happens automatically on a compact $G_2$ manifold. From the analytical viewpoint, this is an elliptic equation, whereas the $G_2$ instanton equation is overdetermined.

Now we take the formal limit after scaling the equation by $\epsilon^{-1}$. The resulting equation is 
\begin{equation}
F_A \wedge \underline\Theta+ *_4 (d_A \Phi)^{(0,1)} \wedge \underline{\lambda}=0.
\end{equation}
Here $(d_A \Phi)^{(0,1)}$ is the vertical component of $d_A \Phi$. Solutions will be called the \textbf{adiabatic $G_2$ monopoles}. 
 A perhaps surprising feature of the limiting equation is that it does not see the horizontal variation of $\Phi$. The Higgs field on different fibres decouple.

\begin{prop}\label{Adiabatic$G_2$monopoleadiabatic$G_2$instanton}
If $(A, \Phi)$ is an adiabatic $G_2$ monopole, then $d_A\Phi$ vanishes necessarily, so adiabatic $G_2$ monopoles are actually equivalent to adiabatic $G_2$ instantons.
\end{prop}

\begin{proof}
We notice that $\underline\Theta$ in Donaldson's adiabatic fibration is closed. Therefore the Bianchi identity $d_A F_A=0$ implies $d_A( F_A \wedge \underline\Theta)=0$, so \[d_A *_4 (d_A \Phi)^{(0,1)} \wedge \underline{\lambda}=0.\]
Here $\underline{\lambda}$ is just pulled back from the base $B$, and is closed. A moment of thought reveals $d_A^* d_A \Phi=0$ on each K3 fibre independently. Integration by part shows the $L^2$ norm of $d_A \Phi$ vanishes, hence the claim.
\end{proof}

\section{Adiabatic spin structures}\label{Adiabaticspinstructures}

We study the adiabatic limiting description of spinors on Donaldson's adiabatic fibrations. In Section \ref{AdiabaticLeviCivita}, \ref{AdiabaticspinstructureRiemanniangeometry}, we set out the general Riemannian framework of adiabatic spinors on $7=3+4$ dimensions, and special attention is paid to the Levi-Civita connection. This is essentially a standard construction in family index theory \cite[Chapter 1]{Bismut}. The general principle is that adiabatic spin structure on $M$ is encoded in the variation of the fibrewise spin structures, and the spinor bundle is better understood by a chiral decomposition. 

We then specialise to Donaldson's setting. In Section \ref{Adiabaticspinstructure$G_2$geometry} we show the adiabatic analogue of the well known characterisation of torsion free $G_2$ structures in terms of existence of parallel spinors. This gives a satisfactory understanding of the positive spinors.

Characterising the negative spinors is a deeper question. We prepare ourselves with some calculations about the variations of the fibrewise hyperk\"ahler metrics in Section \ref{Digressionhyperkahlermetric}. We then discuss in Section \ref{Curvatureoperators}  a certain curvature operator which captures information about Donaldson's adiabatic fibrations in a very essential way. This operator is intimately related to the variation of the fibrewise Dirac operator, which is a crucial ingredient in our study of the Nahm transform in Chapter \ref{TheNahmtransform}.

\subsection{Adiabatic Levi-Civita connection}\label{AdiabaticLeviCivita}

We wish to define the analogue of the Levi-Civita connection on Donaldson's adiabatic fibration $\pi: M\to B$. We first recall the usual formula of the Levi-Civita connection on a bona fide Riemannian manifold, written in terms of Lie brackets of tangent vectors:
\begin{equation}\label{LeviCivita}
\begin{split}
g(\nabla_X Y, Z)=\frac{1}{2}\{X g(Y,Z)+Y g(Z,X)-Z g(X,Y) \\
+g([X,Y],Z)-g([Y,Z], X  )-g([X,Z],Y) \}.
\end{split}
\end{equation}
In general, suppose we are given a horizontal distribution $H$ inducing a splitting of tangent bundle into vertical and horizontal subbundles $TM=TM_b\oplus H$, the fibrewise metrics $g^{fibre}$, and a metric $g^{base}$ on $B$, then there is a family of metrics $g_\epsilon=\epsilon g^{fibre}+\pi^*g^{base}$ such that $\pi:M\to B$ is a Riemannian submersion. The Levi-Civita connection induces subbundle connections $\nabla^\epsilon$ on $TM_b$ and $H$, for which we can take limits $\nabla^{LC}$ as $\epsilon\to 0$. We describe the results in cases:
\begin{itemize}
\item 
For vertical $X, Y$, the limit $\nabla^{LC}_X Y$ agrees with the Levi-Civita on the fibre.

\item 
Let $X$ be a horizontal lift of a vector field from the base, and $Y$ be vertical. Thus $[X,Y]$ is vertical, by the property of the Lie bracket. If $Z$ is vertical, one sees
\[
\begin{split}
\lim \frac{1}{\epsilon} g_\epsilon (\nabla^{g_\epsilon}_Y X, Z) &=\frac{1}{2}\{X g^{fibre} (Z,Y)+ g^{fibre}([Y,X], Z)- g^{fibre}([X,Z], Y) \}\\
& =\frac{1}{2} (\mathcal{L}_X g^{fibre})(Y, Z)
\end{split}
\]
where $\mathcal{L}$ is the Lie derivative. Using that the Levi-Civita of $g_\epsilon$ is a symmetric connection, we take the limit to get
\begin{equation}\label{adiabaticLeviCivitacase3}
g^{fibre} (\nabla^{LC}_X Y, \cdot{})=\frac{1}{2} (\mathcal{L}_X g^{fibre})(Y, \cdot{})+g^{fibre}([X,Y]. \cdot).\end{equation}

\item If $X$ and $Y$ are both horizontal lift of vector fields from the base, then $\nabla^{LC}_X Y$ agrees with the Levi-Civita connection on $B$.

\end{itemize}

We can think of $\nabla^{LC}$ as the \textbf{adiabatic Levi-Civita} on Donaldson's adiabatic fibration, a connection on the subbundles $TM_b$ and $H$ compatible with the metric on fibres and on the base. The reader is warned not to confuse this with the horizontal distribution $H$. The adiabatic Levi-Civita connection also acts on horizontal/vertical 1-forms as the natural connection on the dual bundle. Using the metrics, we can identify horizontal/vertical vectors $Y$ with horizontal/vertical 1-forms $Y^\sharp$, and $(\nabla^{LC}_X Y)^\sharp= \nabla^{LC}_X Y^\sharp$.

We take the opportunity to discuss the \textbf{curvature} $R(X,Y)=\nabla^{LC}_X \nabla^{LC}_Y-\nabla^{LC}_Y \nabla^{LC}_X-\nabla^{LC}_{[X,Y]}$. The $\nabla^{LC}$ acting on the horizontal vectors/1-forms agree with the Levi-Civita on the base $B$. So it is enough to understand the curvature tensor $R$ of $\nabla^{LC}$ acting on vertical vectors/ 1-forms. Suppose $X$, $Y$ are vertical, then  $R(X,Y)$  agrees with the Riemannian curvature operator for the fibre metric. The more interesting mixed term of $R$ is given by

\begin{lem}
Let $X, Y, Z$ be vertical vectors, and $\frac{\partial}{\partial t}$ be horizontal. Then
\begin{equation}\label{adiabaticLeviCivitacurvature}
\langle R(X,\frac{\partial}{\partial t} )Y, Z\rangle= \frac{1}{2}\{
(\nabla^{fibre}_Z \mathcal{L}_{ \frac{\partial }{\partial t}  }g^{fibre})(X,Y)
-
(\nabla^{fibre}_Y \mathcal{L}_{ \frac{\partial }{\partial t}  }g^{fibre})(X,Z)
\},
\end{equation}
where $\nabla^{fibre}$  means the Levi-Civita connection for the fibrewise metric, and here it is acting on the symmetric 2-tensor $\mathcal{L}_{ \frac{\partial }{\partial t}  }g^{fibre}$.
\end{lem}

\begin{proof}
The problem makes sense over a first order neighbourhood of $b\in B$, so we can take coordinates $x_1, x_2, x_3, x_4, t_i$, and write Lie derivatives $\mathcal{L}_{\frac{\partial}{\partial t}}=\frac{\partial}{\partial t}$, where $t$ can be any of the $t_i$. We compute in coordinates with Einstein's summation convention, and write $\Gamma$ for the Levi-Civita connection of the fibre metric. For $X=\frac{\partial}{\partial x_j}$ and $Y=\frac{\partial}{\partial x_k}$, using the compatibility of $\sharp$ with the  adiabatic Levi-Civita connection,
\[
(\nabla^{LC}_{   \frac{\partial}{\partial t}}Y)^\sharp= \frac{1}{2}\frac{\partial g_{ik}}{\partial t} dx^i, 
\quad
(\nabla^{LC}_X\nabla^{LC}_{   \frac{\partial}{\partial t}}Y)^\sharp
=\frac{1}{2}\{
\frac{\partial^2 g_{ik}}{\partial x_j\partial t} -\Gamma^s_{ji} \frac{\partial g_{sk}}{\partial t}
\}dx^i,
\]
\[
\nabla^{LC}_{   X} Y= \Gamma^s_{jk} \frac{\partial}{\partial x_s},
\quad
(\nabla^{LC}_{   \frac{\partial}{\partial t}}\nabla^{LC}_XY)^\sharp= \{  
\frac{1}{2}\frac{\partial g_{is}}{ \partial t} \Gamma^s_{jk}+
 g_{is} \frac{\partial}{\partial t}\Gamma^s_{jk}
\}dx^i.
\]
Hence by taking the difference,
\[
(R(X,\frac{\partial}{\partial t})Y)^\sharp=(\nabla^{LC}_X \nabla^{LC}_{ \frac{\partial}{\partial t} }Y-\nabla^{LC}_{ \frac{\partial}{\partial t}} \nabla^{LC}_X Y)^\sharp= \{\frac{1}{2} (\nabla^{fibre}_j \frac{\partial g}{\partial t})_{ik}
- g_{is} \frac{\partial}{\partial t}\Gamma^s_{jk} \}dx^i.
\]
We calculate further
\[
\begin{split}
g_{is} \frac{\partial}{\partial t}\Gamma^s_{jk}&= \frac{\partial}{\partial t}(g_{is} \Gamma^s_{jk})- \frac{\partial g_{is}}{\partial t} \Gamma^s_{jk} \\
& =\frac{1}{2}\frac{\partial}{\partial t}(
g_{ij,k}+g_{ik,j}-g_{jk,i}
)- \frac{\partial g_{is}}{\partial t} \Gamma^s_{jk} \\
& =\frac{1}{2}(
\frac{\partial^2 g_{ij}}{\partial x_k\partial t}+
\frac{\partial^2 g_{ik}}{\partial x_j\partial t}
-\frac{\partial^2 g_{jk}}{\partial x_i\partial t}
)- \frac{\partial g_{is}}{\partial t} \Gamma^s_{jk} \\
& =\frac{1}{2}\{
(\nabla^{fibre}_k\frac{\partial g}{\partial t})_{ij}+
(\nabla^{fibre}_j\frac{\partial g}{\partial t})_{ik}
-(\nabla^{fibre}_i\frac{\partial g}{\partial t})_{kj} \},
\end{split}
\]
so after cancellation,
\[
(R(X,\frac{\partial}{\partial t})Y)^\sharp= \frac{1}{2}\{
-
(\nabla^{fibre}_k\frac{\partial g}{\partial t})_{ij}
+(\nabla^{fibre}_i\frac{\partial g}{\partial t})_{kj} \} dx^i
.
\]
This is the claimed result in coordinate form.
\end{proof}

\subsection{Adiabatic spin structure: Riemannian geometry}\label{AdiabaticspinstructureRiemanniangeometry}

It is well known that genuine $G_2$ manifolds can be characterised in terms of the existence of parallel spinors. It is therefore interesting to examine how to describe spinors on Donaldson's adiabatic fibrations. In this Section we lay out the Riemannian geometric aspects of the adiabatic limit of the spin structure.

We first describe how to build up the spinor bundle on the 7-fold $M$ with Riemannian metric $\epsilon g^{fibre}+g^{base}$ before taking the adiabatic limit. This is quite general and is not tied to $G_2$ geometry. Let $S_B$ be the spinor bundle on the base $B$, so the pullback $S_B\to M$ is a complex rank 2 bundle, with a canonical complex volume form compatible with the $SU(2)$ structure. Let $S_{X}=S^+_X\oplus S^-_X$ be the spinor bundle on any fibre $X=M_b$, so these fit into a complex rank 4 bundle $S_X\to M$. Here $S^+_X$ and $S^-_X$ each has a complex volume form compatible with the fibrewise $SU(2)$-structure. The bundles $S_B$ and $S_X$ are equipped with Clifford multiplication actions $c_B$, $c_X$ by tangent vectors $\frac{\partial}{\partial t_i}$ on $B$ and  $\frac{\partial}{\partial x_i}$ on $X$, respectively. In our choice of conventions, $c_B(dt_1)c_B(dt_2)c_B(dt_3)=-1$.
We can make the pointwise construction to obtain a complex rank 8 bundle $S=S_X\otimes S_B\to M$. The tangent bundle and the cotangent bundle on $M$ act on $S$ by
\begin{equation}\label{Cliffordactiononfibration}
\begin{cases}
c(\frac{\partial}{\partial x_i})= c_X( \frac{\partial}{\partial x_i}  )\otimes 1, \quad S^-_X\otimes S_B\to S^+_X\otimes S_B, \\
c(dx_i)= c_X( dx_i  )\otimes 1, \quad S^+_X\otimes S_B\to S^-_X\otimes S_B, \\
c(\frac{\partial}{\partial t_i})=1\otimes   c_B( \frac{\partial}{\partial t_i}  ), \quad S^-_X\otimes S_B\to S^-_X\otimes S_B, \\
c(dt_i)= -1\otimes   c_B( dt_i  ), \quad S^+_X\otimes S_B\to S^+_X\otimes S_B.
\end{cases}
\end{equation}
Here we have hidden the dependence on $\epsilon$ by carefully arranging the appearance of tangent vectors and cotangent vectors. To write in full,
\begin{equation}\label{Cliffordactiononfibration1}
\begin{cases}
c_\epsilon(\frac{\partial}{\partial x_i})= \epsilon c_X( \frac{\partial}{\partial x_i}  )\otimes 1, \quad S^+_X\otimes S_B\to S^-_X\otimes S_B, \\
c_\epsilon(dx_i)= \frac{1}{\epsilon}c_X( dx_i  )\otimes 1, \quad S^-_X\otimes S_B\to S^+_X\otimes S_B, \\
c(dt_i)=1\otimes   c_B( dt_i  ), \quad S^-_X\otimes S_B\to S^-_X\otimes S_B, \\
c(\frac{\partial}{\partial t_i})= -1\otimes   c_B( \frac{\partial}{\partial t_i} ), \quad S^+_X\otimes S_B\to S^+_X\otimes S_B.
\end{cases}
\end{equation}
where $c_\epsilon$ is written to emphasize the $\epsilon$ dependence.
This action satisfies the Clifford relations: 
\begin{equation}
\begin{split}
\begin{cases}
c_\epsilon(\frac{\partial}{\partial x_i}) c_\epsilon(\frac{\partial}{\partial x_j})+ c_\epsilon(\frac{\partial}{\partial x_j}) c_\epsilon(\frac{\partial}{\partial x_i})=-2\epsilon g^{fibre}(\frac{\partial}{\partial x_i}, \frac{\partial}{\partial x_j}),\\
c_\epsilon(\frac{\partial}{\partial x_i}) c_\epsilon (\frac{\partial}{\partial t_j})+c_\epsilon(\frac{\partial}{\partial t_j}) c_\epsilon (\frac{\partial}{\partial x_i}) =0, \\
c_\epsilon(\frac{\partial}{\partial t_i}) c_\epsilon(\frac{\partial}{\partial t_j})+ c_\epsilon(\frac{\partial}{\partial t_j}) c_\epsilon(\frac{\partial}{\partial t_i})=-2 g^{base}(\frac{\partial}{\partial t_i}, \frac{\partial}{\partial t_j}).
\end{cases}
\end{split}
\end{equation}
This means $S$ is a Clifford module bundle on $M$. Morever, the bundle $S$ has a canonical complex volume form (\ie a section of the determinant bundle) independent of $\epsilon$, induced by the complex volume forms on $S^+_X$, $S^-_X$ and $S_B$. Thus $S$ can be identified as the spinor bundle on $M$. (In our setting $M$ is simply connected, so topologically the spinor bundle is unique). The subbundles $S^+_X\otimes S_B$ and $S^-_X\otimes S_B$ are each equipped with a canonical Hermitian metric. We leave the simple exercise for the reader to determine the scaling convention to build a Hermitian metric on $S$ compatible with Clifford multiplication $c_\epsilon$ and the complex volume form.

\begin{rmk}
The part of the Clifford action (\ref{Cliffordactiononfibration}) is defined  independent of $\epsilon$. This motivates us to think of $S$ as the \textbf{adiabatic spinor bundle} on Donaldson's adiabatic fibration. 
\end{rmk}

The next aim is to understand the Levi-Civita connection on $S$ (also known as the spin connection in the literature) when $\epsilon$ is finite, before taking the adiabatic limit. We write $\nabla^{\epsilon, ++}$ as the subbundle connection on $S^+_X\otimes S_B$. Similarly with $\nabla^{\epsilon, --}$ on $S^-_X\otimes S_B$. These are characterised by the Leibniz rules
\begin{equation}\label{LeviCivitaLeibnizrule1}
\begin{cases}
\nabla^{\epsilon, ++}_v (c(w)\cdot s)= c(w)\cdot \nabla^{\epsilon, --}_v s+  c(\nabla^{\epsilon}_v w)\cdot s, \quad \text{w is vertical}, \\
\nabla^{\epsilon, --}_v (c(w)\cdot s)= c(w)\cdot \nabla^{\epsilon, --}_v s+  c(\nabla^{\epsilon}_v w)\cdot s, \quad \text{w is horizontal},  
\end{cases}
\end{equation}
and the preservation of the canonical complex volume forms.

\begin{lem}
The components $\nabla^{\epsilon, ++}$ and $\nabla^{\epsilon, --}   $ of the Levi-Civita connection  on $S$ have limits $\nabla^{LC, ++ }$ and $\nabla^{LC, -- }$ as $\epsilon\to 0$, which satisfy the appropriate Leibniz rules similar to (\ref{LeviCivitaLeibnizrule1}) with $c(\nabla^{\epsilon}_v w)$ replaced by $c(\nabla^{LC}_v w)$.
\end{lem}

\begin{proof}
The Levi-Civita connection makes sense if the data are given on the first order neighbourhood of a point. Viewed this way, (\ref{LeviCivitaLeibnizrule1}) coupled with the complex volume preserving property, amount to some finite dimensional inhomogeneous system of linear equations on the matrix valued unknowns  $\nabla^{\epsilon, ++}_v$ and $\nabla^{\epsilon, --}_v$. The associated homogeoneous linear equation is the following system of equations on the matrix $B\in \End(S^+_X\otimes S_B)|_x \oplus \End(S^-_X\otimes S_B)|_x$ at a given point $x\in M$:
\[
\begin{cases}
B(c(w)\cdot s   )=c(w) \cdot Bs,\quad \forall w\in T_x M, s\in S|_x, \\
\quad \Tr B=0.
\end{cases}
\]
By writing tangent vectors and cotangent vectors in appropriate places we can make this homogeneous equation \emph{independent of $\epsilon$}.
This system of equations only has the trivial solution, because the first equation says $B$ commutes with all Clifford multiplications, which constrains $B$ to a scalar multiple of the identity, so must be zero by the trace condition.

As $\epsilon\to 0$, the inhomogeneous term $c(\nabla^{\epsilon}_v w) \cdot s$ in (\ref{LeviCivitaLeibnizrule1}) converges to $c(\nabla^{LC}_v w) \cdot s$, so elementary property of the linear equation says that the finite $\epsilon$ solutions must converge to the unique solution when $\epsilon=0$. 
\end{proof}

\begin{rmk}
By a limiting argument  $\nabla^{LC,++}$ and $\nabla^{LC,--}$ are compatible with the Hermitian metric and complex volume forms on $S^+_X\otimes S_B$ and $S^-_X\otimes S_B$.
\end{rmk}

\begin{prop}
The connections $\nabla^{LC,++}$ and $\nabla^{LC,--}$ are induced by the tensor product of the Levi-Civita connection on $S_B\to B$, with certain connections $\nabla^{LC,+}$ on the bundle $S^+_X\to M$ and $\nabla^{LC, -}$ on $S^-_X\to M$ satisfying the Leibniz rule
\begin{equation}\label{adiabaticspinstructurevariationLeibnizrule}
\begin{cases}
\nabla^{LC,+}_v (c_X(w)\cdot s)= c_X(w)\cdot \nabla^{LC,-}_v s + c_X(\nabla^{LC}_v w)\cdot s, \quad \text{w is vertical}\\
\nabla^{LC,-}_v (c_X(w)\cdot s)= c_X(w)\cdot \nabla^{LC,+}_v s + c_X(\nabla^{LC}_v w)\cdot s, \quad \text{w is vertical},
\end{cases}
\end{equation}
and morever $\nabla^{LC,+}$ and $\nabla^{LC,-}$ are compatible with the Hermitian structures and the complex volume forms on $S^+_X\to M$ and $S^-_X\to M$. If $v$ is vertical, then the derivatives $\nabla_v^{LC, +}$ and $\nabla_v^{LC, -}$ agree with the spin connections on fibres.
\end{prop}

\begin{proof}
We work in geodesic coordinates $t_1, t_2, t_3$ on $B$. It is enough to understand the derivatives $\nabla_v^{LC, ++}$ and $\nabla_v^{LC, --}$ when $v$ is horizontal. The Leibniz rules similar to (\ref{LeviCivitaLeibnizrule1}) imply
\[
\nabla^{LC,++}_v \circ c(\frac{\partial}{\partial t_j})=c(\frac{\partial}{\partial t_j})\circ \nabla^{LC,++}_v, \quad \nabla^{LC,--}_v \circ c(\frac{\partial}{\partial t_j})=c(\frac{\partial}{\partial t_j})\circ \nabla^{LC,--}_v,
\]
and therefore by thinking about the centralizer in the Clifford algebra, $\nabla^{LC,++}$ and $\nabla^{LC,--}$ must be tensor connections. The compatibility conditions are easy consequences.
\end{proof}

\begin{rmk}
Recall that $\nabla^{LC}$ encodes Lie derivatives of the fibrewise metric. Then the above Proposition expresses the principle that adiabatic spin structures on the total space are encoded by the variation of fibrewise spin structures.
\end{rmk}

\begin{rmk}
The Leibniz rule and the compatibility with the complex volume form uniquely characterise $\nabla^{LC,+}$ and $\nabla^{LC,-}$.
\end{rmk}

\subsection{Adiabatic spin structure: $G_2$ geometry}\label{Adiabaticspinstructure$G_2$geometry}

We now bring in the $G_2$ aspects of Donaldson's adiabatic fibration. The forms $\underline{\lambda}$, $\underline{\mu}$, $\underline{\omega}$ and $\underline{\Theta}$ act on elements of $S$. By our convention $\underline{\lambda}$ act as 1 on $S_B$, so by (\ref{Cliffordactiononfibration}), (\ref{Cliffordactiononfibration1}) it acts on $S$ by
\begin{equation}
c(\underline\lambda)=\begin{cases}
1 \quad &\text{on } S^-_X\otimes S_B \\
-1 \quad &\text{on } S^+_X\otimes S_B. 
\end{cases}
\end{equation}
To define the action of $\underline{\mu}$ on the adiabatic spinor bundle, we simply scale away the $\epsilon$ dependence. Thus
\begin{equation}
c(\underline\mu)=\begin{cases}
1 \quad &\text{on } S^-_X\otimes S_B \\
-1 \quad &\text{on } S^+_X\otimes S_B. 
\end{cases}
\end{equation}
Hence $\underline{\lambda}\wedge \underline{\mu}$ acts as 1 on $S$, as it should.
The more interesting actions come from $\underline{\omega}$ and $\underline{\Theta}$. We take an orthonormal basis $dt_i$ at a point on $B$, to write
\begin{equation}
c(\underline{\omega})=-\sum_i c_X(\omega_i)\otimes c_B(dt_i)=\begin{cases}
0 \quad & \text{on } S^-_X\otimes S_B, \\
-2\sum_i I_i^{S^+}\otimes c_B(dt_i) \quad & \text{on } S^+_X\otimes S_B.
\end{cases}
\end{equation}
Here the action on negative spinors is trivial, because $\omega_i$ are self-dual. The minus sign is inserted to be compatible with (\ref{Cliffordactiononfibration}). The operators $I_i^{S^+}=\frac{1}{2}c_X(\omega_i)$ are the natural operators on the spin bundle of a hyperk\"ahler 4-fold (\cf Appendix of \cite{Mukaidualitypaper}). Similarly,
\begin{equation}
c(\underline{\Theta})=-\sum_{cyc} c_X(\omega_i)\otimes c_B(dt_j)c_B(dt_k)=-\sum_i c_X(\omega_i)\otimes c_B(dt_i)=c(\underline{\omega}).
\end{equation}
Here we use $c_B(dt_i)c_B(dt_j)=c_B(dt_k)$ for cyclic $i, j, k$.

The next aim is to study the covariant derivatives of the operator $c(\underline{\Theta})$ acting on $S^+_X\otimes S_B$. Evantually we will show $c(\underline{\Theta})$ is parallel on $M$.

\begin{lem}
The operator $c(\underline{\Theta})$ is parallel along fibres of $M$.
\end{lem}

\begin{proof}
The fibres are just hyperk\"ahler K3 surfaces, so all the $c(\omega_i)$ are parallel with respect to $\nabla^{LC}$.
\end{proof}

Thus it suffices to understand the horizontal variation. Let $t_i$ be the geodesic coordinates on $B$ at the chosen point, so $\omega_i$ can be treated as orthonormal over the first order neighbourhood of $b\in B$. By the Leibniz rule, for a horizontal vector $v$,
\begin{equation}\label{cThetacovariantderivative}
\nabla_v c(\underline{\Theta})=\sum_i c_X(\nabla_v^{LC} \omega_i) \otimes c_B(dt_i) .
\end{equation}
We begin with a formula for $\nabla^{LC}$ on vertical 1-forms, which can be readily deduced from the discussions in Section \ref{AdiabaticLeviCivita}. 
\begin{lem}\label{adiabaticLeviCivitaonvertical1forms}
If $\alpha$ is a vertical 1-form, and $v=\frac{\partial}{\partial t_k}$ is a horizontal vector, then when regarded as 1-forms on the fibres,
\[
\nabla^{LC}_v \alpha= \mathcal{L}_v \alpha-\frac{1}{2}  (\mathcal{L}_v g^{fibre})(\alpha^{\sharp}, \_),
\]
where we recall $\alpha^{\sharp}$ means the vertical vector dual to $\alpha$ with respect to $g^{fibre}$.
\end{lem}

This entails by a short calculation
\begin{lem}
The covariant derivative 
\begin{equation}\label{LeviCivitaactingonhyperkahlerforms}
\nabla_v^{LC} \omega_i= \mathcal{L}_v \omega_i +\frac{1}{2}\sum_j I_i e_j^\sharp \wedge \iota_{e_j}( \mathcal{L}_v g^{fibre}),
\end{equation}
where $e_j$ is a local orthonormal basis of tangent vectors on $X=M_b$, and $\iota_{e_j}\mathcal{L}_v g^{fibre} $ means contracting $e_j$ with the first entry of the symmetric 2-tensor $\mathcal{L}_v g^{fibre}$.
\end{lem}

We are thus lead to calculate $\mathcal{L}_v g^{fibre}$ in terms of variations of $\omega_i$. As a remark, equalities in this Section means equalities when restricted to fibres. We start from a relatively standard linear algebraic fact.

\begin{lem}
A hyperk\"ahler metric $g^{fibre}$ can be expressed in terms of the hyperk\"ahler forms $\omega_i$ by
\begin{equation}
g^{fibre}(Y,Z)\underline{\mu}=\iota_Y \omega_1\wedge \iota_Z \omega_2 \wedge \omega_3= \iota_Y \omega_2\wedge \iota_Z \omega_3 \wedge \omega_1=\iota_Y \omega_3\wedge \iota_Z \omega_1 \wedge \omega_2.
\end{equation}
\end{lem}

The variation of this formula is
\begin{cor}
The Lie derivative of the fibrewise hyperk\"ahler metric is given in terms of $\mathcal{L}_v\omega_i$ by
\begin{equation}\label{hyperkahlermetricvariation}
\begin{split}
& (\mathcal{L}_v g^{fibre})(X, Y)\underline{\mu}+ g(X,Y)\mathcal{L}_v \underline{\mu}
\\
=&\iota_X \mathcal{L}_v\omega_1\wedge \iota_Y \omega_2 \wedge \omega_3+\iota_X \omega_1\wedge \iota_Y\mathcal{L}_v \omega_2 \wedge \omega_3+\iota_X \omega_1\wedge \iota_Y \omega_2 \wedge \mathcal{L}_v\omega_3,
\end{split}
\end{equation}
where $X,Y$ are vertical vectors.
\end{cor}

To proceed further, we differentiate the hyperk\"ahler relations $\omega_1^2=\omega_2^2=\omega_3^2=2\underline{\mu}$ and $\omega_i\wedge \omega_j=0$ for $i\neq j$, to write
\begin{equation}\label{hyperkahlerformvariationakbk}
\begin{cases}
\mathcal{L}_{ \frac{\partial}{\partial t_k}} \omega_1= b^k \omega_1+ a^k_{12}\omega_2+ a^k_{13} \omega_3 \quad \mod \text{ ASD terms, } \\
\mathcal{L}_{ \frac{\partial}{\partial t_k}} \omega_2= a^k_{21}\omega_1+ b^k \omega_2+ a^k_{23} \omega_3 \quad \mod \text{ ASD terms, }  \\
\mathcal{L}_{ \frac{\partial}{\partial t_k}} \omega_3= a^k_{31}\omega_1+ a^k_{32} \omega_2+ b^k \omega_3   \quad \mod \text{ ASD terms. }
\end{cases}
\end{equation}
Here the coefficients satisfy $a^k_{ij}=-a^k_{ji}$, and $\mathcal{L}_{ \frac{\partial}{\partial t_k}} \underline{\mu}=2b^k \underline{\mu}$. We claim

\begin{lem}
The Clifford action 
\begin{equation}\label{LeviCivitaactingonhyperkahlerformscliffordmultiplicationoperator}
c_X(\nabla^{LC}_{ \frac{\partial}{\partial t_k}      } \omega_i)=c_X( \mathcal{L}_{ \frac{\partial}{\partial t_k}      }\omega_i   )-b^k c_X(\omega_i)=c_X( \sum_{j\neq i} a^k_{ij}\omega_j   ).
\end{equation}
\end{lem}

\begin{proof}
One can show that when we substitute (\ref{hyperkahlerformvariationakbk}) into
 formula (\ref{hyperkahlermetricvariation}), the terms involving $a^{k}_{ij}$ vanish by an explicit calculation. The geometric interpretation is that the coefficients $a^k_{ij}$ come from hyperk\"ahler rotation, which does not change the metric. 

The formula (\ref{hyperkahlermetricvariation}) thus yields
\begin{equation}
\mathcal{L}_{\frac{\partial}{\partial t_k}    } g^{fibre}= b^k g^{fibre} \quad \text{ mod ASD terms}.
\end{equation}
The geometric interpretation of $b^k$ is just the scaling of volume forms. We substitute this into (\ref{LeviCivitaactingonhyperkahlerforms}) to get
\[
\nabla^{LC}_{ \frac{\partial}{\partial t_k}      }\omega_i=\mathcal{L}_{ \frac{\partial}{\partial t_k}      }\omega_i -b^k \omega_i \quad \text{ mod ASD terms}.
\]
But when we use this to calculate $c_X(\nabla^{LC}_v \omega_i)$, the ASD terms we are ignoring cannot contribute,
because by the representation theory of $Spin(4)=SU(2)\times SU(2)$, the ASD terms cannot act on positive spinors. The result follows. The fact that $b^k$ drops out of the final formula is because spinors have certain conformal invariance properties.
\end{proof}

\begin{prop}
Assuming the conditions of Donaldson's adiabatic fibration, then the operator $c(\underline{\Theta})$ acting on $S^+_X\otimes S_B\to M$ is parallel.
\end{prop}

\begin{proof}
By formula (\ref{cThetacovariantderivative}) and (\ref{LeviCivitaactingonhyperkahlerformscliffordmultiplicationoperator}), it is enough to check $a^k_{ij}=0$.

Observe now that the condition $d_H\underline{\omega}=0$ in Donaldson's adiabatic fibration means precisely that
\[
\mathcal{L}_{ \frac{\partial}{\partial t_j}      }\omega_i= \mathcal{L}_{ \frac{\partial}{\partial t_i}      }\omega_j, 
\]
which implies $a^{k}_{ij}=a^i_{kj}$, where we write $b^k=a^k_{ii}$. 
The condition $d_H \underline{\mu}=0$ precisely means $b^k=0$.

 Using also the antisymmetry $a^k_{ij}=-a^k_{ji}$, this shows that $a^k_{ij}=0$ as required. To give some sample calculations,
\[
a^1_{12}=-a^1_{21}=-a^2_{11}=-b^2=0,
\]
and
\[
a^1_{23}=-a^1_{32}=-a^3_{12}=a^3_{21}=a^2_{31}=-a^2_{13}=-a^1_{23}=0.
\]
The other calculations are entirely analogous.
\end{proof}

\begin{rmk}
The condition $d_H\underline{\Theta}=0$ means precisely 
\[
\sum_i \mathcal{L}_{ \frac{\partial}{\partial t_i}      }\omega_i=0,
\]
although we did not need this in the above argument. So the parallel nature of $c(\underline\Theta)$ is weaker than the full strength of Donaldson's conditions.
\end{rmk}

\begin{rmk}
We saw in the proof above that $a^k_{ij}=0$, $b^k=0$. This means $\mathcal{L}_{\frac{\partial}{\partial t_k}} \omega_i$ are all ASD 2-forms. This is crucial in the next Section.

\end{rmk}

Now we use the operator $c(\underline{\omega})=c(\underline{\Theta})$ to decompose $S^+_X\otimes S_B$ into eigenspace subbundles. A short linear algebraic computation yields

\begin{lem}
The vector space $(S^+_X\otimes S_B)|_x$ at any $x\in X$ splits into a complex 1-dimensional eigenspace of $c(\underline{\omega})$ with eigenvalue $-6$, and a complex 3-dimensional eigenspace with eigenvalue $2$. Any nonzero vector in the 1-dimensional eigenspace defines an isomorphism between $S^+_X$ and $S_B$ at the point.
\end{lem}

\begin{cor}
There is a canonical complex line bundle which is a parallel subbundle of $S^+_X\otimes S_B\to M$.
\end{cor}

\begin{thm}
There is a \textbf{canonical bundle isomorphism} $\Phi: S^+_X\to S_B$ over $M$, which intertwines the Levi-Civita connection on $S_B$ and the connection $\nabla^{LC,+}$ on $S^+_X$. Morever $\Phi$ preserves the Hermitian metrics and the complex volume forms on $S^+_X$ and $S_B$.
\end{thm}

\begin{proof}
We observe that $S^+_X\otimes S_B\to M$ has a Hermitian metric compatible with $\nabla^{LC,++}$. Morever, the complex volume forms on $S^+_X$ and $S_B$ are complex symplectic, so induce antilinear structures on $S^+_X$ and $S_B$, hence $S^+_X\otimes S_B$ has a canonical real structure, which is compatible with the connection $\nabla^{LC,++}$. This real structure commutes with $c(\underline{\omega})$, so is well defined on the canonical eigenspace line bundle. This produces a canonical real line bundle, whose unit norm sections must be parallel with respect to $\nabla^{LC,++}$. Since we are working on a K3 fibration over a topologically trivial base, $M$ is simply connected, so the real line bundle is orientable, and the two unit norm sections are globally defined. Choose any of them.

This section of $S^+_X\otimes S_B$ can be viewed as a section of $\Hom(S^+_X, S_B)\to M$. Since $\nabla^{LC,++}$ is a tensor connection by the discussions in Section \ref{AdiabaticspinstructureRiemanniangeometry}, the fact that the section is parallel is precisely saying it intertwines $\nabla^{LC,+}$ and the Levi-Civita connection on $S_B$. If we scale the section by a global real constant to have unit operator norm, then it preserves the Hermitian metric. Denote this scaled section as $\Phi$. Morever, since $\Phi$ is a real section, tautologically it intertwines the antilinear structures on $S_B$ and $S_X$, so must preserve the complex volume form.
\end{proof}

\begin{rmk}
Viewed as a section of $S^+_X\otimes S_B$, we think of $\Phi$ as the adiabatic analogue of the parallel spinor which characterises the torsion free $G_2$ condition.
\end{rmk}

\begin{rmk}
Via the isomorphism $\Phi$ the operator $c_B(dt_i)$ on $S_B$ is identified with $I_i^{S^+}=\frac{1}{2}c_X(\omega_i)$ on $S^+_X$, at the given point.
\end{rmk}

\subsection{Digression on variation of hyperk\"ahler metric}\label{Digressionhyperkahlermetric}

We make a slightly technical detour to better understand the linear algebraic relation between the variation of the hyperk\"ahler metric $\mathcal{L}_v g^{fibre}$ and the variation of the hyperk\"ahler forms $\mathcal{L}_v \omega_i$, where $v=\frac{\partial}{\partial t_k}$ is a horizontal vector. We begin by observing that the variations $\mathcal{L}_v\omega_i$ for $i=1,2,3$ are independent deformations, in the sense that we can recover each one individually from $\mathcal{L}_vg^{fibre}$.

\begin{lem}
We have the inverse formula of (\ref{hyperkahlermetricvariation})
\begin{equation}\label{hyperkahlerformvariationintermsofmetricvariation}
\frac{1}{2}\sum_j I_i e_j^\sharp \wedge \iota_{e_j}\mathcal{L}_v g^{fibre}= -\mathcal{L}_v \omega_i.
\end{equation}
where $v=\frac{\partial}{\partial t_k}$ is a horizontal vector field and $e_j$ is an orthonormal basis on the tangent space of the fibre.
\end{lem}

\begin{proof}
This is equivalent to $\nabla^{LC}_v \omega_i=0$ by (\ref{LeviCivitaactingonhyperkahlerforms}), which is equivalent to $c(\underline\Theta)$ being parallel on both positive and negative spin. We have showed the positive spin case in Section \ref{Adiabaticspinstructure$G_2$geometry}, and for negative spin this is trivially true
because $c(\underline\Theta)=0$ there.		
\end{proof}

\begin{eg}\label{exampleforspincalculationonhyperkahlerformvariation}
	We give a sample calculation.
	Working in a standard orthonormal basis on $T_x M_b\simeq \R^4$, if $\mathcal{L}_v \omega_1=0$, $\mathcal{L}_v \omega_2=0$, $\mathcal{L}_v\omega_3= dx_1dx_2-dx_3 dx_4$, then by (\ref{hyperkahlermetricvariation}),
	\[
	\mathcal{L}_v g^{fibre}= -dx_1\otimes dx_3- dx_3\otimes dx_1 + dx_2\otimes dx_4+ dx_4\otimes dx_2,
	\]
	so $\frac{1}{2}\sum_j I_i e_j^\sharp \wedge \iota_{e_j} \mathcal{L}_v g^{fibre}$ is zero for $i=1,2$ and is $-\mathcal{L}_v \omega_3$ when $i=3$.
\end{eg}

We can interpret the linear algebraic relation between $\mathcal{L}_v \omega_i$
and $\mathcal{L}_v g^{fibre}$ from the perspective of the $Sp(1)=SU(2)$ holonomy on the K3 surfaces and representation theory. The variation of the metric is a symmetric 2-tensor, which is traceless because of $d_H\underline{\mu}=0$. When we have $Sp(1)$ holonomy, the traceless symmetric 2-tensors decompose into holonomy representations, 
\[
S^2_0(\R^4)= \text{Im }\mathbb{H}\otimes \Lambda^2_- =\R^3\otimes \Lambda^2_-.
\]
This means the deformation of the metric naturally has 3 components corresponding to 3 ASD 2-forms, which in our concrete description is given by (\ref{hyperkahlerformvariationintermsofmetricvariation}) and (\ref{hyperkahlermetricvariation}). One can either see abstractly or using these formulae that the identification of these components with $\mathcal{L}_v \omega_i$ respects the Levi-Civita connection on the fibre.

We give some alternative formulations.

\begin{lem}
The variation of the hyperk\"ahler metric is given in terms of $\mathcal{L}_{\frac{\partial}{\partial t_k}  }\omega_i$ by
\begin{equation}\label{hyperkahlervariationcyclicsymmetry}
\begin{cases}
(\mathcal{L}_{\frac{\partial}{\partial t_k}  }g^{fibre}) ( Y, Z)=
-\mathcal{L}_{\frac{\partial}{\partial t_k}  }\omega_1 (I_1 Y, Z) -\mathcal{L}_{\frac{\partial}{\partial t_k}  }\omega_2 (I_2 Y, Z)
-\mathcal{L}_{\frac{\partial}{\partial t_k}  }\omega_3 ( I_3  Y, Z) 
\\
(\mathcal{L}_{\frac{\partial}{\partial t_k}  }g^{fibre}) (I_1 Y, Z)=
\mathcal{L}_{\frac{\partial}{\partial t_k}  }\omega_1 ( Y, Z) +\mathcal{L}_{\frac{\partial}{\partial t_k}  }\omega_2 (I_3 Y, Z)
-\mathcal{L}_{\frac{\partial}{\partial t_k}  }\omega_3 ( I_2  Y, Z) 
\\
	(\mathcal{L}_{\frac{\partial}{\partial t_k}  }g^{fibre}) (I_2 Y, Z)=
	-\mathcal{L}_{\frac{\partial}{\partial t_k}  }\omega_1 (I_3 Y, Z) +\mathcal{L}_{\frac{\partial}{\partial t_k}  }\omega_2 (Y, Z)
	+\mathcal{L}_{\frac{\partial}{\partial t_k}  }\omega_3 ( I_1  Y, Z)
	\\
	(\mathcal{L}_{\frac{\partial}{\partial t_k}  }g^{fibre}) (I_3 Y, Z)=
	\mathcal{L}_{\frac{\partial}{\partial t_k}  }\omega_1 (I_2 Y, Z) -\mathcal{L}_{\frac{\partial}{\partial t_k}  }\omega_2 (I_1 Y, Z)
	+\mathcal{L}_{\frac{\partial}{\partial t_k}  }\omega_3 (  Y, Z).
	\end{cases}
	\end{equation}

\end{lem}

\begin{proof}
The first equation is equivalent to (\ref{hyperkahlerformvariationintermsofmetricvariation}) and is perhaps most easily verified by thinking about independent contributions of $\mathcal{L}_{\frac{\partial}{\partial t_k}  }\omega_i$ to $\mathcal{L}_{\frac{\partial}{\partial t_k}  }g^{fibre}$. The other equations are obtained by precomposing the operators $I_1, I_2, I_3$.
\end{proof}

We recall any 2-form $F$ on $X$ acts on spinors by 
\[
c(F)\cdot s=\sum_{i<j} F(e_i, e_j) c_ic_j \cdot s,
\]
where $e_i$ is an orthonormal basis of the tangent space. We shall apply this to forms such as $\mathcal{L}_v\omega_i$.

\begin{lem}
	The Clifford action of $\mathcal{L}_v g^{fibre}$ on the negative spinors is given in terms of $\mathcal{L}_v \omega_i$ by
	\begin{equation}\label{hyperkahlermetricvariationformulaCliffordaction}
	c( \mathcal{L}_v \omega_k  )=\frac{1}{2}\sum_{i,j} (\mathcal{L}_v g^{fibre})(I_ke_i, e_j) c_ic_j 
	=-\frac{1}{2}\sum_{i,j} (\mathcal{L}_v g^{fibre})(e_i, I_k e_j) c_ic_j 
	.
	\end{equation}
The action on positive spinors are  zero by the ASD property of $\mathcal{L}_v \omega_k$.
\end{lem}

\begin{eg}
	In example \ref{exampleforspincalculationonhyperkahlerformvariation}, we have $c( \mathcal{L}_v \omega_3  )=c_1c_2-c_3c_4=2c_1c_2$, and
	$\sum_{i,j} (\mathcal{L}_v g^{fibre})(I_ke_i, e_j) c_ic_j$ is zero for $k=1, 2$, and is $4c_1c_2$ for $k=3$. Similarly $
	\sum_{i,j} (\mathcal{L}_v g^{fibre})(e_i,I_k e_j) c_ic_j
	$ is zero for $k=1,2$, and is $-4c_1c_2$ for $k=3$.
	In fact this calculation implies the above lemma by a Schur's lemma argument in the representation theory of $Sp(1)=SU(2)$.
\end{eg}

Finally, holonomy principle implies

\begin{cor}\label{hyperkahlermetricvariationderivative}
Let $w$ be a tangent vector to the fibre.
The fibre covariant derivatives $\nabla^{fibre}_w\mathcal{L}_v\omega_i$ are ASD 2-forms, and they are related to $\nabla^{fibre}_w\mathcal{L}_vg^{fibre}$ in exactly the same way $\mathcal{L}_v\omega_i$ is related to $\mathcal{L}_vg^{fibre}$, \ie we can replace $\mathcal{L}_v\omega_i$ by $\nabla^{fibre}_w\mathcal{L}_v\omega_i$, and $\mathcal{L}_vg^{fibre}$ by $\nabla^{fibre}_w\mathcal{L}_v\omega_i$ in the equations (\ref{hyperkahlerformvariationintermsofmetricvariation}), (\ref{hyperkahlervariationcyclicsymmetry}), (\ref{hyperkahlermetricvariationformulaCliffordaction}), and then the new equations still hold.
\end{cor}

\subsection{Curvature operators}\label{Curvatureoperators}

In this Section we examine the curvature tensor $R$ introduced in Section \ref{AdiabaticLeviCivita}, in the context of Donaldson's adiabatic fibration, and we show how a certain operator constructed out of $R$ encodes much of the information contained in Donaldson's conditions, which is not captured by the existence of the parallel positive spinor in Section \ref{Adiabaticspinstructure$G_2$geometry}. We then interpret this operator as arising from the variation of the fibrewise Dirac operators.

Since $S^+_X$ is canonically isomorphic to $S_B\to M$ preserving all structures, only the negative spin remains to be understood. Given tangent vectors $v,w$ on $M$, the curvature operator $R(v,w)\in \End(T_x M_b)\oplus \End(T_b B)  $ canonically acts on $S^-_X\otimes S_B$.
Because $\nabla^{LC,--}$ is a tensor product connection, we see easily that the $\End(T_x M_b)$ matrix component is responsible for the action on $S^-_X$, and the $\End(T_b B)$ matrix component is responsible for the action on $S_B$. The latter is essentially just the Riemannian curvature of $g^{base}$. We are thus only interested in the $\End{T_x M_b}$ matrix component. We can write out the curvature operator concretely as
\[
R(v,w)\cdot s=-\frac{1}{4}\sum_{i,j} \langle R(v,w  )e_j, e_i \rangle c_ic_j\cdot s,
\]
where $e_i$ is an orthonormal basis of $T_xM_b\simeq\R^4$. The coefficients here come from the theory of spin representation of $so(4)$. This $R(v,w)$ Clifford action is the same as the curvature operator of $\nabla^{LC,-}$ acting on spinors, by the compatibility of $\nabla^{LC,-}$ with the adiabatic Levi-Civita connection.

For $v, w$ both vertical, the curvature operator $R(v, w)$ agrees with the Riemannian curvature of the fibre metric. The more interesting case is when $v=\frac{\partial}{\partial t_k}$ is some horizontal vector, and $w$ is vertical. 

We define three operators  $\tilde{R}_k: S^-_X\to S^+_X$ at any given point,
\[
\tilde{R}_k= \sum_l c(e_l) R(e_l,\frac{\partial }{\partial t_k}  ) = -\frac{1}{4}\sum_{i,j,l} \langle R(e_l, \frac{\partial }{\partial t_k}   )e_j, e_i\rangle c_l c_i c_j,
\]
which can be compressed into a $\Hom(S^-_X, S^+_X)$-valued 1-form $\tilde{R}=\sum \tilde{R}_k dt_k$. We recall also the operators $I^{S^+}_k$ acting on $S^+_X$. The following delicate result is a manifestation of Donaldson's conditions in  adiabatic spin geometry.

\begin{prop}
The operators $\tilde{R}_k$ satisfies 
$
\sum_{k=1}^3 I^{S^+}_k\circ \tilde{R}_k=0,
$	
or in other words,
\begin{equation}
( I_1^{S^+} dt_2dt_3+ I_2^{S^+} dt_3dt_1+ I_3^{S^+} dt_1dt_2      )\tilde{R}=0.
\end{equation}
\end{prop}

\begin{proof}
Using the curvature formula (\ref{adiabaticLeviCivitacurvature}), we have
\[
2\langle R(e_l, \frac{\partial }{\partial t_k}   )e_j, e_i\rangle= (\nabla^{fibre}_i \mathcal{L}_{ \frac{\partial }{\partial t_k}} g^{fibre  })( e_l, e_j  )- (\nabla^{fibre}_j \mathcal{L}_{ \frac{\partial }{\partial t_k}  }g^{fibre})( e_l, e_i  ).
\]
Thus writing in summation convention for the indices $i, j, l$,
\[
\begin{split}
2\tilde{R}_k& =-\frac{1}{4}\{(\nabla^{fibre}_i \mathcal{L}_{ \frac{\partial }{\partial t_k}  }g^{fibre})( e_l, e_j  )- (\nabla^{fibre}_j \mathcal{L}_{ \frac{\partial }{\partial t_k}  }g^{fibre})( e_l, e_i  )\} c( e_l)c_ic_j \\
&=-\frac{1}{4}\{(\nabla^{fibre}_i \mathcal{L}_{ \frac{\partial }{\partial t_k}  }g^{fibre})( I_k e_l, e_j  )- (\nabla^{fibre}_j \mathcal{L}_{ \frac{\partial }{\partial t_k}  }g^{fibre})( I_k e_l, e_i  )\} c(I_k e_l)c_ic_j \\
&= -\frac{1}{4}I_k^{S^+}\{(\nabla^{fibre}_i \mathcal{L}_{ \frac{\partial }{\partial t_k}  }g^{fibre})( I_k e_l, e_j  )- (\nabla^{fibre}_j \mathcal{L}_{ \frac{\partial }{\partial t_k}  }g^{fibre})( I_k e_l, e_i  )\} c_lc_ic_j .
\end{split}
\]
The last equality uses that $I_k^{S^+} (v\cdot s)=c(I_k v) \cdot s$ for negative spinor $s$ and tangent vector $v$. We can  simplify further by Corollary \ref{hyperkahlermetricvariationderivative} and equation (\ref{hyperkahlermetricvariationformulaCliffordaction}):
\[
\begin{split}
2\tilde{R}_k = &-\frac{1}{4}I_k^{S^+}(\nabla^{fibre}_i \mathcal{L}_{ \frac{\partial }{\partial t_k}  }g^{fibre})( I_k e_l, e_j  )c_lc_ic_j+ \frac{1}{2} I_k^{S^+} c(\nabla^{fibre}_j \mathcal{L}_{ \frac{\partial }{\partial t_k}  }\omega_k    )c_j \\
=&\frac{1}{4}I_k^{S^+}(\nabla^{fibre}_i \mathcal{L}_{ \frac{\partial }{\partial t_k}  }g^{fibre})( I_k e_l, e_j  )c_ic_lc_j 
+\frac{1}{2} I_k^{S^+}(\nabla^{fibre}_i \mathcal{L}_{ \frac{\partial }{\partial t_k}  }g^{fibre})( I_k e_i, e_j  )c_j \\
&+ \frac{1}{2} I_k^{S^+} c(\nabla^{fibre}_j \mathcal{L}_{ \frac{\partial }{\partial t_k}  }\omega_k    )c_j \\
=&
\frac{1}{2} I_k^{S^+} \{c_i c( \nabla^{fibre}_i \mathcal{L}_{ \frac{\partial }{\partial t_k}  }\omega_k        )+  
(\nabla^{fibre}_i \mathcal{L}_{ \frac{\partial }{\partial t_k}  }g^{fibre})( I_k e_i, e_j  )c_j \\
& +c(\nabla^{fibre}_i \mathcal{L}_{ \frac{\partial }{\partial t_k}  }\omega_k    )c_i
 \}\\
 =& \frac{1}{2} I_k^{S^+} \{c_i c( \nabla^{fibre}_i \mathcal{L}_{ \frac{\partial }{\partial t_k}  }\omega_k        )+  
(\nabla^{fibre}_i \mathcal{L}_{ \frac{\partial }{\partial t_k}  }g^{fibre})( I_k e_i, e_j  )c_j \}.
\end{split}
\]
The last step is because $(\nabla^{fibre}_i \mathcal{L}_{ \frac{\partial }{\partial t_k}  }\omega_k)$ acts trivially on positive spinors by the ASD property.

Now we apply $-I_k^{S^+}$ and sum over $k=1,2,3$. We get
\[
-\sum_k I^{S^+}_k\tilde{R}_k= \frac{1}{4}
\sum_{k} \{c_i c( \nabla^{fibre}_i \mathcal{L}_{ \frac{\partial }{\partial t_k}  }\omega_k        )+  
(\nabla^{fibre}_i \mathcal{L}_{ \frac{\partial }{\partial t_k}  }g^{fibre})( I_k e_i, e_j  )c_j \}.
\]
We recall that $\sum_k \mathcal{L}_{ \frac{\partial }{\partial t_k}  }\omega_k=0$ is precisely the condition $d_H\underline\Theta=0$ in Donaldson's adiabatic fibration. Therefore the first summand is zero, and
\[
\sum_k I^{S^+}_k\tilde{R}_k= -\frac{1}{4}\sum_{k,i,j} (\nabla^{fibre}_i \mathcal{L}_{ \frac{\partial }{\partial t_k}  }g^{fibre})( I_k e_i, e_j  )c_j.
\]
Now we write out the RHS by Corollary \ref{hyperkahlermetricvariationderivative} and equation (\ref{hyperkahlervariationcyclicsymmetry}). One obtains
\[
\begin{split}
\sum_k I^{S^+}_k\tilde{R}_k=
& -\frac{1}{4}\sum_{i,j}\nabla^{fibre}_i 
(\sum_k  \mathcal{L}_{ \frac{\partial }{\partial t_k}  }\omega_k) (e_i, e_j) c_j \\
& -\frac{1}{4}\sum_{i,j}\nabla^{fibre}_i 
 ( \mathcal{L}_{ \frac{\partial }{\partial t_2}  }\omega_3  - \mathcal{L}_{ \frac{\partial }{\partial t_3}  }\omega_2) (I_1 e_i, e_j) c_j \\
& -\frac{1}{4}\sum_{i,j}\nabla^{fibre}_i 
( \mathcal{L}_{ \frac{\partial }{\partial t_3}  }\omega_1  - \mathcal{L}_{ \frac{\partial }{\partial t_1}  }\omega_3) (I_2 e_i, e_j) c_j  \\
& -\frac{1}{4}\sum_{i,j}\nabla^{fibre}_i 
( \mathcal{L}_{ \frac{\partial }{\partial t_1}  }\omega_2  - \mathcal{L}_{ \frac{\partial }{\partial t_2}  }\omega_1) (I_3 e_i, e_j) c_j .
\end{split}
\]
But we recall that 
$
 \mathcal{L}_{ \frac{\partial }{\partial t_i}  }\omega_j=  \mathcal{L}_{ \frac{\partial }{\partial t_j}  }\omega_i
$
is precisely Donaldson's condition $d_H \underline{\omega}=0$. Thus this whole expression $\sum_k I^{S^+}_k\tilde{R}_k$ vanishes as required.
\end{proof}

\begin{rmk}
This proof makes essential use of Donaldson's adiabatic fibration conditions. The reader may compare this with the proof of the fact that $c(\underline\Theta)$ is parallel.
\end{rmk}

Geometrically, the operator $\tilde{R}$ arises from the variation of the \textbf{fibrewise Dirac operators} 
\[
D^-=\sum_{i=1}^4 c(e_i) \nabla^{LC,-}_i : \Gamma(M, S^-_X)\to \Gamma(M, S^+_X).
\]

\begin{lem}
The variation of the fibrewise Dirac operators is
\begin{equation}
[ \nabla_{\frac{\partial}{\partial t_k}   }^{LC}, D^-       ]=-\tilde{R}_k  - \frac{1}{2}\sum_{i,j} (\mathcal{L}_{ \frac{\partial}{\partial t_k}   }   g^{fibre})(e_i, e_j) c_j \nabla^{LC,-}_{ e_i   }.
\end{equation}
for an orthonormal basis $e_i$ on the fibre.
\end{lem}

\begin{proof}
We compute over the first order neighbourhood of $b\in B$ by taking coordinates $x_i$, $i=1,2,3,4$, and using the summation convention:
\[
\begin{split}
[ \nabla_{\frac{\partial}{\partial t_k}   }^{LC}, D^-       ]
&=[ \nabla_{\frac{\partial}{\partial t_k}   }^{LC},c( dx_i)  \nabla^{LC,-}_{\frac{\partial}{\partial x_i}   }      ]=[ \nabla_{\frac{\partial}{\partial t_k}   }^{LC}, c(dx_i)           ]\nabla^{LC,-}_{\frac{\partial}{\partial x_i}   }+c(dx_i) [ \nabla_{\frac{\partial}{\partial t_k}   }^{LC} , \nabla^{LC,-}_{\frac{\partial}{\partial x_i}   }   ] \\
&=c(  \nabla_{\frac{\partial}{\partial t_k}   }^{LC} dx_i       )\nabla^{LC,-}_{\frac{\partial}{\partial x_i}   }- c(dx_i) R(\frac{\partial}{\partial x_i}, \frac{\partial}{\partial t_k}   ) \\
& =c(  \nabla_{\frac{\partial}{\partial t_k}   }^{LC} dx_i       )\nabla^{LC,-}_{\frac{\partial}{\partial x_i}   }- \tilde{R}_k.
\end{split}
\]
Now using the formula for taking derivatives on 1-forms, Lemma \ref{adiabaticLeviCivitaonvertical1forms},
\[
\nabla^{LC}_{\frac{\partial}{\partial t_k}   } dx_i= \mathcal{L}_{ \frac{\partial}{\partial t_k}   }dx_i- \frac{1}{2}(\mathcal{L}_{ \frac{\partial}{\partial t_k}   }   g^{fibre})( (dx_i)^\sharp, \_    )= - \frac{1}{2}(\mathcal{L}_{ \frac{\partial}{\partial t_k}   }   g^{fibre})( (dx_i)^\sharp, \_    ).
\]
Now if we take $e_i$ an orthonormal basis on the tangent space of the fibre,
then
\[
c(  \nabla_{\frac{\partial}{\partial t_k}   }^{LC} dx_i       )\nabla^{LC,-}_{\frac{\partial}{\partial x_i}   }=- \frac{1}{2}(\mathcal{L}_{ \frac{\partial}{\partial t_k}   }   g^{fibre})(e_i, e_j) c_j \nabla^{LC,-}_{ e_i   }.
\]
The result follows.
\end{proof}

\begin{prop}
The variation of the Dirac operators satisfies
\begin{equation}\label{variationofuncoupledDiracoperator$G_2$instantonequation}
\sum_k I_k^{S^+} [ \nabla_{\frac{\partial}{\partial t_k}   }^{LC}, D^-       ]=0.
\end{equation}
\end{prop}

\begin{proof}
We apply (\ref{hyperkahlervariationcyclicsymmetry}) to get
\[
\begin{split}
&-(\mathcal{L}_{ \frac{\partial}{\partial t_k}   }   g^{fibre})(e_j, e_i)c_j \\
=& \{(\mathcal{L}_{ \frac{\partial}{\partial t_k}   }   \omega_1)(I_1 e_j, e_i)
+(\mathcal{L}_{ \frac{\partial}{\partial t_k}   }   \omega_2)(I_2 e_j, e_i)
+(\mathcal{L}_{ \frac{\partial}{\partial t_k}   }  \omega_3)(I_3 e_j, e_i)\} c_j
\\
=& \sum_{l=1}^3 (\mathcal{L}_{ \frac{\partial}{\partial t_k}   } \omega_l)(e_j, e_i) c(-I_l e_j  ) 
= - \sum_{l=1}^3 I^{S^+}_l(\mathcal{L}_{ \frac{\partial}{\partial t_k}   } \omega_l)(e_j, e_i) c( e_j  ) \\
=&  \sum_{l=1}^3 I^{S^+}_l c(\iota_{e_i} \mathcal{L}_{ \frac{\partial}{\partial t_k}   } \omega_l ),
\end{split}
\]
Hence 
by the previous lemma,
\[
[ \nabla_{\frac{\partial}{\partial t_k}   }^{LC}, D^-       ]=-\tilde{R}_k+  \frac{1}{2}\sum_{l=1}^3 I^{S^+}_l c(\iota_{e_i} \mathcal{L}_{ \frac{\partial}{\partial t_k}   } \omega_l ) \nabla^{LC,-}_{ e_i} .
\]
We have seen that $\sum_k I_k^{S^+} \tilde{R}_k=0$, so it remains to calculate
\[
\begin{split}
&\sum_{k,l,i} I_k^{S^+}  I^{S^+}_l c(\iota_{e_i} \mathcal{L}_{ \frac{\partial}{\partial t_k}   } \omega_l ) \nabla^{LC,-}_{ e_i} \\
=& \sum_{k,l,i,j} 
\epsilon_{klj} I^{S^+}_j 
 c(\iota_{e_i} \mathcal{L}_{ \frac{\partial}{\partial t_k}   } \omega_l ) \nabla^{LC,-}_{ e_i}
 -\sum_{k,i} c(\iota_{e_i} \mathcal{L}_{ \frac{\partial}{\partial t_k}   } \omega_k ) \nabla^{LC,-}_{ e_i} \\
=& \sum_{i,j} 
 I^{S^+}_j 
c(\iota_{e_i} 
\sum_{k,l}\epsilon_{klj}
\mathcal{L}_{ \frac{\partial}{\partial t_k}   } \omega_l ) \nabla^{LC,-}_{ e_i}
-\sum_{i} c(\iota_{e_i} \sum_k\mathcal{L}_{  \frac{\partial}{\partial t_k}   } \omega_k ) \nabla^{LC,-}_{ e_i}.
\end{split}
\]
Here we used that the operators $I_i^{S^+}$ satisfy the quaternionic relations 
\[
I_i I_j =-\delta_{ij}+\sum_k \epsilon_{ijk} I_k. \]
Now we observe $\sum_{k,l}\epsilon_{klj}
\mathcal{L}_{ \frac{\partial}{\partial t_k}   } \omega_l =0$ is equivalent to $d_H\underline{\omega}=0$, and $\sum_k\mathcal{L}_{  \frac{\partial}{\partial t_k}   } \omega_k =0$ is equivalent to $d_H\underline\Theta=0$, so the whole expression vanishes, and the result follows.
\end{proof}

\section{Fueter equations, adiabatic $G_2$ instantons, and moduli bundles}\label{Fueterequationsadiabatic$G_2$instantonsmodulibundles}

The work of Haydys \cite{Haydys}, whose ideas we shall borrow here, relates a degenerate version of Spin(7) instantons over the spinor bundle of a Riemannian 4-fold, to the Fueter equation for sections of some ASD moduli bundle over the 4-fold. This picture of higher dimensional gauge theory reducing to lower dimensional Fueter type equations, is a quite common phenomenon.

Here we wish to understand how this works for solutions of the adiabatic $G_2$ instanton equation (\ref{limiting$G_2$instanton}). After briefly recalling the definitions of Fueter equations in Section \ref{TheFueterequationrecap}, we introduce in Section \ref{Themodulibundle} a moduli bundle $\mathcal{M}$ over the base $B$, whose fibres are moduli spaces of ASD connections. We show how to put the relevant structures on the moduli bundle to make sense of the Fueter equation, and how to do computations with them. In Section \ref{Fueterequation$G_2$instantons} we show that under appropriate smoothness assumptions, the adiabatic $G_2$ instantons are equivalent to solutions of the Fueter equations on the moduli bundle, in a sense to be made clear. We then relate this to the variational viewpoint in Section \ref{Canonicalformsonmodulibundle}, and interpret this correspondence as a manifestation of the equality of a submanifold theoretic Chern-Simons functional with a gauge theoretic Chern-Simons functional. 

This turns out to be intimately related to the fact that, we can define a canonical 3-form, a canonical 4-form, and a fibre volume form on the moduli bundle, exactly as the data appearing in Donaldson's adiabatic fibration, and satisfying exactly the analogues of the relevant integrability conditions. (\cf Section \ref{Canonicalformsonmodulibundle}, \ref{Canonicalfibrevolumeform}).

\subsection{The Fueter equation}\label{TheFueterequationrecap}

The Fueter equation over a Riemannian 3-fold $B$ involves a fibre bundle $\mathcal{M}$ over $B$ with hyperk\"ahler fibres, equipped with an Ehresmann connection $\nabla^{\mathcal{M}}$, and we require that each tangent space of $B$ acts as the imaginary quaternions on the tangent bundle of the fibre. More concretely, take an orthonormal basis $\frac{\partial }{\partial t_1}$,$\frac{\partial }{\partial t_2}$,$\frac{\partial }{\partial t_3}$ of $T_b B$, then these act on the tangent bundle of the fibre $\mathcal{M}_b$, by the standard hypercomplex triple $I_1, I_2, I_3$. Using this set of data, one can define the (nonlinear) Dirac type operator \[\slashed{D}=I_1 \nabla^{\mathcal{M}}_{\frac{\partial }{\partial t_1}   }+I_2 \nabla^{\mathcal{M}}_{\frac{\partial }{\partial t_2}   }+I_3 \nabla^{\mathcal{M}}_{\frac{\partial }{\partial t_3}   } \]
on sections of the bundle $\mathcal{M}\to B$. Here $\nabla^{\mathcal{M}}_{\frac{\partial }{\partial t_i}   }s$ is the vertical projection of the vector $ds(\frac{\partial }{\partial t_i})$. The \textbf{Fueter equation} is
\begin{equation}
\slashed{D}s=0.
\end{equation}
\begin{rmk}
There is also a parallel story related to $Spin(7)$ manifolds over 4-folds as discussed in \cite{Haydys}.	
\end{rmk}

\subsection{The moduli bundle}\label{Themodulibundle}

We need to construct an appropriate \textbf{moduli bundle} $\mathcal{M}$ over $B$ with additional data, to make sense of the Fueter equation. For backgrounds on moduli space of bundles on K3 surfaces, see \cite{Mukaidualitypaper}. Recall we have a principal $U(r)$ bundle $P$ over the 7-fold $M$. The restriction to each K3 fibres $M_b$ specifies the topological data to determine a Mukai vector $v$, thus giving the ASD instanton moduli space on $M_b$ with Mukai vector $v$, which in general is a hyperk\"ahler manifold under the appropriate smoothness assumptions. One can think of this moduli space as obtained by applying to each K3 fibre $M_b$ a standard hyperk\"ahler quotient construction from the affine Euclidean space $\mathcal{A}_b$ of all connections on $P|_{M_b}\to M_b$, under the gauge group action, \ie
\[
\mathcal{M}_b = \mathcal{A}_b /// \mathcal{G}_b= \mu_b^{-1}(0)/\mathcal{G}_b,
\]
where $\mu_b$ is the fibrewise hyperk\"ahler moment map, and $\mathcal{G}_b$ is the gauge group of $P|_{M_b}\to M_b$. We assume $0$ is a regular value for all the moment maps $\mu_b$, and the gauge groups act freely (modulo central $S^1$ which acts trivially), so no analytical issue arises and every construction can be done smoothly.
 At least set theoretically, these ASD moduli spaces $\mathcal{M}_b$ shall be the fibres of the moduli bundle $\mathcal{M}$. Similarly, the spaces $\mathcal{A}_b$ shall be the fibres of the inifinite dimensional bundle $\mathcal{A}\to B$, and the spaces $\mu_b^{-1}(0)$ shall be the fibres of the bundle $\mu^{-1}(0)\to B$.
 Schematically, one is encouraged to think of the moduli bundle as a `relative hyperk\"ahler quotient':  $\mathcal{M}=\mathcal{A}///\mathcal{G}=\mu^{-1}(0)/\mathcal{G}$, meaning by this no more than a fibrewise construction.

\begin{rmk} 
 The convention we are using for the hyperk\"ahler metric on K3 surfaces is (\cf \cite{Mukaidualitypaper})
 \begin{equation}\label{hyperkahlermetric}
 g^{\mathcal{M}_b }(a,b)=\frac{1}{4\pi^2}\int_{M_b} \langle a, b \rangle d\text{Vol}_{M_b}, \quad a, b\in T_A \mathcal{M}_b,
 \end{equation}
 \begin{equation}
 \label{hyperkahlerforms}
 \omega_i^{\mathcal{M}}(a, b)=g^{\mathcal{M}_b }(I_i a, b) =-\frac{1}{4\pi^2} \int_{M_b} \Tr (a\wedge b)\wedge \omega_i.
 \end{equation}
 \begin{equation}
 I_ia=-a\circ I_i.
 \end{equation} 
The tangent space to the moduli space $T_A\mathcal{M}_b$ is identified as the subspace of $T_A\mathcal{A}_b=\Omega^1(M_b, ad P|_{{M}_b})$ defined by the linearised ASD equation and the Coulumb gauge condition. For later use, we will denote the orthogonal projection operator to this subspace as $pr_{T_A\mathcal{M} }$.
\end{rmk}

The tangent vectors of $B$ are naturally identified, via contraction with $\underline{\omega}$, with the hyperk\"ahler forms on the fibres of $M\to B$, so naturally induce hyperk\"ahler forms on the fibres of $\mathcal{M}\to B$. This gives the imaginary quaternion action needed to define the Fueter equation.

Our next goal is to define a \textbf{canonical horizontal distribution} on the moduli bundle.

Now $P\to M\to B$ displays $P$ as a fibre bundle over $B$. Donaldson's adiabatic fibration gives a horizontal distribution $H$ for $M\to B$. Given a connection $A$ for the bundle $P\to M$, there is a naturally induced connection $\tilde{H}$ on $P\to B$ covering $H$; viewing connections as a way to lift vector fields, this is just the composition of lifting from $B$ to $M$, and then from $M$ to $P$.

A choice of $\tilde{H}$ covering $H$ gives an infinitesimal trivialisation of $P\to B$ over a first order neighbourhood of $b\in B$, so induces an infinitesimal trivialisation, \ie a connection $\nabla^{\tilde{H} }$ of the  bundle $\mathcal{A}\to B$. As a covariant derivative, this takes a section $\tilde{s}$ of $\mathcal{A}\to B$, and outputs $\nabla^{\tilde{H} }_{\frac{\partial }{\partial t_i}} \tilde s$, which at each point $b\in B$ takes value in the infinite dimesnional Euclidean space $\Omega^1(M_b, ad P|_{M_b})$.
 This induces a submanifold connection $\overline{\nabla}^{\tilde{H} }$ on the fibre bundle $\mu^{-1}(0)\to B$. Given a section $\tilde{s}$ for $\mu^{-1}(0)\to B$, the induced connection  $\overline{\nabla}^{\tilde{H} }\tilde{s}$ at the point $b\in B$ is just $\nabla^{\tilde{H} }\tilde{s}$ orthogonally projected in $\Omega^1(M_b, ad P|_{M_b})$ to the tangent space of $\mu_b^{-1}(0)$.

Now we can describe the canonical horizontal distribution on the moduli bundle $\mathcal{M}\to B$. Take a section $s$, lift it to a section $\tilde{s}$ of the bundle $\mu^{-1}(0)\to B$, calculate $\overline{\nabla}^{\tilde{H} }\tilde{s}$, and then project down to the quotient $\mathcal{M}\to B$: this is the covariant derivative of $s$ with respect to the canonical horizontal distribution. This construction depends on the choice of the lift $\tilde{H}$ and the lift $\tilde{s}$, but the 
 ambiguity lies in the gauge group $\mathcal{G}_b$ action and is invisible to the quotient. 

\begin{prop}
There exists a canonical horizontal distribution $\nabla^{\mathcal{M}}$ on the moduli bundle $\mathcal{M}\to B$. In particular the Fueter equation is well defined.
\end{prop}

We need to understand how to compute with this canonical horizontal distribution in practice. Let $A$ be a connection on $P\to M$, which induces the lift $\tilde{H}$ of $H$. The section $s$ of $\mathcal{M}\to B$ is represented by a section $\tilde{s}$ of $\mathcal{A}\to B$, which in turn is represented by a connection $A'$ of $P\to M$. Here $A'$ restricted to the K3 fibres precisely gives the fibrewise values of $\tilde{s}$ in $\mathcal{A}_b$. We can think of connections $A$, $A'$ as covariant derivatives $\nabla^A$, $\nabla^{A'}$. The quantity $\nabla^{\tilde{H} }\tilde{s}$ takes value at each $b\in B$ in the vector space $\Omega^1(M_b, ad P|_{M_b}  )$.

\begin{lem}\label{canonicalconnectionmodulibundle}
The covariant derivative $\nabla^{\tilde{H} }_{\frac{\partial }{\partial t_i}} \tilde{s} $ at the point $b\in B$  can be represented as the commutator $[ \nabla^A_{ \frac{\partial }{\partial t_i}   }, \nabla^{A'}  ]$ restricted to the fibre $M_b$. Here we use the horizontal distribution $H$ to lift $\frac{\partial }{\partial t_i}$ to $M$.
\end{lem}

\begin{proof}
The discussion only involves the first order neighbourhood of $b\in B$. We can suppose the local structure of the  bundle $P\to M$ to be the Cartesian product of $P|_{M_b}\to M_b$ with the first order neighbourhood of $b\in B$, and the component of the connection $ \nabla^A_{ \frac{\partial }{\partial t_i}   }$ to be just the trivial $\frac{\partial }{\partial t_i}$. Then the result is clear.
\end{proof}

\begin{cor}
If the connection $A$ on $P\to M$ represents the section $s$ of $\mathcal{M}\to B$, then $ \nabla^\mathcal{M}_{\frac{\partial }{\partial t_i}} s$ is represented by the contraction of the curvature $ \iota_{ \frac{\partial }{\partial t_i}}  F_A $ orthogonally projected to $T_A\mathcal{M} \subset \Omega^1(M_b, ad P|_{M_b}  )$.
\end{cor}

\begin{proof}
If we take $A=A'$ in the above lemma, we find $\iota_{ \frac{\partial }{\partial t_i}}  F_A $ represents $\nabla^{\tilde{H} }_{\frac{\partial }{\partial t_i}} \tilde{s} $. Now we recall the description of hyperk\"ahler quotient construction (\cf \cite{Mukaidualitypaper}) gives a canonical orthogonal decomposition
\begin{equation}\label{decompositionhyperkahlerquotient}
\begin{split}
\Omega^1(M_b, ad P|_{M_b}  )= & T_A\mathcal{M}  \bigoplus \\ & (\Lie \mathcal{G}_b)  A \oplus I_1(\Lie \mathcal{G}_b)  A \oplus I_2(\Lie \mathcal{G}_b)  A \oplus I_3(\Lie \mathcal{G}_b)  A,
\end{split}
\end{equation}
where $T_A\mathcal{M}$ is invariant under the quaternionic action, and $(\Lie \mathcal{G}_b)  A$ are the linearised deformations of $A|_{M_b}$ generated by the gauge group action, \ie
\[
(\Lie \mathcal{G}_b)  A=\{ d_{A|_{M_b}}\Phi: \Phi\in \Omega^0(M_b, ad P|_{M_b})      \}
.\] The orthogonal projection to $T_A\mathcal{M}$ is just the natural consequence of our prescription.
\end{proof}

\begin{cor}
The expression of the nonlinear Dirac operator defining the Fueter equation is 
\[
\slashed{D} s=pr_{T_A \mathcal{M} }\sum I_i (\iota_{ \frac{\partial }{\partial t_i}}  F_A ).
\]
\end{cor}

\begin{proof}
The only subtlety is to notice is that the action of complex structure on bundle valued 1-forms is compatible with the action on the tangent bundle of the moduli spaces, when we do the hyperk\"ahler reduction. 
\end{proof}

\subsection{Fueter equation and adiabatic $G_2$ instantons}\label{Fueterequation$G_2$instantons}

We can now establish the picture that under appropriate smoothness assumptions (the moment maps have zero as a regular value, and the gauge group action is free modulo central $S^1$), adiabatic $G_2$ instantons should be equivalent to Fueter sections on the moduli bundle.

\begin{rmk}
For $U(r)$ connections on a hyperk\"ahler K3 surface, Serre duality implies that the obstruction space for deformations of ASD connections vanishes if and only if the ASD connection is irreducible. Thus the smoothness assumption just means irreducibility. This Section presents the case of $U(r)$ adiabatic $G_2$-instantons, but the arguments also works for $PU(r)$ adiabatic $G_2$-instantons with cosmetic changes.
\end{rmk}

\begin{prop}
In our previous setup,
if $A$ is an adiabatic $G_2$ instanton, then the section $s$ of the moduli bundle $\mathcal{M}\to B$ represented by $A$ satisfies the Fueter equation.
\end{prop}

\begin{proof}
 We compute $\slashed{D}s$ using the basic linear algebraic model in Section \ref{linearalgebraicmodelsection}:
\[
 *_4 (\sum I_i  \iota_{\frac{\partial }{\partial t_i}   } F_A )=  -\iota_{\frac{\partial }{\partial t_i}   } F_A \wedge \omega_i,
 \]
Here complex structures $I_i$ act on 1-forms by the minus of precomposition, so picks up a negative sign.
On the other hand, the (1,1) type component of $F_A$  in the horizontal-vertical decomposition of forms can be computed by $F_A^{(1,1)}=\sum dt_i \wedge \iota_{\frac{\partial }{\partial t_i}   } F_A$, so
\[
\begin{split}
F^{(1,1)} \wedge \underline\Theta= -\sum_{cyc} dt_i \wedge\iota_{\frac{\partial }{\partial t_i}   } F_A \wedge \omega_i dt_j dt_k
=(\sum_i  \iota_{\frac{\partial }{\partial t_i}   } F_A \wedge \omega_i )\wedge dt_1dt_2dt_3. 
\end{split}
\]
Recall that adiabatic $G_2$ instantons are ASD when restricted to K3 fibres. The claim is now clear.
\end{proof}

The next aim is to show the converse.

\begin{prop}\label{Fueterimpliesmonopole}
If $s$ is a section of the moduli bundle solving the Fueter equation, then we can find a representing connection $A$ which is an adiabatic $G_2$-instanton.
\end{prop}

\begin{proof}
Since we are working over a contractible basis, and we are assuming the gauge group action is essentially free,  we can always lift $s$ to be represented by some connection $A'$ on $P\to B$.	
Using the orthogonal decomposition (\ref{decompositionhyperkahlerquotient}) in the description of the hyperk\"ahler quotient, if $A'$ defines a Fueter section, then $\sum_i I_i\iota_{ \frac{\partial }{\partial t_i} } F_{A'} $ has vanishing $T_{A'}\mathcal{M}$ component, so can be written as \[d_{A'} \Phi_0+ I_1 d_{A'} \Phi_1+I_2 d_{A'} \Phi_2+ I_3 d_{A'} \Phi_3,\]
where $\Phi_i$ live in $\Gamma(M_b, ad P|_{M_b})$ on each K3 fibre.
Modulo the issue of central constant in $u(1)$, the choices of $\Phi_i$ are unique. Now we modify $A'$ to $A=A'+\sum \Phi_i dt_i$, which can be defined in a smooth way independent of the orthonormal basis $dt_i$. Then we have $\sum I_i \iota_{ \frac{\partial }{\partial t_i} } F_A= d_A \Phi_0$, \ie
\begin{equation}\label{adiabaticmonopole11typecurvature}
F_A^{(1,1)} \wedge \underline\Theta=\sum
\iota_{ \frac{\partial }{\partial t_i} } F_A \wedge \omega_i dt_1dt_2dt_3= (*_4 d_A \Phi_0) \wedge \underline{\lambda}.
\end{equation}
Thus we have produced the adiabatic monopole equation. But by Section \ref{adiabaticlimitof$G_2$instantons}, adiabatic $G_2$ instantons are equivalent to adiabatic $G_2$ monopoles.
\end{proof}

\begin{prop}\label{instantonFuetercorrespondenceuniqueness}
Two adiabatic $G_2$ instantons representing the same Fueter section are related by two kinds of operations. The first is to apply a gauge transformation of $P\to M$. The second is to add to the connection $A$ a 1-form on $B$ with values in the central $u(1)\subset \Gamma(M_b, ad P|_{M_b})$.
\end{prop}

\begin{proof}
Let $A$ and $A'$ be the two adiabatic $G_2$ instantons. On each fibre, we can apply some gauge transformation to move from $A|_{M_b}$ to $A'|_{M_b}$, and by the free action assumption this is unique up to an $S^1$ worth of ambiguity. On a contractible base $B$, we can find a smooth gauge transform to achieve so globally. So without loss of generality, $A$ and $A'$ agree as connections on each fibre. Then $A=A'+ \sum \Phi_i dt_i$, and our previous calculations imply $d_A \Phi_i=0$ on each K3 fibre, so $\Phi_i$ have to take value in the central $u(1)$.
\end{proof}

\begin{rmk}
If the smoothness assumptions fail on part of the moduli spaces $\mathcal{M}_b$, the discussion is still valid if the Fueter section avoids the singular locus of the moduli bundle. This is intimately related to the question of extending the instanton-Fueter correspondence to (partial) compactification of the moduli bundle, as mentioned by Haydys \cite{Haydys}. An interesting question is to understand what happens to the actual $G_2$ instantons arising from perturbations of the adiabatic model, but we shall not address this issue in this paper.
\end{rmk}

\subsection{Canonical 3-form and 4-form on the  moduli bundles}\label{Canonicalformsonmodulibundle}

The moduli bundle is equipped with the fibrewise hyperk\"ahler structures $\omega^{\mathcal{M } }_i$, the imaginary quaternion action of $T_b B$ on the tangent bundle of the fibre $M_b$, and the canonical horizontal distribution $\nabla^{\mathcal{M}}$. This is curiously analogous to some of the data appearing in Donaldson's adiabatic fibration. More precisely, this specifies a \textbf{canonical 3-form} and a \textbf{canonical 4-form} on $\mathcal{M}$, defined by
\[
\underline{\omega}^{\mathcal{M} }= \sum \omega^\mathcal{M}_i dt_i, \quad \underline{\Theta}^{\mathcal{M} }= -\sum_{cyc} \omega^\mathcal{M}_i dt_j dt_k,
\]
where $dt_1, dt_2, dt_3$ are an orthonormal basis of $T_b B$, and we use the horizontal distribution to graft $\omega_i^\mathcal{M}$ from fibres to $\mathcal{M}$. This is easily seen to be well defined independent of the choice of orthonormal basis.

The canonical 4-form is motivated by a variational viewpoint on the Fueter equation.
We consider a map
$N: [0,1]\times B\to \mathcal{M}$, which at fixed time $t$ restricts to a section $s_t$ of the moduli bundle, and defines a homotopy between a fixed section $s_0$ and $s=s_{t=1}$. Using the argument in Section \ref{adiabaticlimitofassociative}, it is not hard to see

\begin{lem}
	Critical points of the Chern-Simons type functional on $s$
	\[
	CS^{associative}(s)=\int_{[0,1]\times B} N^* \underline{\Theta}^{\mathcal{M} }
	\]
	are precisely the solutions to the Fueter equation.
	
\end{lem}
\begin{rmk}
We choose this notation in analogy with our discussions of adiabatic associative sections,  and we think of $N$ as a cylindrical  submanifold in $\mathcal{M}$. It will turn out to be invariant under homotopies fixing the boundary, so we think of this as a function of $s$ rather than $N$.	
\end{rmk}

Now we represent the 1-parameter family of sections of the moduli bundle by a 1-parameter family $A_t$ of connections on $P\to M$ joining $A_0$ to $A$, parametrised by $t\in [0,1]$. Restricted to the K3 fibres, the connections $A_t$ will be ASD. It is interesting to calculate the above Chern-Simons functional from the gauge theoretic viewpoint.

\begin{prop}
The above Chern-Simons type functional agrees with the Chern-Simons functional (\ref{ChernSimons}) defined by viewing $A_t$ as connections, \ie
\begin{equation}\label{ChernSimonssubmanifoldequalsChernSimonsgauge}
CS^{associative} (s) =CS^{instanton}(A)=\frac{-1}{4\pi^2}\int_{M\times [0,1]} \Tr(F_{A_t} \wedge \frac{ \partial A_t}{\partial t}   ) \wedge \underline{\Theta} \wedge dt.
\end{equation}
\end{prop}

\begin{proof}
We need to understand the volume element $N^*\underline{\Theta}^\mathcal{M}$. This means to evaluate
\[
\underline{\Theta}^{\mathcal{M} }( ds_t( \frac{\partial}{\partial t_1} ) , ds_t( \frac{\partial}{\partial t_2} ), ds_t( \frac{\partial}{\partial t_3} ) , pr_{T_{A_t}\mathcal{M} }\frac{ \partial A_t}{\partial t}     )
\]
Now by the definition of $\underline\Theta^{\mathcal{M} }$ and the description of the canonical horizontal distribution, this is
\[
-\sum_i{\omega_i}^\mathcal{M} ( pr_{T_{A_t}\mathcal{M }} (\iota_{ \frac{\partial}{\partial t_i}  }  F_{A_t} ), pr_{T_{A_t}\mathcal{M} }\frac{ \partial A_t}{\partial t}         )
=-g^{\mathcal{M}_b}(  \sum_i I_i (\iota_{\frac{\partial}{\partial t_i} }   F_{A_t}   )  , pr_{T_{A_t}\mathcal{M} }\frac{ \partial A_t}{\partial t}          ),
\]
where $g^\mathcal{M}$ is the moduli $L^2$ metric (\ref{hyperkahlermetric}). It is convenient that we can drop the orthogonal projection operator of $ \sum_i I_i (\iota_\frac{\partial}{\partial t_i}    F_{A_t}   ) $, because it is paired with $pr_{T_{A_t}\mathcal{M} }\frac{ \partial A_t}{\partial t}  \in T_{A_t}\mathcal{M}$.

Now recalling the construction of the hyperk\"ahler quotient and the description of the symplectic forms (\ref{hyperkahlerforms}), one can write the above formula as
\[
\frac{1}{4\pi^2}\int_{M_b} \sum_i \Tr (   ( \iota_{ \frac{\partial}{\partial t_i} } F_{A_t}) \wedge pr_{T_{A_t}\mathcal{M} }\frac{ \partial A_t}{\partial t}         ) \wedge \omega_i.
\]
This can be interpreted as
\[
\begin{split}
N^*\underline{\Theta}^\mathcal{M} & =\frac{1}{4\pi^2}\int_{M_b} \sum_i \Tr (   ( \iota_{ \frac{\partial}{\partial t_i} } F_{A_t}) \wedge pr_{T_{A_t}\mathcal{M} }\frac{ \partial A_t}{\partial t}         ) \wedge  \omega_i \wedge dt_1dt_2dt_3dt\\
& =\frac{-1}{4\pi^2} \int_{M_b} \Tr (F_{A_t}^{(1,1)} \wedge pr_{T_{A_t}\mathcal{M} }\frac{ \partial A_t}{\partial t}         )\wedge \underline{\Theta} \wedge dt.
\end{split}
\]
Since the connections are ASD when restricted to fibres, this is equal to
\[
\frac{-1}{4\pi^2}\int_{M_b} \Tr (F_{A_t} \wedge pr_{T_{A_t}\mathcal{M} }\frac{ \partial A_t}{\partial t}         )\wedge \underline{\Theta} \wedge dt.
\]
Thus the Chern-Simons type functional
\[
\begin{split}
CS^{associative}(s)= & \frac{-1}{4\pi^2}\int_{B\times [0,1]}\int_{M_b } \Tr (F_{A_t} \wedge pr_{T_{A_t}\mathcal{M} }\frac{ \partial A_t}{\partial t}         )\wedge \underline{\Theta} \wedge dt \\
= & \frac{-1}{4\pi^2}\int_{M\times [0,1]} \Tr (F_{A_t} \wedge pr_{T_{A_t}\mathcal{M} }\frac{ \partial A_t}{\partial t}         )\wedge \underline{\Theta} \wedge dt .
\end{split}
\]
Now we observe $\frac{ \partial A_t}{\partial t}   $ is the variation of a family of ASD connections on $M_b$, so is tangent to $\mu_b^{-1}(0)$. This means $\frac{ \partial A_t}{\partial t} $ agrees with $pr_{T_{A_t}\mathcal{M} }\frac{ \partial A_t}{\partial t}          $ up to an element $d_{A_t} \Phi$ in $(\Lie \mathcal{G}_b) A_t$. Now using the Bianchi identity $d_{A_t} F_{A_t}=0$, and $d\underline{\Theta}=0$,
\[
\int_M \Tr (F_{A_t} \wedge d_{A_t} \Phi       )\wedge \underline{\Theta}=\int_M d \{ \Tr( F_{A_t}\wedge \Phi)    \wedge \underline{\Theta} \}.
\]
Here $M$ is noncompact, so we need to be a bit careful: $M$ is fibred over $B$, and the boundary of $M$ is just the preimage of $\partial B$, denoted $\pi^{-1}(\partial B)$. Using Stokes theorem, we can rewrite the above as
\[
\int_{\pi^{-1}(\partial B)} \Tr( F_{A_t}\wedge \Phi)    \wedge \underline{\Theta}=-\int_{\partial B} \sum_{cyc} dt_jdt_k \int_{M_b} \Tr( F_{A_t}\wedge \Phi)\wedge \omega_i =0.
\]
The final term is zero, because $F_{A_t}$ is ASD.

The upshot is that we can drop the projection operator, to get (\ref{ChernSimonssubmanifoldequalsChernSimonsgauge}).
\end{proof}

This provides a deeper explanation why the adiabatic $G_2$ instantons correspond to solutions of the Fueter equation on the dual side: they are critical points of essentially the same Chern-Simons functional. This suggests the correspondence may be a classical manifestation of a quantum duality, but it is beyond the author's competence to explore the quantum theory.

\begin{rmk}
Here is a subtlety. The gauge theoretic Chern-Simons functional can be defined for general connections, or for connections which restrict to ASD connections on K3 fibres. The concept of critical points will change, because we allow for different classes of variations. So it is more accurate to say adiabatic $G_2$ instantons are particular solutions to the critical point condition. Similar remarks apply to the parallel discussion about the 3-form $\underline{\omega}^\mathcal{M}$ below.
\end{rmk}

Another interesting consequence of the above proposition is

\begin{prop}
The canonical 4-form $\underline{\Theta}^\mathcal{M}$ is closed.
\end{prop}

\begin{proof}
We can replace the base $B$ by any small smooth domains $B'\subset B$ in the above, and express 
\[
\int_{[0,1] \times B'} N^*\underline{\Theta}^\mathcal{M}
\]
as a gauge theoretic Chern-Simons functional by (\ref{ChernSimonssubmanifoldequalsChernSimonsgauge}), which is well defined up to boundary fixing homotopies because $d\underline{\Theta}=0$. But by Stokes theorem, for this homotopy invariance to be true on any cylindrical maps $N:[0,1]\times B'\to \mathcal{M}$, it is necessary and sufficient for $\underline{\Theta}^\mathcal{M}$ to be closed.
\end{proof}

Now we do an analogous discussion for the canonical 3-form $\underline{\omega}^\mathcal{M}$. Consider $\Sigma\subset B$ a 2-submanifold with boundary. We can restrict $P\to M$ over the locus $\Sigma$ to a principal bundle  $P|_\Sigma\to M|_\Sigma$. Recall $\underline{\omega}$ on $M$ restricted to $M|_\Sigma$ is closed, because $d_f\underline{\omega}=0$, $d_H \underline{\omega}=0$. Conversely, if this happens for all $\Sigma$, then $d_f\underline{\omega}=0$, $d_H \underline{\omega}=0$.

Let $s: \Sigma\to \mathcal{M}|_\Sigma$ be any section. Fix an arbitrary section $s_0$ as a base point, so $s(\Sigma)$ and $s_0$ are viewed as the boundary of the 3-fold $N: [0,1] \times \Sigma \to \mathcal{M}|_\Sigma$, and at each time $t$ we have a section $s_t: \Sigma\to \mathcal{M}|_{\Sigma}$.
We consider the \textbf{Chern-Simons type functional} 
\begin{equation}
CS^{holo}(s)=\int_{[0,1]\times \Sigma} N^*\underline{\omega}^{\mathcal{M} },
\end{equation}
which is a priori dependent on $N$. This functional can be interpreted in terms of gauge theory on $M$, by representing $s_t$ as a 1-parameter family of connections $A_t$ on $P|_\Sigma\to M|_\Sigma$. By a computation very similar to before,

\begin{prop}
The above Chern-Simons type functional is
\begin{equation}\label{ChernSimonsholomorphiccurveholomorhicbundle}
CS^{holo}(s)=\frac{1}{4\pi^2}\int_{ [0,1]\times M|_\Sigma} \Tr(F_{A_t} \wedge \frac{\partial A_t}{\partial t })\wedge \underline{\omega}\wedge dt.
\end{equation}
\end{prop}

The key point is that the RHS of (\ref{ChernSimonsholomorphiccurveholomorhicbundle}) is invariant under boundary fixing homotopies, because $d\underline{\omega}=0$ on $M|_{\Sigma}$. Thus so must the LHS. This implies

\begin{prop}
The canonical 3-form $\underline{\omega}^\mathcal{M}$ is closed when restricted to $\mathcal{M}|_{\Sigma}$.
\end{prop}

Now since this is true for any choice of $\Sigma$, we know $d\underline{\omega}^{\mathcal{M}  }$ is a type (3,1) form. This is analogous to the following conditions in Donaldson's adiabatic fibration
\[
d_f \underline \omega=0, \quad d_H \underline \omega=0,
\]
which says exactly $d\underline\omega$ is a type (3,1) form.

\begin{rmk}
Conan Leung and J-H Lee \cite{Conan} had the idea of defining canonical 3-forms and 4-forms on moduli spaces of submanifold theoretic and gauge theoretic objects, which are related to ours in spirit but not in details.
\end{rmk}

\subsection{Canonical fibre volume form on the moduli bundle}\label{Canonicalfibrevolumeform}

We see in Section \ref{Canonicalformsonmodulibundle}  that most of the conditions in Donaldson's adiabatic fibrations have exact analogues for the moduli bundle $\mathcal{M}\to B$. This could be taken further to suggest that there is some kind of special fibration structure on $3+4m$ dimensions, which is analogous to $G_2$ geometry. To provide further evidence in this direction, we describe the \textbf{canonical fibre volume form} on the moduli bundle, which is analogous to $\underline{\mu}$ in Donaldson's adiabatic fibration, and we show it satisfies the relevant integrability condition.

On each fibre $\mathcal{M}_b$ of the moduli bundle, there is a volume form given by \[ \underline{\mu}^{\mathcal{M}  }=\frac{1}{(2m)!}(\omega_1^\mathcal{M})^{2m}=\frac{1}{(2m)!}(\omega_2^\mathcal{M})^{2m}=\frac{1}{(2m)!}(\omega_3^\mathcal{M})^{2m} ,\]
where $m$ is the quaternionic dimension of $\mathcal{M}|_b$. Using the canonical horizontal distribution $\nabla^{\mathcal{M} }$, this $\underline{ \mu }^{\mathcal{M} }$ is defined as a vertical form on the moduli bundle in degree $4m$.

\begin{prop}
The canonical fibre volume form $\underline{\mu}^{ \mathcal{M} }$ satisfies the integrability condition $d_{\nabla^\mathcal{M}  }\underline{\mu}^{ \mathcal{M} }=0$ with respect to the horizontal distribution $\nabla^\mathcal{M}$.
\end{prop}

\begin{rmk}
This is exactly the analogue of the condition $d_H\underline{\mu}=0$ for Donaldson's adiabatic fibrations. The notation $d_{\nabla^\mathcal{M}  }$ has the same meaning as $d_H$, namely a certain component of the exterior differentiation operator $d$.
\end{rmk}

\begin{proof}
We use $\nabla^{\mathcal{M}    }$ to give an infinitesimal trivialisation around $\mathcal{M}_b$ for $b\in B$; by the definition of $\nabla^{\mathcal{M}    }$, this is induced by an infinitesimal trivialisation of $P\to M\to B$ around $P|_{M_b}\to M_b\to b$, compatible with the horizontal distribution $H$. Let $\frac{\partial}{\partial t}\in T_b B$ be a tangent vector, which has a horizontal lift to the moduli bundle, denoted still by $\frac{\partial}{\partial t}$. We are required to show the Lie derivative $\mathcal{L}_{ \frac{\partial}{\partial t}   }\underline{\mu}^\mathcal{M}=0$, which in the trivialisation can be written simply as $\frac{\partial}{\partial t}   \underline{\mu}^\mathcal{M}=0$.

Let $a_j$ for $j=1, 2, \dots , 4m$ be an orthonormal basis of the tangent space of $A\in \mathcal{M}_b$. We can think of $a_j$ concretely as $ad(P|_{M_b})$-valued 1-forms on $M_b$. As a general fact in Riemannian geometry, the variation of the fibrewise volume can be written in terms of variation of Riemannian metric
\[
\frac{\partial}{\partial t}   \underline{\mu}^\mathcal{M}= \frac{1}{2}\sum_j\frac{\partial g^{\mathcal{M}_b }}{\partial t} ( a_j, a_j).
\] 
The crucial observation is that $T_A\mathcal{M_b}$ is a quaternionic module: for eack $k$, the $I_k a_j$ for $j=1, \dots 4m$ is also an orthonormal basis of $T_A\mathcal{M}_b$. Hence
\[
\frac{\partial}{\partial t}   \underline{\mu}^\mathcal{M}= \frac{1}{8}\sum_j \{\frac{\partial g^{\mathcal{M}_b }}{\partial t} ( a_j, a_j)+  \sum_{k=1}^{3}
\frac{\partial g^{\mathcal{M}_b  }}{\partial t} ( I_k a_j, I_ka_j)
\}.
\]
Now recalling the formula for the $L^2$ moduli metric (\ref{hyperkahlermetric}), we have
\[
\frac{\partial g^{\mathcal{M}_b} }{\partial t} ( a_j, a_j)= \frac{1}{4\pi^2} \int_{M_b} -\Tr_{u(r)}\{ \frac{\partial g_b}{\partial t  }( a_j, a_j)\} d\text{Vol}_{M_b}.
\]
Here $-\Tr_{u(r)}$ is the trace pairing on the $ad(P|_{M_b})$ part, and $\frac{\partial g_b}{\partial t  }$ is a pointwise symmetric bilinear form on the 1-form part of $a_j$, which is obtained by the Lie derivative of the fibrewise K3 metric $g_{M_b}$.

Thus up to a numerical factor $\frac{\partial  }{\partial t} \underline{\mu}^{\mathcal{M} }$ 
is the integral of
\[
-\frac{1}{2}\Tr_{u(r)}\{\frac{\partial g_b}{\partial t} ( a_j, a_j)+  \sum_{k=1}^{3}
\frac{\partial g_b}{\partial t} ( I_k a_j, I_ka_j)\} d\text{Vol}_{M_b}.
\]
A pointwise calculation on the K3 surface $M_b$ shows that this integrand is 
\[
-\Tr_{u(r)}\{g_{M_b} ( a_j, a_j)\} \frac{\partial }{\partial t} d\text{Vol}_{M_b}.
\]
This is because if we average the symmetric bilinear form $\frac{\partial g_b}{\partial t}$ over an orthonormal basis with respect to $g_b$, we will get a trace scalar multiplied by $g_b$.

But $d_H \underline{\mu}=0$ precisely means the fibre volume form does not have first order variation, so $\frac{\partial }{\partial t} d\text{Vol}_{M_b}=0$. This implies the above integrand is zero, so the integral  $\frac{\partial}{\partial t}   \underline{\mu}^\mathcal{M}=0$ as claimed.
\end{proof}

\begin{rmk}
The fibre volume $\int_{M_b} \underline{\mu}^\mathcal{M}$ is in fact intimately related to Donaldson's polynomial invariants, and the Proposition gives a local mechanism for its deformation invariance in this special case. (Compare Section \ref{Dualityofmaximalsubmanifoldequations})
\end{rmk}

We can now collect the integrability conditions on the canonical forms we constructed on the moduli bundle. The reader will not fail to notice the striking resemblance with the conditions defining Donaldson's adiabatic fibrations (\ref{adiabaticclosed}), (\ref{adiabaticcoclosed}), (\ref{maximalsubmanifold}).

\begin{thm}\label{canonicalformsintegrability}
The canonical 3-form $\underline{\omega}^{\mathcal{M } }$, the canonical 4-form $\underline{\Theta}^{\mathcal{M } }$, the base form $\underline{\lambda}$ and the canonical fibre volume form $\underline{\mu}^{\mathcal{M } }$ satisfy the equations
\begin{equation}\label{adiabaticclosedmodulibundle}
d_f \underline{\omega}^{ \mathcal{M } }=0, \quad  d_{\nabla^\mathcal{M} } \underline{\omega}^{\mathcal{M }  }=0, \quad d_f \underline{\lambda}=0,
\end{equation}
\begin{equation}\label{adiabaticcoclosedmodlibundle}
d_{\nabla^\mathcal{M} }\underline{\mu}^{ \mathcal{M }  }=0, \quad d_f \underline{\Theta}^{\mathcal{M }  }=0, 
\end{equation}
\begin{equation}\label{maximalsubmanifoldmodulibundle}
d_{\nabla^\mathcal{M} } \underline{\Theta}^{ \mathcal{M }  }=0.
\end{equation}
\end{thm}
The same results hold if we replace the structure group by $PU(r)$, with cosmetic changes in the arguments. It suffices to say there is also a trace pairing on $\Lie PU(r)\simeq \Lie SU(r)$, and the numerical normalisation we use is the trace pairing for the fundamental representation of $\Lie SU(r)$.

\section{Mukai dual fibration}\label{Mukaidualfibrations}

Given a data set $\pi: M\to B$, with $ (\underline{\omega}, \underline{\lambda}, \underline{\Theta}, \underline{\mu}, H) $ defining Donaldson's adiabatic fibration , we wish to produce a dual fibration, by replacing every K3 fibre with a Mukai dual K3 fibre. For backgrounds on Mukai duality in the K3 setting, the reader is advised to consult the companion paper \cite{Mukaidualitypaper}.

More precisely, our setup is a parametrised version of the Mukai dual construction (\cf \cite{Mukaidualitypaper}). Namely, we take a Hermitian bundle over $E\to M$ with associated $PU(r)$ bundle $P\to M$, such that over each fibre K3 surface $X=M_b$, the Mukai vector $v$ of the restricted bundle $E|_b\to X$ has Mukai pairing $(v,v)=0$; thus the moduli space $X^\vee=M_b^\vee$ of irreducible HYM connections on $E|_b\to X$ (called the Mukai dual and assumed to be compact and nonempty) must be K3 surfaces, and when $b\in B$ varies they fit together topologically into the moduli bundle $M^\vee\to B$ (compare with Section \ref{Themodulibundle}). The universal family of ASD 
$PU(r)$-connections exists automatically on a $PU(r)$-bundle over the fibred product $M\times_B M^\vee$, and since $M\times_B M^\vee$ has no torsion cohomology, this lifts to  a Hermitian vector bundle $\mathcal{E}\to M\times_B M^\vee$, such that for each $b\in B$, its restriction $\mathcal{E}|_b\to X\times X^\vee$
is a universal family of irreducible  HYM connections on $E|_b\to X$.

In this setup, we show we can put the structure of Donaldson's adiabatic fibration on $M^\vee$,  compatible with the hyperk\"ahler structure on $M^\vee_b$. We shall call the resulting data the \textbf{Mukai dual fibration}. This is done from two perspectives: verifying Donaldson's maximal submanifold equation by a cohomological computation (\cf Section \ref{Dualityofmaximalsubmanifoldequations}), and the geometric construction of all the data from the moduli bundle interpretation (\cf Section \ref{DualityofDonaldsonadiabaticfibrations}). The cohomological approach compresses the data more efficiently, but the geometric approach is essential for comparisons of gauge theoretic and submanifold theoretic objects on the two fibrations.

The rest of the Chapter is devoted to the question of putting an optimal global connection on the universal bundle $\mathcal{E}\to M\times_B M^\vee$, whose restriction for each fixed $b\in B$ is just the given HYM connections on the fibres of $\mathcal{E}|_b\to M_b\times M_b^\vee$. This is a major ingredient in the Nahm transform treated in  Chapter \ref{Dualityforgaugetheories}. There are two main steps: constructing a (twisted) \textbf{triholomorphic connection} on each $\mathcal{E}|_b\to M_b\times M_b^\vee$, and then further extend these to a (twisted) \textbf{generalised adiabatic $G_2$ instanton} on $\mathcal{E}\to M\times_BM^\vee$, to be defined in the text. The first step is dealt with in the companion paper \cite{Mukaidualitypaper}, and the primary consequence of the existence of such triholomorphic connections is that $M$ can be regarded as the \textbf{double Mukai dual fibration} (\cf Section \ref{Theuniversalconnection1triholomorphicproperty}). The second step deals with a highly overdetermined equation, and the question of existence relies heavily on the \textbf{integrability} inherent in Donaldson's adiabatic fibrations. For technical convenience to do with the nontrivial centre in $U(r)$, we choose to work first with $PU(r)$ connections and finally lift to $U(r)$ connections.

\subsection{Duality of maximal submanifold equations}\label{Dualityofmaximalsubmanifoldequations}

We shall examine the possibility of dualisation by testing Donaldson's maximal submanifold equation for the positive section $h: B\to H^2(X)=H^2(K3)$ (\cf Section  \ref{Donaldsonadiabaticfibration}). This is deceptively simple. Recall from Section \ref{Donaldsonadiabaticfibration} that the hyperk\"ahler forms satisfy $[\omega_i]=\frac{\partial h}{\partial t_i}$. We define a map $h^\vee: B\to H^2(X^\vee,\R)$ by the slant product
\begin{equation}
h^\vee= \tilde{\mu}\circ h, \quad \tilde{\mu}(\alpha)=-\frac{1}{2r} p_1(ad (\mathcal{E}|_b))\cup \alpha /[X ], \forall \alpha\in H^2(X).
\end{equation}
Here $p_1$ means the first Pontrjagin class, and $\tilde{\mu}: H^2(X)\to H^2(X^\vee)$ is the higher rank generalisation of \textbf{Donaldson's $\mu$-map} in the context of 4-manifold polynomial invariants.
The hyperk\"ahler forms on a K3 and on the Mukai dual are related by $
[\omega_i^{X^\vee}]=\tilde{\mu}([\omega_i]), 
$
their volumes are equal: $\int_{X^\vee} [\omega_i^{X^\vee}]^2=\int_X [\omega_i]^2$, and $\tilde{\mu}$ is an isometry on the second real cohomology (\cf \cite{Mukaidualitypaper}). For rank 2 bundle case, this volume amounts to the well known polynomial invariants on K3 surfaces. From these facts, it is immediate that

\begin{prop}\label{Mukaidualpositivesection}
	If $h$ is the positive section encoding Donaldson's adiabatic fibration, then $h^\vee$ also satisfies the maximal submanifold equation. Morever $\frac{\partial h^\vee}{\partial t_i}=[\omega_i^{X^\vee}]\in H^2(X^\vee)$, so the positive section $h^\vee$ is compatible with the hyperk\"ahler structure on $M^\vee\to B$. 
\end{prop}

\begin{proof}
If we start from any positive section $h: B\to H^2(X)$, not necessarily maximal, then since $\tilde{\mu}$ is isometric, $h^\vee=\tilde{\mu}\circ h$ has the same area functional. Maximal submanifolds are the critical points of the area functional, hence if $h$ is a maximal submanifold so must $h^\vee$.
\end{proof}

The upshot is that by Donaldson's result (see the review in Section \ref{Donaldsonadiabaticfibration}), this dual positive section $h^\vee$ encodes the data of a dual adiabatic fibration on $\mathcal{M}\to B$, which we see is compatible with fibrewise hyperk\"ahler structures.

\subsection{Duality of Donaldson's adiabatic fibrations}\label{DualityofDonaldsonadiabaticfibrations}

We wish to enhance the cohomological understanding of the Mukai dual fibration to a geometric understanding. HYM connections induce $PU(r)$ ASD instantons, so we readily specialise the $PU(r)$ version of Theorem \ref{canonicalformsintegrability} to achieve

\begin{thm}\label{MukaidualfibrationisDonaldsonadiabaticfibration}
Under the setup of the introduction to this chapter, the bundle $\pi^\vee: M^\vee\to B   $ inherits the canonical 3-form $\underline{\omega}^\vee$, the canonical 4-form $\underline{\Theta}^\vee$, the base 3-form $\underline{\lambda}$ and the fibre volume form $\underline{\mu}^\vee$ from its interpretation as a moduli bundle, and these structures satisfy all the requirements of Donaldson's adiabatic fibration. Morever, the fibres of $\pi: M\to B$ and $\pi^\vee: M^\vee\to B$ have the same fibre volume 1.
\end{thm}

In other words, the structures on the moduli bundle give the geometric realisation of the Mukai dual fibration provided by Proposition \ref{Mukaidualpositivesection}. This is useful for comparing gauge theory/submanifolds on $M$ and $M^\vee$.

\begin{thm}
Adiabatic associative sections on the Mukai dual fibration are equivalent to adiabatic $G_2$ instantons on the principal $PU(r)$-bundle $P\to M$ up to gauge equivalence.
\end{thm}

\begin{proof}
We showed in Section \ref{Fueterequation$G_2$instantons} that adiabatic $G_2$ instantons are equivalent to the solutions of the Fueter equation on the moduli bundle; the modification to $PU(r)$ bundles is cosmetic. Here $M^\vee\to B$ is the moduli bundle with its canonical structures, and Proposition \ref{associativesectionisFueter} says adiabatic associative sections are the same as solutions to the Fueter equation, hence the result. 
\end{proof}

\subsection{The universal connection I: triholomorphic property}\label{Theuniversalconnection1triholomorphicproperty}

We aim to put an optimal connection on the global universal bundle $\mathcal{E}\to M\times_B M^\vee$ over the fibred product. We  focus first on the associated $PU(r)$ bundle and work on $X=M_b$ for any fixed $b\in B$, and recall from the companion paper \cite{Mukaidualitypaper} how to extend the tautological family of irreducible ASD $PU(r)$ connections on $X$ parametrised by $\tau\in X^\vee$, into a \textbf{triholomorphic} $PU(r)$ connection over $X\times X^\vee$, meaning its curvature has Dolbeault type $(1,1)$ for any complex structure in the hyperk\"ahler triple of $X\times X^\vee$. It may be commented that a more direct approach working  with $U(r)$ bundles has many topological subtleties coming from the non-discrete centre, which we shall circumvent.

There is a universal connection $\nabla^{univ}$ on the tautological infinte dimensional principal $PU(r)$ bundle $P|_{X} \times \mathcal{A}_b\to X\times \mathcal{A}_b$, where $\mathcal{A}_b$ is (the locus of irreducible connections inside) the affine space of $PU(r)$ connections on $P|_{X}\to X$, and $\mathcal{G}_b$ is the group of $PU(r)$ gauge transformations.  This $\nabla^{univ}$ can be thought of as a $u(r)$-valued 1-form $A^{univ}$ on $P|_{M_b} \times \mathcal{A}_b$. For $A\in \mathcal{A}_b$,
\begin{equation}
\begin{cases}
A^{univ}|_{P|_{M_b}\times\{A\}}= A, \\
A^{univ}(a)= G_A d_A^* a, \quad a\in T_A\mathcal{A}_b,
\end{cases}
\end{equation}
where the Green operator $G_A$  
is the inverse of the Laplacian $\Lap_A=d_A^*d_A$; morever $\nabla^{univ}$ descends to the quotient bundle $(P|_{X}\times \mathcal{A}_b   )/\mathcal{G}_b \to X \times \mathcal{A}_b/\mathcal{G}_b$.  The curvature $F(\nabla^{univ})$ at the point $(x,A)\in X\times \mathcal{A}_b$ is given by
\begin{equation}\label{universalconnectioncurvatureformula}
\begin{cases}
F( \nabla^{univ} ) (u_1, u_2)= F_A(u_1, u_2), \quad u_1, u_2\in T_x X, \\
F(  \nabla^{univ}  )( a, u   )=\langle a, u\rangle, \quad a\in T_A \mathcal{A}, u\in T_x X,\quad  d_A^*a=0, \\
F(    \nabla^{univ})(a_1, a_2)= -2 G_A \{ a_1, a_2 \}, \quad a_1, a_2 \in T_A \mathcal{A},\quad d_A^*a_1=d_A^*a_2=0.
\end{cases}
\end{equation}
where $\{ a_1, a_2\}$ means pointwise taking the Lie bracket on the bundle factor, and contracting the 1-form factor using the metric on $X$.

Now we pull back $\nabla^{univ}$ to $X^\vee$ along the tautological map $X^\vee \to \mathcal{A}_b/\mathcal{G}_b$ to obtain a connection, still denoted $\nabla^{univ}$, on the $PU(r)$ bundle associated to $\mathcal{E}|_b\to X\times X^\vee$.

\begin{lem}
(\cf \cite{Mukaidualitypaper}) The $PU(r)$ connection $\nabla^{univ}$ is triholomorphic.
\end{lem}

The triholomorphic property of the universal connection has implication on duality: it induces by restriction a family of ASD $PU(r)$ connection over $X^\vee$ parametrised by $X$. \emph{Suppose} that these are all \emph{irreducible} for all $x\in X$, then $X$ is naturally identified as the moduli space of ASD $PU(r)$ connections on $E'\simeq\mathcal{E}^\vee|_x\to X^\vee$, and the hyperk\"ahler structure on $X$ induced from this moduli interpretation agrees with the original hyperk\"ahler structure on $X$. (\cf \cite{Mukaidualitypaper}) Hence we can construct the Mukai dual fibration $M^{\vee\vee}\to B$ of $M^\vee \to B$, called the \textbf{double Mukai dual fibration}, whose underlying differentiable fibration is the same as $M\to B$.

\begin{prop}\label{doubleMukaidualfibration}
The double Mukai dual fibration is isomorphic to $M\to B$ as Donaldson's adiabatic fibrations.
\end{prop}

\begin{proof}
It is enough to check they define the same positive section $h^{\vee\vee}=h$. This follows from the fact that the fibrewise hyperk\"ahler periods on $M^{\vee\vee}\to B$ and $M\to B$ are equal.
\end{proof}

We have so far identified the desirable structure for the global connection on $\mathcal{E}$ restricted to fibres with fixed $b\in B$. In the next few Sections we wish to understand what happens when $b\in B$ varies.

\subsection{Generalised adiabatic $G_2$ instanton}\label{generalisedadiabtic$G_2$instantons}

We introduce a special structure called generalised adiabatic $G_2$ instantons, for applications to the Nahm transform treated in Chapter \ref{Dualityforgaugetheories}. A basic analogy is that generalised adiabatic $G_2$ instantons are related to adiabatic $G_2$ instantons, in the same way triholomorphic connections are related to ASD connections.

\begin{Def}
A $PU(r)$ \textbf{generalised  adiabatic $G_2$ instanton} is a connection on the associated $PU(r)$ bundle of $\mathcal{E}\to M\times_B M^\vee$, such that restricted to $M_b \times M_b ^\vee$ it is triholomorphic, and the horizontal-vertical type (1,1) part of its curvature $F$ satisfies
\begin{equation}\label{generalisedadiabatic$G_2$istantonequation}
\sum_i I_i \iota_{\frac{\partial}{\partial t_i } }  F^{(1,1)}=0.
\end{equation}
where $\frac{\partial}{\partial t_i }$ is an orthonormal basis of $B$, lifted to $M\times_B M^\vee$ using the canonical horizontal distribution on $M\times_B M^\vee\to B$, and the complex structures $I_i$ act on the 1-forms in the fibre direction. 
\end{Def}

We can motivate this condition by analogy with the adiabatic $G_2$ instantons $\alpha$, which as we recall amounts to the fibrewise ASD condition, and the condition
\[
\sum_i \omega_i \wedge  \iota_{\frac{\partial}{\partial t_i } }  F(\alpha)^{(1,1)}=0.
\]
This second condition is the same as 
\[
\sum_i I_i \iota_{\frac{\partial}{\partial t_i } }  F(\alpha)^{(1,1)}=0,
\]
where the equation holds on K3 fibres instead of higher dimensional fibres.
This definition in fact generalises to bundles with arbitrary hyperk\"ahler fibres.

\begin{Def}
A $U(r)$ \textbf{twisted generalised adiabatic $G_2$ instanton} is a $U(r)$ connection on $\mathcal{E}\to M\times_B M^\vee$, such that 
the (1,1) part of its curvature satisfies (\ref{generalisedadiabatic$G_2$istantonequation}), and when restricted to $M_b\times M_b^\vee$, its associated $PU(r)$ connection is triholomorphic and its central curvature satisfies
\begin{equation}\label{twistedgeneralisedadiabatic$G_2$instantoncentral}
\frac{\sqrt{-1}}{2\pi r} \Tr  F_A^{(0,2)}= \mathcal{B}+\mathcal{B}',
\end{equation}
where $\mathcal{B}$ is the harmonic representative of $\frac{1}{r}c_1(E)$, and $-\mathcal{B}'$ is the harmonic representative of $\frac{1}{r}c_1(E')$, where $E'$ is the underlying Hermitian bundle of $\mathcal{E}^\vee|_x \to M_b^\vee$    for any $x\in M_b$.
\end{Def}

This is the higher dimensional analogue of twisted adiabatic $G_2$ instantons, where the twisting is introduced likewise for compatibility with Chern class constraints.

\subsection{The universal connection II}\label{Theuniversalconnection2}

The rest of this Chapter will extend the fibrewise triholomorphic $PU(r)$ connection $\nabla^{univ}$ from Section \ref{Theuniversalconnection1triholomorphicproperty}, to a generalised adiabatic $G_2$ instanton  over $M\times_B M^\vee$, still to be denoted $\nabla^{univ}$. This amounts to solving a linear PDE (\ref{generalisedadiabatic$G_2$istantonequation}) for the unknown horizontal  covariant derivative operator
 $\nabla^{univ}_{\frac{\partial}{\partial t_k} }$, which clearly decouples along different fibres $M_b\times M^\vee_b$. 
The problem is overdetermined, but the integrability conditions for Donaldson's adiabatic fibrations will work in our favour to ensure the existence of solution.

To set up, denote $X=M_b$ with local coordinates $x_i$, and $X^\vee=M_b^\vee$ with local coordinates $\tau_i$. We shall work in the first order normal neighbourhood of $X\times X^\vee\subset M\times_B M^\vee$, where the horizontal distribution canonically trivialises the total space as a Cartesian product of $X\times X^\vee$ with the first order neighbourhood $b[\epsilon]$ of $b\in B$; the coordinates on $b[\epsilon]$ are denoted $t_i$. The effect of $F_H$ does not appear over the first order normal neighbourhood, and we can write Lie derivatives $\mathcal{L}_{\frac{\partial}{\partial t_k} }$ simply as $\frac{\partial}{\partial t_k}$.

Recall we are given
a family of $PU(r)$ ASD connections $A$ on $X$, 
parametrised by the product  $X^\vee\times b[\epsilon]$. We fix a topological identification of the underlying bundles with $P|_X\to X$, and write
\[
b_k=\frac{\partial A}{\partial t_k}, \quad \beta_j=\frac{\partial A}{\partial \tau_j} \in \Omega^1(X, ad(P|_{X})). 
\]
If we arrange the topological identification carefully, which amounts to a choice of gauge, we can impose the Coulumb gauge condition $d_A^* b_k=0$ along $X\times X^\vee$. The term $\beta_j$ arises from variation of ASD connections for the fixed $b\in B$, so it satisfies the linearised ASD equation over $X=M_b$. Write $\beta_j'$ as the $L^2$ projection of $\beta_j$ to $T_{A}X^\vee\subset \Omega^1(X, ad P|_{M_b})$.

There is a tautological map $f: X^\vee\times b[\epsilon]\to \mathcal{A}_b/\mathcal{G}_b$ just like in Section \ref{Theuniversalconnection1triholomorphicproperty}, which pulls back the universal connection to a $PU(r)$ connection $\nabla^{univ}$ over $X\times X^\vee\times b[\epsilon]$, and agrees with the triholomorphic connection over $X\times X^\vee$. Here the Green operator involved in the definition of $\nabla^{univ}$ uses the fixed hyperk\"ahler metric on the central fibre $X$. Although $\nabla^{univ}$ is only given on $X\times X^\vee\times b[\epsilon]$, 
the type $(0,2)$ and type $(1,1)$ components of the curvature $F(\nabla^{univ})$ are well defined along $X\times X^\vee$ since they involve at most one horizontal differentiation.

\begin{lem}
Some \textbf{curvature components} of $\nabla^{univ}$ are given by
\begin{equation}
\begin{cases}
F( \nabla^{univ} ) (\frac{\partial }{\partial t_i}, \frac{\partial }{\partial x_j})=\langle b_i, \frac{\partial }{\partial x_j}\rangle,
\\
F(    \nabla^{univ})(\frac{\partial }{\partial t_i}, \frac{\partial }{\partial \tau_j})= -2 G_A \{ b_i, \beta_j' \}.
\end{cases}
\end{equation}
\end{lem}
\begin{proof}
This is merely the pullback of the curvature formula (\ref{universalconnectioncurvatureformula}), since $b_i$ and $\beta_j'$ are both in the Coulumb gauge.
\end{proof}

Our next goal is to show

\begin{prop}\label{universalconnectiongeneralisedadiabatic$G_2$equationlocalversion}
(\textbf{Local existence})
The (1,1) type curvature component for the connection $\nabla^{univ}$ satisfies (\ref{generalisedadiabatic$G_2$istantonequation}) along $X\times X^\vee$.
\end{prop}

We prepare an elementary lemma about general connections on a hyperk\"ahler surface. It underlies the familiar fact that, the solution set of the linearised ASD equation with the Coulumb gauge condition is a quaternionic module.

\begin{lem}
On a hyperk\"ahler K3 surface, if $A$ is a connection, and $a$ is an adjoint bundle valued 1-form, then
\begin{equation}
d_A (I_i a  )\wedge \omega_j= 
\begin{cases}
(-d_A^* a )\underline{\mu}\quad &\text{if } i=j,\\
\sum_k \epsilon_{jik}d_A a \wedge \omega_k \quad &\text{if } i\neq j,
\end{cases}
\end{equation}
and
\begin{equation}
d_A^* (I_i a ) \underline{\mu}= d_A a \wedge \omega_i.
\end{equation}
\end{lem}

\begin{proof}
One uses that $I_1, I_2,I_3$ are parallel, to see that for $a=\sum_l a_l dx_l$ where $\{dx_l\}$ is a geodesic frame at the point of interest,
\[
\begin{cases}
d_A(I_i a) \wedge \omega_j= \sum_{k,l}   \nabla_k a_l dx_k\wedge I_idx_l \wedge \omega_j, \\
d_A^* a= - \sum_i \nabla_i a_i.
\end{cases}
\]
The rest is a mechanical and elementary linear algebraic calculation.
\end{proof}

The next lemma is where the \textbf{integrability conditions} for Donaldson's adiabatic fibration come in.

\begin{lem}
Assuming $\frac{\partial}{\partial t_i}$ is an orthonormal basis at $b\in B$, then for $b_k= \frac{\partial A}{\partial t_k}$,
\begin{equation}\label{universalconnectioncurvatureisadiabatic$G_2$instantonlocalversionkeylemma}
I_1 b_1+ I_2b_2+I_3 b_3=0.
\end{equation}
\end{lem}

\begin{proof}
We compute using the above lemma, and the Coulumb gauge condition, 
\[
\begin{split}
d_{A} (I_1 b_1+I_2b_2+I_3b_3    )\wedge \omega_1=& -(d_{A}^* b_1)\underline{\mu} + d_{A} b_2 \wedge \omega_3 -d_{A} b_3\wedge \omega_2\\
=& d_{A} b_2 \wedge \omega_3 -d_{A} b_3\wedge \omega_2.
\end{split}
\]
Now we take the variation of the ASD condition $F_{A} \wedge \omega_i=0$ when the fibre varies, to see
\[
d_{A} b_k \wedge \omega_i+ F_{A} \wedge \frac{\partial \omega_i}{\partial t_k}
=
\frac{\partial}{\partial t_k} ( F_{A} \wedge \omega_i     )=0.
\]
Thus by comparing the two equations, and then applying Donaldson's condition $d_H\underline{\omega}=0$,
\[
d_{A} (I_1 b_1+I_2b_2+I_3b_3    )\wedge \omega_1=F_{A} \wedge ( 
\frac{\partial \omega_2}{\partial t_3}-
\frac{\partial \omega_3}{\partial t_2}
)=0.
\]
Similarly, 
\[
d_{A} (I_1 b_1+I_2b_2+I_3b_3    )\wedge \omega_2=d_{A} (I_1 b_1+I_2b_2+I_3b_3    )\wedge \omega_3=0,
\]
and therefore the adjoint valued 1-form $I_1 b_1+I_2b_2+I_3 b_3$ satisfies the linearised ASD condition
\begin{equation}
d_{A}^+ ( I_1 b_1+I_2b_2+I_3 b_3       )=0.
\end{equation} 
We also have by an analogous argument
\begin{equation}
\begin{split}
\underline{\mu}d_{A}^*(I_1 b_1+I_2b_2+I_3b_3        )=& d_{A} b_1\wedge \omega_1+ d_{A} b_2\wedge \omega_2+ d_{A} b_3\wedge \omega_3 \\
=& -F_{A} \{  \frac{\partial \omega_1}{\partial t_1} +\frac{\partial \omega_2}{\partial t_2}+\frac{\partial \omega_3}{\partial t_3}         \}=0.
\end{split}
\end{equation}
Thus we see that $I_1 b_1+I_2b_2+ I_3 b_3$ is also in the Coulumb gauge. In short, 
\[
I_1 b_1+I_2b_2+ I_3 b_3\in T_{A}X^\vee\subset \Omega^1(X, ad P|_{M_b}).
\]

Finally, from the definition of the canonical horizontal distribution on  $M^\vee\to B$, by writing $\frac{\partial}{\partial t_k}$ as a horizontal vector, we are imposing that
$b_k=\frac{\partial A}{\partial t_k}$ has zero $L^2$  projection to 
the finite dimensional space $T_{A}X^\vee\subset \Omega^1(X, ad P|_{M_b})$.
By the quaternionic module property, so must $I_1 b_1+ I_2b_2+I_3 b_3$. This implies (\ref{universalconnectioncurvatureisadiabatic$G_2$instantonlocalversionkeylemma})
as required.
\end{proof}

We can now prove Proposition \ref{universalconnectiongeneralisedadiabatic$G_2$equationlocalversion}.

\begin{proof}
Equation (\ref{generalisedadiabatic$G_2$istantonequation}) boils down to two pieces
\[
\sum_i F( \nabla^{univ}  )( \frac{\partial}{\partial t_i},  I_i \frac{\partial}{\partial x_j}     )=0, 
\quad
\sum_i F( \nabla^{univ}  )( \frac{\partial}{\partial t_i},  I_i \frac{\partial}{\partial \tau_j}     )=0.
\]
Using the curvature formula (\ref{universalconnectioncurvatureformula}), these are equivalent to, respectively
\[
\sum_i \langle b_i,  I_i \frac{\partial}{\partial x_j}     \rangle=0,
\quad
G_{A}\sum_i \{ b_i,  I_i \beta_j     \}=0.
\]
The first equation is tautologically equivalent to (\ref{universalconnectioncurvatureisadiabatic$G_2$instantonlocalversionkeylemma}). The second is also implied by (\ref{universalconnectioncurvatureisadiabatic$G_2$instantonlocalversionkeylemma}), because $\{ b_i, I_i\beta_j  \}=-\{ I_ib_i, \beta_j \}$ at every point on $X$.
\end{proof}

We can perform this construction to define the horizontal directional derivatives $\nabla^{univ}_{\frac{\partial}{\partial t_k}}$ for all $b\in B$, so that 
 (\ref{generalisedadiabatic$G_2$istantonequation}) is solved fibrewise:

\begin{prop}
The fibrewise triholomorphic connection can be extended to a $PU(r)$ generalised adiabatic $G_2$ connection $\nabla^{univ}$ over $M\times_B M^\vee$.
\end{prop}

\subsection{Twisted generalised adiabatic $G_2$ instantons}\label{Technicalmodificationsnonzeroc1}

We now lift the $PU(r)$ generalised adiabatic $G_2$ instanton to $U(r)$ connection on $\mathcal{E}\to M\times_B M^\vee$. This requires additionally  specifying a $U(1)$ connection on $\Lambda^r \mathcal{E}$. Since $B$ is contractible, the line bundle $\Lambda^r \mathcal{E}\to M\times_B M^\vee$ is topologically the tensor product of line bundles $L\to M$ and $L'\to M^\vee$.

\begin{lem}
There is a $U(1)$ twisted adiabatic $G_2$ instanton on $L\to M$, and similarly for $L'\to M^\vee$.
\end{lem}

\begin{proof}
(Sketch) Let $\mathcal{B}$ be the type (0,2) form which fibrewise is the harmonic representative of $c_1(L)$. Choose a connection $A$ on $L\to M$, such that its restriction to K3 fibres has curvature $2\pi \sqrt{-1}\mathcal{B}$. Using the proof of Lemma \ref{Fueterimpliesmonopole} and the fact that $H^1(K3)=0$, we can adjust $A$ such that $(A, \Phi_0)$ is a `twisted adiabatic $G_2$ monopole' similar to (\ref{adiabaticmonopole11typecurvature}),  for some Higgs field $\Phi_0$:
\begin{equation}\label{twistedadiabatic$G_2$monopoleistwistedadiabaticinsttanton}
F_A \wedge \underline{\Theta}= (*_4 d_A \Phi_0) \wedge \underline{\lambda} +2\pi \sqrt{-1} \mathcal{B}\wedge \underline{\Theta}.
\end{equation}
Using $d\underline{\Theta}=0$, $d_H \underline{\mu}=0$ and  $\mathcal{B}\wedge \omega_i= ( \int_X c_1(L)\wedge \omega_i  )\underline{\mu}$, we observe
\[
\begin{split}
& d\mathcal{B}\wedge \underline{\Theta}= d\{ \mathcal{B}\wedge  \underline{\Theta}   \}
=d\{  -\sum_{cyc}(  \int_{M_b} \mathcal{B}\wedge \omega_i       )  dt_jdt_k  \wedge \underline{\mu}    \}  \\
& =d\{  -\sum_{cyc} (\int_{M_b} c_1(L)\wedge \omega_i       )  dt_jdt_k \wedge \underline{\mu}         \} 
=d\{ ( \int_{M_b} c_1(L)\wedge \underline{\Theta}      )  \wedge \underline{\mu}         \} \\
& =  (\int_{M_b} c_1(L)\wedge d \underline{\Theta}      )  \wedge \underline{\mu}  +  (\int_{M_b} c_1(L)\wedge  \underline{\Theta}      )  \wedge d\underline{\mu}   =0.   
\end{split}
\]
Using this fact, we take $d_A$ of (\ref{twistedadiabatic$G_2$monopoleistwistedadiabaticinsttanton}) and argue as in Proposition \ref{Adiabatic$G_2$monopoleadiabatic$G_2$instanton} to show $d_A\Phi_0=0$ on K3 fibres, so $A$ is a twisted adiabatic $G_2$ instanton.
\end{proof}

\begin{thm}\label{generalisedadiabatic$G_2$instantonexistencequestion}
In the setup of the introduction to this Chapter, there is a twisted generalised adiabatic $G_2$ instanton on $\mathcal{E}\to M\times_B M^\vee$, still denoted $\nabla^{univ}$, which restricts to the tautological HYM instantons over $M_b$ for $b\in B$.
\end{thm}

\begin{proof}
We can prescribe the tensor product connection on $\Lambda^r\mathcal{E}\simeq L\otimes L'$, which is a $U(1)$ twisted generalised adiabatic $G_2$ instanton. This data allows us to lift the $PU(r)$ generalised $G_2$ instanton to a $U(r)$ twisted generalised $G_2$ instanton.
\end{proof}

\begin{rmk}
The topological identification $\Lambda^r\mathcal{E}\simeq L\otimes L'$ causes non-uniqueness in the construction: one can twist $\nabla^{univ}$ by adding an exact $u(1)$-valued 1-form pulled back from $M^\vee$, or add a $u(1)$-valued 1-form pulled back from $B$, to obtain different solutions.
\end{rmk}

\section{Duality for gauge theories}\label{Dualityforgaugetheories}

The twisted generalised $G_2$ instanton $\nabla^{univ}$ on the universal bundle $\mathcal{E}\to M\times_B M^\vee$ allows us to define the \textbf{Nahm transform}, which transforms twisted $G_2$ instantons on a Hermitian bundle $\mathcal{F}\to M$ with a certain slope potential (\cf Definition \ref{slopepotentialdef}), to  twisted $G_2$ instantons on a new Hermitian bundle $\hat{\mathcal{F}}\to M^\vee$ with a certain slope potential. The main ingredients aside from the existence theory of $\nabla^{univ}$, are the Nahm transform on K3 surfaces 
(\cf the companion paper \cite{Mukaidualitypaper}, which uses core ideas from \cite{BraamBaal}), and the information on the variation of the Dirac operator in Chapter \ref{Adiabaticspinstructures}.

With more input from the companion paper \cite{Mukaidualitypaper}, and using the instanton-Fueter correspondence in Section \ref{Fueterequation$G_2$instantons}, we show, under some regularity assumptions, a version of the \textbf{Fourier inversion theorem}, which says that the inverse Nahm transform is isomorphic to the given twisted adiabatic $G_2$ instanton. This is another manifestation of the duality between $M$ and $M^\vee$.

\subsection{The Nahm transform}\label{TheNahmtransform}

Assume the setup of Chapter \ref{Mukaidualfibrations}.
Let $\pi:M\to B$ be a Donaldson's adiabatic fibration, with Mukai dual fibration $\pi^\vee: M^\vee\to B$, and there is a twisted generalised adiabatic $G_2$ instanton $\nabla^{univ}$ on the universal bundle $\mathcal{E}\to M\times_B M^\vee$.

\begin{rmk}
For a more familiar analogy, $\nabla^{univ}$ shall play the same role as the universal connection on the Poincar\'e line bundle for the Nahm transform on an Abelian surface \cite{BraamBaal}. In less precise terms,  $\nabla^{univ}$ behaves like the Schwartz kernel of Fourier transform operator, or algebraic cycles acting as correspondences.
\end{rmk}

Suppose we are given a Hermitian vector bundle $\mathcal{F}\to M$ with a connection $\alpha$, whose restriction to each fibre is an irreducible HYM connection. Restricted to any K3 fibre $X=M_b$, we assume $\mathcal{F}|_b\to X$ has a different Mukai vector compared to the Hermitian bundle $E|_b\simeq \mathcal{E}|_\tau\to X$ for any $\tau\in X^\vee$. (If the Mukai vectors are the same, then $\alpha$ is represented by a section of $\pi^\vee: M^\vee\to B$, which is the situation of Section \ref{Fueterequation$G_2$instantons}.) Assume also that the gradient of the slope potential function of the bundles $\mathcal{F}\to M$ and $E\to M$ are the same; this ensures that for any fixed $b\in B$, $\tau\in X^\vee$, the connection $\alpha_\tau$ on $\Hom(\mathcal{E}|_\tau, \mathcal{F}|_b)$ obtained by coupling $\alpha$ to $\nabla^{univ}$ has no central curvature, so is actually ASD.

Then for any fixed $b\in B$, we can perform the \textbf{Nahm transform on the K3 fibres}, treated in the companion paper \cite{Mukaidualitypaper}. We briefly recall how this is defined. We couple the Dirac operator on the negative spinor bundle $S^-_X$ (\cf Section \ref{Curvatureoperators}) to the connection $\alpha_\tau$ parametrised by $\tau \in X^\vee$,
\begin{equation}
D^-_{\alpha_\tau} : \Gamma(X , S^-_X\otimes \Hom(\mathcal{E}|_{\tau } , \mathcal{F}|_{b}) ) \to \Gamma(X , S^+_X\otimes\Hom(\mathcal{E}|_{\tau} , \mathcal{F}|_{b}) ).
\end{equation}
The formal adjoint operator is the coupled Dirac operator $D^+_{\alpha_\tau}$.
A Bochner formula argument shows that $D^+_{\alpha_\tau}$ has vanishing kernel, using the assumption that $\Hom(\mathcal{E}|_\tau , \mathcal{F}|_b)$ has no global covariantly constant section on $X$, which is a consequence of our setup. Thus the kernels of $D^-_{\alpha_\tau}$ fit together into a vector bundle $\hat{\mathcal{F}}|_b$ over $X^\vee$, equipped with a natural connection $\hat{\alpha}|_b$ which turns out to be HYM, with the same slope as the bundle  $E'|_b\simeq \mathcal{E}^\vee|_x\to X^\vee$ for any $x\in X$. The Mukai vector of $\hat{\mathcal{F}}|_b$ is the Fourier-Mukai transform of the Mukai vector of $\mathcal{F}|_b$, and in particular $\hat{\mathcal{F}}|_b\not \simeq E'|_b\to X^\vee$. (\cf \cite{Mukaidualitypaper} for more details).

Now we turn to the problem on the 7-fold and allow $b\in B$ to vary. The cokernel vanishing property guarantees these bundles $\hat{\mathcal{F}}|_b$ to fit into a smooth vector bundle $\hat{\mathcal{F}}\to M^\vee$. We can put a canonical connection $\hat{\alpha}$ on $\hat{\mathcal{F}}$ as follows. 
We think of the vector bundle $\hat{\mathcal{F}}|_b=\ker D^-_{\alpha_\tau}$ as a subbundle of the infinite rank vector bundle $\hat{\mathcal{H} }\to M^\vee,
$
whose fibre at the point $(b, \tau)\in M^\vee$ is \[ \hat{\mathcal{H} }|_{b, \tau}  = \Gamma(X , S^-_X\otimes \Hom(\mathcal{E}|_{\tau } , \mathcal{F}|_{b}) ) .\] This bundle $\hat{\mathcal{H}}$ has a natural tautological connection $\hat{d}$, induced from the given connections on $S^-_X\to M$, $\mathcal{F}\to M$, and $\mathcal{E}\to M\times_B M^\vee$, because a first order trivialisation of all the ingredients $S^-_X$, $\mathcal{F}$, $\mathcal{E}$  around a given copy of $X\subset M\times_B M^\vee$ implies a first order trivialisation on $\hat{\mathcal{H} }$. The reader should bear in mind that $\hat{d}$ is a connection over $M^\vee$, so involves differentiation in the $X^\vee$ and the horizontal direction, but not in the $X$ direction.

The subbundle $\hat{\mathcal{F}}\to M^\vee$ is then equipped with a Hermitian metric and the subbundle connection $\hat{\alpha}$. More concretely,  denote by $G_\tau$ the Green operator acting on $\Gamma(X, S^+_X \otimes \Hom(\mathcal{E}|_{\tau} , \mathcal{F}|_{b})   ) $, \ie the inverse to the Laplace operator $\Lap_{\alpha_\tau}=\nabla^*_{\alpha_\tau} \nabla_{\alpha_\tau}$, which by virtue of the ASD property of $\alpha_\tau$ is 
\begin{equation}\label{Lichnerowiczformula}
\Lap_{\alpha_\tau}= D^-_{\alpha_\tau }D^+_{ \alpha_\tau} .
\end{equation}
Thus the $L^2$ orthogonal projection operator from $\hat{\mathcal{H}}|_{b,\tau}$ to the subspace $\hat{\mathcal{F}}|_b$ is given by $
P=1-D^+_{\alpha_\tau} G_\tau D^-_{\alpha_\tau}, 
$
and the covariant derivative associated to $\hat\alpha$ is $\hat{\nabla}= P\hat{d}$.

\begin{Def}
The pair of the Hermitian bundle $\hat{\mathcal{F}}\to M^\vee$ with the canonical connection $\hat{\alpha}$, is called the \textbf{Nahm transform} of the pair $(\mathcal{F}, \alpha)$.
\end{Def}

Our setup here is compatible with the K3 version of the Nahm transform when restricted to fibres. This gives

\begin{lem}(\cf \cite{Mukaidualitypaper})\label{Nahmtransformpreservesadiabatic$G_2$instantononK3}
The Nahm transform $(\hat{\mathcal{F}},  \hat{\alpha} )$ restricted to each K3 fibre on $M^\vee$ is HYM,  and the gradient of the slope potentials of $\hat{\mathcal{F}}\to M^\vee$ and $E'\to M^\vee$ are the same.
\end{lem}

Our next goal is to compute the curvature matrix of  $\hat{\alpha}$, which is very similar to the K3 setting in \cite{Mukaidualitypaper}, and in essence  goes back to the paper \cite{BraamBaal} for Nahm transforms on Abelian surfaces. We reproduce it here for the reader's benefit. The summation convention will be used.

Let $\hat{f}^j  \in \Gamma( X, S^-_X \otimes \Hom( \mathcal{E}|_{\tau}, \mathcal{F}|_b   ) ) $ with $\tau\in X^\vee, b\in B$ be a local orthonormal framing of $\hat{\mathcal{F}}\to M^\vee$,  where $j=1,2\ldots , rk(\hat{\mathcal{F}})$. For a section $\hat{s}(\tau,b)= \sum_j \hat{s}_j \hat{f}^j$, with $\hat{s}_j$ being local  $C^\infty$ functions on $M^\vee$, one can compute \[
\hat{\nabla} \hat{s}
= P\hat{d} \hat{s} =(1-D_{\alpha_\tau}^+ G_\tau D^-_{\alpha_\tau}   ) [ \hat{d} \sum \hat{s}_j (\tau,b)\hat{f}^j  ].
\]
In components, we can write \[
\hat{\nabla} \hat{s}= (d \hat{s}_j + \hat{\alpha}_{jk} \hat{s}_k       ) \hat{f}^j,
\]
where the connection matrix of $\alpha$   is 
$
\hat{\alpha}_{jk}= \langle \hat{f}^j, \hat{d} \hat{f}^k\rangle.
$
The curvature matrix is
\[
\hat{F}_{ij}= d\hat{\alpha}_{ij} +\hat{\alpha}_{ik} \wedge \hat{\alpha}_{kj}=
\langle \hat{d}\hat{f}^i, \wedge \hat{d} \hat{f}^j   \rangle+ \langle \hat{f}^i, \wedge \hat{d}^2\hat{f}^j   \rangle
+ 
\langle\hat{f}^i, \hat{d}\hat{f}^k   \rangle \wedge \langle\hat{f}^k, \hat{d}\hat{f}^j   \rangle.
\]
Now $\langle \hat{d} \hat{f}^i, \hat{f}^k   \rangle=-\langle\hat{f}^i, \hat{d}\hat{f}^k   \rangle$ by the compatibility with Hermitian structures, so the third term above is recognized as $-\langle P \hat{d} \hat{f}^i,\wedge \hat{d}\hat{f}^j   \rangle$, and
\[
\begin{split}
\hat{F}_{ij}= & \langle \hat{d}\hat{f}^i, \wedge \hat{d}\hat{f}^j   \rangle- 
\langle P \hat{d} \hat{f}^i,\wedge \hat{d}\hat{f}^j   \rangle + \langle \hat{f}^i, \wedge \hat{d}^2\hat{f}^j   \rangle\\
= & \langle  D_{\alpha_\tau}^+ G_\tau D^-_{\alpha_\tau}      \hat{d}\hat{f}^i, \wedge \hat{d}\hat{f}^j   \rangle
= \langle   G_\tau D^-_{\alpha_\tau}      \hat{d}\hat{f}^i, \wedge D_{\alpha_\tau}^- \hat{d} \hat{f}^j  \rangle+ \langle \hat{f}^i, \wedge \hat{d}^2\hat{f}^j   \rangle.
\end{split}
\]
But since $ D^-_{\alpha_\tau}\hat{f}^i=0$,
\[ D^-_{\alpha_\tau}      \hat{d}\hat{f}^i=[  D^-_{\alpha_\tau},       \hat{d}   ]\hat{f}^i.
\]
To summarize,
\begin{lem}
The \textbf{curvature matrix} of $\hat{\alpha}$ is 
\begin{equation}\label{curvaturematrixNahmtransform}
\hat{F}_{ij}=\langle   G_\tau [D^-_{\alpha_\tau} ,     \hat{d}]\hat{f}^i, \wedge [D_{\alpha_\tau}^-, \hat{d}] \hat{f}^j  \rangle+ \langle \hat{f}^i, \wedge \hat{d}^2\hat{f}^j   \rangle.
\end{equation}
\end{lem}

This computation is formal in nature. The actual geometric content is to interpret the curvature term $\hat{d}^2$ and the variation of the coupled fibrewise Dirac operator $[D_{\alpha_\tau}^-, \hat{d}]$. We shall denote as in Chapter \ref{Mukaidualfibrations} the coordinates on $X$ as $x_i$, the coordinates on $X^\vee$ as $\tau_i$, and the base coordinates in the first order neighbourhood of $b\in B$ as $t_i$.
We will be interested mainly in the terms which contribute to the horizontal-vertical type $(1,1)$ component of the curvature $\hat{F}$, namely $\sum_{i,j}\hat{F}(\frac{\partial}{\partial \tau_i} , \frac{\partial}{\partial t_j} )d\tau_i \wedge dt_j$.

\begin{rmk}
Even though
the horizontal distribution $H$ is not integrable in general, we can always choose coordinates over the first order neighbourhood such that $\frac{\partial}{\partial t_i}$ gives $H$.
\end{rmk}

We first analyse $\hat{d}^2$. The tautological connection $\hat{d}$ on the infinite rank bundle $\hat{\mathcal{H}}$ is induced from the natural connections on $S^-_X\to M$, $\mathcal{E}\to M\times_B M^\vee$, and $\mathcal{F}\to M$, and so its curvature is tautologically induced by the curvature of these 3 individual connections. If $s$ is a section in $\Gamma(X, S^-_X\otimes \Hom(\mathcal{E}|_\tau, \mathcal{F}|_b))$ representing an element in $\hat{\mathcal{H}}_{\tau,b} $, then $\hat{d}^2$ acting on $s$ just means the sum of the following 3  curvature operators acting on $s$ pointwise on $X$.

\begin{itemize}
\item The connection on $S^-_X\to M$ is $\nabla^{LC,-}$ from Chapter \ref{Adiabaticspinstructures}. Since this is pulled back from $M$ to $M\times_B M^\vee$,  the curvature component  $F(\nabla^{LC,-})(\frac{\partial}{\partial \tau_i} ,\frac{\partial}{\partial t_j}         )=0$ does not contribute to $\hat{d}^2$.  

\item The connection on $\mathcal{E}$ was denoted $\nabla^{univ}$, and its relevant curvature component is $F(\nabla^{univ})( \frac{\partial}{\partial \tau_i} ,\frac{\partial}{\partial t_j}       )d\tau_i \wedge dt_j$. Since this is acting on a Hom bundle, in terms of matrices we need to take the minus transpose action.

\item The connection on $\mathcal{F}$ is $\alpha$, and its contribution is zero similar to the $S^-_X$ case.
\end{itemize}

Combining these remarks,
\begin{lem}
The term $\langle \hat{f}^i, \wedge \hat{d}^2\hat{f}^j   \rangle$ is equal to
\begin{equation}\label{contributiontoNahmtransformcurvaturehatdsquare}
\langle \hat{f}^i, \{ \sum_{k,l} F(\nabla^{univ})( \frac{\partial}{\partial \tau_k} ,\frac{\partial}{\partial t_l}       )d\tau_k \wedge dt_l  \}\hat{f}^j  \rangle.
\end{equation}
\end{lem}

Next we analyse the variation of the coupled fibrewise Dirac operator 
$[ D_{\alpha_\tau}^-, \hat{d} ]$, which in components has the formula
\[
[\hat{d}, D_{\alpha_\tau}^-]=\sum [ \nabla_{\frac{\partial}{\partial t_i}}, D^-_{\alpha_\tau}    ] dt_i +\sum [ \nabla_{\frac{\partial}{\partial \tau_i}}, D^-_{\alpha_\tau}    ] d\tau_i.
\]
Here $\nabla$ refers to the natural covariant derivatives of $\hat{d}$, so really comes from the 3 individual connections as above. Those commutators are operators acting on coupled negative spinors.

\begin{itemize}
\item The most fundamental contribution is when we ignore the coupling with $\alpha$ and $\nabla^{univ}$. Then we are dealing with the variation of the fibrewise uncoupled Dirac operator
$\sum[ \nabla^{LC}_{ \frac{\partial}{\partial t_i}  } , D^-      ] dt_i$
, which is the subject of the delicate calculations in Section \ref{Curvatureoperators}.

\item The presence of the connection $\nabla^{univ}$ has two effects. Let us write
\[
\begin{cases}
\Omega= \sum F(\nabla^{univ})( \frac{\partial}{\partial x_i},   \frac{\partial}{\partial t_j}  )dx_i \wedge dt_j, 
\\
 \Omega'=\sum F(\nabla^{univ})( \frac{\partial}{\partial x_i},   \frac{\partial}{\partial \tau_j}  )dx_i \wedge d\tau_j.
\end{cases}
\]
The Clifford actions of these curvature tensors are
\[
\begin{cases}
\Omega\cdot s=\sum F(\nabla^{univ})( \frac{\partial}{\partial x_i},   \frac{\partial}{\partial t_j}  )c_X(dx_i) s \otimes dt_j, \\
\Omega'\cdot s=\sum F(\nabla^{univ})( \frac{\partial}{\partial x_i},   \frac{\partial}{\partial \tau_j}  )c_X(dx_i)s \otimes d\tau_j.
\end{cases}
\]
The contribution to the variation of the coupled fibrewise Dirac operator $[D^-_{\alpha_\tau} ,     \hat{d}]$ is $\Omega+\Omega'$, which are Clifford operator valued 1-forms. We remark that the action of this operator on the $\Hom(\mathcal{E},\mathcal{F})$ factor of $s$ involves minus transpose in matrix language.

\item 
The connection $\alpha$ has curvature matrix $F(\alpha)$, and its horizontal-vertical (1,1) part is 
\[
F(\alpha)^{(1,1)}=\sum F(\alpha)( \frac{\partial}{\partial x_i},   \frac{\partial}{\partial t_j}  )dx_i \wedge dt_j.
\]
The Clifford action of this curvature operator is
\[
F(\alpha)^{(1,1)}\cdot s=\sum F(\alpha)^{(1,1)}( \frac{\partial}{\partial x_i},   \frac{\partial}{\partial t_j}  )c_X(dx_i) s \otimes dt_j.
\]
The contribution to $[ D^-_{\alpha_\tau} ,     \hat{d} ]$ is $F(\alpha)^{(1,1)}\cdot$ viewed as the Clifford operator valued 1-form.
\end{itemize}

Combining these contributions, and collecting $dt_i$ terms together,
\begin{lem}
The  (1,1) type component of $
\langle   G_\tau [D^-_{\alpha_\tau} ,     \hat{d}]\hat{f}^i, \wedge [D_{\alpha_\tau}^-, \hat{d}] \hat{f}^j  \rangle$
can be written as the sum of
\[
\langle   G_\tau ( [D^- ,     \hat{d}]+ \Omega+F(\alpha)^{(1,1)})\cdot\hat{f}^i, \wedge   \Omega'\cdot
 \hat{f}^j  \rangle
\]
and
\[
\langle   G_\tau \Omega'\cdot\hat{f}^i, \wedge  ( [D^- ,     \hat{d}]+ \Omega+F(\alpha)^{(1,1)})\cdot
\hat{f}^j  \rangle.
\]
\end{lem}

\subsection{Nahm transform preserves adiabatic $G_2$ instantons}\label{Nahmtransformpreservesadiabatic$G_2$instantons}

Assume in this Section that $\alpha$ on the bundle $\mathcal{F}\to M$ is a twisted adiabatic $G_2$ instanton.
Consider $\hat{F}^{(1,1)}\wedge \underline{\Theta}^{M^\vee}$, or equivalently 
\[
(I_1 dt_2 dt_3+ I_2 dt_3 dt_1+I_3 dt_1dt_2) \hat{F}^{(1,1)}
\]
where $I_k$ acts on the 1-forms $d\tau_j$ but not on $dt_j$. The contributing terms  from last Section will cancel out exactly.

\begin{lem}\label{contributioncurvatureofNahmtransformhatdsquaredterm1}
The contributing term $\sum F(\nabla^{univ})( \frac{\partial}{\partial \tau_i} ,\frac{\partial}{\partial t_j}       )d\tau_i \wedge dt_j$ to $\hat{d}^2$ satisfies
\[
(I_1 dt_2 dt_3+ I_2 dt_3 dt_1+I_3 dt_1dt_2)\{\sum F(\nabla^{univ})( \frac{\partial}{\partial \tau_i} ,\frac{\partial}{\partial t_j}       )d\tau_i \wedge dt_j \}=0.
\]
Therefore the contribution (\ref{contributiontoNahmtransformcurvaturehatdsquare}) vanishes.
\end{lem}

\begin{proof}
This can be rephrased as one of the components of the twisted generalised adiabatic $G_2$ instanton equation
$
\sum_k I_k \iota_{\frac{\partial}{\partial t_k} } F(\nabla^{univ})^{(1,1)} =0.
$
\end{proof}

\begin{lem}
The operator valued 1-form $[D^- ,     \hat{d}]+ \Omega\cdot+ F(\alpha)^{(1,1)} \cdot$ satisfies
\begin{equation}\label{variationofcoupledDiracoperator}
(I_1^{S^+}dt_2dt_3+I_2^{S^+}dt_3dt_1+I_3^{S^+}dt_1dt_2 ) \{ [D^- ,     \hat{d}]+ \Omega \cdot   + F(\alpha)^{(1,1)}\cdot    \}=0.
\end{equation}
\end{lem}

\begin{proof}
If we ignore the $\Omega$ and the $F(\alpha)^{(1,1)}$ term, then this is the content of equation (\ref{variationofuncoupledDiracoperator$G_2$instantonequation}), which is the result of a very delicate calculation.

As for the contribution of $\Omega$,
\[
\begin{split}
& (I_1^{S^+}dt_2dt_3+I_2^{S^+}dt_3dt_1+I_3^{S^+}dt_1dt_2 )\Omega \cdot \\
=& (I_1^{S^+}dt_2dt_3+I_2^{S^+}dt_3dt_1+I_3^{S^+}dt_1dt_2 )\{ \sum \Omega( \frac{\partial }{\partial x_i},   \frac{\partial }{\partial t_j} ) c_X(dx_i) dt_j \}\\
=& dt_1dt_2dt_3\{      
\sum \Omega( \frac{\partial }{\partial x_i},   \frac{\partial }{\partial t_j} ) c_X(I_j dx_i) 
           \}\\
=& c_X \{ (I_1dt_2dt_3+I_2dt_3dt_1+I_3dt_1dt_2 )
\Omega\}=0.
\end{split}
\]
This last step is a component of the twisted generalised adiabatic $G_2$ instanton equation. Here the Clifford action is independent of wedging by $dt_i$.

The discussion for $F(\alpha)^{(1,1)}$ is exactly analogous to $\Omega$.
\end{proof}

\begin{lem}
The Clifford operator valued 1-form $\Omega'$ acting on a coupled negative spinor $s$ satisfies
\begin{equation}\label{Omega'istriholomorphicspinorinterpretation}
\sum_{i,j} \Omega'( \frac{\partial}{\partial x_i} ,   \frac{\partial}{\partial \tau_j}      ) c_X ( dx_i )s \otimes I_k d\tau_j= -I_k^{S^+} \Omega'\cdot s.
\end{equation}
in the local bases $\{ \frac{\partial}{\partial x_i} \}$ and  $\{ \frac{\partial}{\partial \tau_j} \}$.
\end{lem}

\begin{proof}
We compute using the triholomorphic property of $\Omega'$ (which is part of the definition of twisted generalised adiabatic $G_2$ instantons)
\[
\begin{split}
&\sum_{i,j} \Omega'( \frac{\partial}{\partial x_i} ,   \frac{\partial}{\partial \tau_j}      ) c_X ( dx_i ) \otimes I_k d\tau_j 
=\sum_{i,j} \Omega'( \frac{\partial}{\partial x_i} ,  -I_k \frac{\partial}{\partial \tau_j}      ) c_X ( dx_i ) \otimes  d\tau_j \\
=& \sum_{i,j} \Omega'( I_k \frac{\partial}{\partial x_i} ,   \frac{\partial}{\partial \tau_j}      ) c_X ( dx_i ) \otimes  d\tau_j 
=\sum_{i,j} \Omega'(  \frac{\partial}{\partial x_i} ,   \frac{\partial}{\partial \tau_j}      ) c_X ( -I_kdx_i ) \otimes  d\tau_j \\
=&  -I_k^{S^+}\sum_{i,j} \Omega'(  \frac{\partial}{\partial x_i} ,   \frac{\partial}{\partial \tau_j}      ) c_X ( dx_i ) \otimes  d\tau_j.
\end{split}
\]
This acting on the coupled negative spinor is the claim.
\end{proof}

We will now come to the crux of the \textbf{miraculous cancellation}.

\begin{prop}\label{contributioncurvatureofNahmtransformvariationofDiracoperator}
The term $\langle   G_\tau ( [D^- ,     \hat{d}]+ \Omega+ F(\alpha)^{(1,1)})\cdot\hat{f}^i, \wedge   \Omega'\cdot
\hat{f}^j  \rangle$ satisfies
\begin{equation}\label{contributioncurvatureofNahmtransformvariationofDiracoperatorequation}
(I_1dt_2dt_3+I_2dt_3dt_1+I_3dt_1dt_2 )\langle   G_\tau ( [D^- ,     \hat{d}]+ \Omega+F(\alpha)^{(1,1)})\cdot\hat{f}^i, \wedge   \Omega'\cdot
\hat{f}^j  \rangle=0.
\end{equation}
Completely analogous result holds for \[\langle   G_\tau \Omega'\cdot\hat{f}^i, \wedge  ( [D^- ,     \hat{d}]+ \Omega+F(\alpha)^{(1,1)})\cdot
\hat{f}^j  \rangle.\]
\end{prop}

\begin{proof}
The complex structures are acting on the $d\tau_i$ factors here. So by (\ref{Omega'istriholomorphicspinorinterpretation}),
\[
\begin{split}
& I_k \langle   G_\tau ( [D^- ,     \hat{d}]+ \Omega+F(\alpha)^{(1,1)})\cdot\hat{f}^i, \wedge   \Omega'\cdot
\hat{f}^j  \rangle \\
=& \langle   G_\tau ( [D^- ,     \hat{d}]+ \Omega+F(\alpha)^{(1,1)})\cdot\hat{f}^i, \wedge 
(\sum_{p,q} \Omega'(  \frac{\partial}{\partial x_p} ,   \frac{\partial}{\partial \tau_q}      ) c_X ( dx_p ) \otimes I_k d\tau_q   )
\cdot
\hat{f}^j  \rangle   
\\
=&  -\langle   G_\tau ( [D^- ,     \hat{d}]+ \Omega+F(\alpha)^{(1,1)})\cdot\hat{f}^i, \wedge I_k^{S^+}
\Omega'
\cdot
\hat{f}^j  \rangle  .
\end{split}
\]
The operators $I_k^{S^+}$ on positive spinors are anti-self-adjoint, so the above is
\[
\langle I_k^{S^+}  G_\tau ( [D^- ,     \hat{d}]+ \Omega+ F(\alpha)^{(1,1)})\cdot\hat{f}^i, \wedge 
\Omega'
\cdot
\hat{f}^j  \rangle.
\]
Now crucially, since $\alpha_\tau$ is ASD, the Lichnerowicz formula (\ref{Lichnerowiczformula}) and the hyperk\"ahler nature of the K3 fibres imply that the square of the Dirac operator commutes with the natural complex structure actions $I_k^{S^+}$, and therefore the Green operator $G_\tau$ commutes with $I_k^{S^+}$. We remark that this observation is also the core of our proof in \cite{Mukaidualitypaper} that the Nahm transform on K3 surfaces preserves the HYM condition.

Thus we convert the above to
\[
\langle   G_\tau I_k^{S^+} ( [D^- ,     \hat{d}]+ \Omega+F(\alpha)^{(1,1)})\cdot\hat{f}^i, \wedge 
\Omega'
\cdot
\hat{f}^j  \rangle.
\]
Hence the LHS of (\ref{contributioncurvatureofNahmtransformvariationofDiracoperatorequation}) is
\[
\langle   G_\tau \{ 
I_1^{S^+}dt_2dt_3+I_2^{S^+}dt_3dt_1+I_3^{S^+}dt_1dt_2
\} ( [D^- ,     \hat{d}]+ \Omega+F(\alpha)^{(1,1)})\cdot\hat{f}^i, \wedge   \Omega'\cdot
\hat{f}^j  \rangle ,
\]
which vanishes by the delicate input (\ref{variationofcoupledDiracoperator}).
\end{proof}

Finally we arrived at the promised result

\begin{thm}
Under the Nahm transform, twisted adiabatic $G_2$ instantons on $\mathcal{F}\to M$ with the same slope potential as
$E\to M$ up to a constant, 
 map to  twisted adiabatic $G_2$ instantons on $\hat{\mathcal{F}}\to M^\vee$
with the same slope potential as $E'\to M^\vee$ up to a constant.
\end{thm}

\begin{proof}
Combining Proposition \ref{contributioncurvatureofNahmtransformvariationofDiracoperator} with Lemma \ref{contributioncurvatureofNahmtransformhatdsquaredterm1}, all the matrix entries of
\[
(I_1 dt_2dt_3+I_2dt_3dt_1+I_3dt_1dt_2)      \hat{F}^{(1,1)}
\]
vanish as required. Now combine with Lemma \ref{Nahmtransformpreservesadiabatic$G_2$instantononK3} to see the Theorem.
\end{proof}

\begin{rmk}
Recall the Nahm transform is defined if $\alpha$ is fibrewise HYM with the required slope potential, and the (twisted) adiabatic instantons are  critical points of a gauge-theoretic Chern-Simons type function. It seems likely that the Nahm tranform induces an equality between the Chern-Simons functionals for fibrewise HYM connections on $\mathcal{F}\to M$ and $\hat{\mathcal{F}}\to M^\vee$, which would give a deeper explanation of the Theorem.
\end{rmk}

\subsection{The inverse Nahm transform and duality}

The theory of  Nahm transform resembles the Fourier transform; in particular there is an analogue of the Fourier inversion theorem.
In the context of HYM connections on K3 surfaces, this is known as the inverse Fourier-Mukai transform in algebraic geometry, and is studied using derived category techniques \cite{Bartocci1}\cite{Huy2}\cite{Mukai1}. The differential geometric counterpart is the inverse Nahm transform on K3 surfaces, treated in our companion paper \cite{Mukaidualitypaper} using techniques similar to \cite{BraamBaal}.

Now let $(\mathcal{F}, \alpha)$  be a twisted adiabatic $G_2$ instanton over $M$ as in the setup of Section \ref{TheNahmtransform}, and let $(\hat{\mathcal{F}}, \hat{\alpha})$ be its Nahm transform, which is also a twisted adiabatic $G_2$ instanton by Section \ref{Nahmtransformpreservesadiabatic$G_2$instantons}.
In order to view $M\to B$ as the Mukai dual fibration of $M^\vee\to B$ , we assume as in Section \ref{Theuniversalconnection1triholomorphicproperty}  that for each $b\in B$, the family of HYM connections induced by $\nabla^{univ}$ on $E'|_b\simeq\mathcal{E}^\vee|_x\to X^\vee$ parametrised by $x\in X=M_b$ are all irreducible. Assume also that the fibrewise restriction $(\hat{\mathcal{F}}|_b,\hat{\alpha}|_b)$ of $(\mathcal{F}, \alpha)$,  which is a HYM connection over $X^\vee$, does not contain any HYM connection on $E'|_b\to X^\vee$ as an irreducible factor (in particular, this is true if $\hat{\alpha}$ is irreducible); this is to ensure some cokernel vanishing condition as in the beginning of Section \ref{TheNahmtransform}.

Under these assumptions, we can make the dualised $\nabla^{univ}$ on $\mathcal{E}^\vee\to M^\vee\times_B M$  play the role of the twisted generalised adiabatic $G_2$ instanton, to define the Nahm transform of $(\hat{\mathcal{F}}, \hat{\alpha})$, which is a twisted adiabatic $G_2$ instanton $(\hat{\hat{\mathcal{F}}}, \hat{\hat{\alpha}})$ over $M$ (here one uses the fact that the slope of $\mathcal{F}$ agrees with $E'$); we call $(\hat{\hat{\mathcal{F}}}, \hat{\hat{\alpha}})$ the \textbf{inverse Nahm transform}. For every fixed $b\in B$, this construction agrees with the setup of inverse Nahm transform on K3 surfaces in \cite{Mukaidualitypaper}. This implies the following fibrewise statement:

\begin{prop}
(\cf \cite{Mukaidualitypaper}) There is a canonical comparison map $\mathcal{F}|_b\to \hat{\hat{\mathcal{F}} }|_b$ over the K3 surface $X=M_b$, which is an isometric isomorphism, and identifies the irreducible HYM connection $\alpha|_b$ with $\hat{\hat{\alpha}}|_b$.
\end{prop}

The following analogue of \textbf{Fourier inversion} theorem can be viewed as a manifestation of the duality between $M$ and $M^\vee$.

\begin{thm}\label{Fourierinversion}
 The inverse Nahm transform $(\hat{\hat{\mathcal{F}}},\hat{\hat{\alpha}})$ is isomorphic to $(\mathcal{F},\alpha)$ up to possibly twisting by a $u(1)$-valued 1-form pulled back from $B$.
\end{thm}

\begin{proof}
From the fibrewise statement, the bundles $\mathcal{F}\to M$ and $\hat{\hat{\mathcal{F}}}\to M$ are identified as Hermitian vector bundles, and the connections $\alpha$ and $\hat{\hat{\alpha}}$ define the same Fueter section on some appropriate moduli bundle $\mathcal{M}\to B$ of HYM connections. An almost verbatim adaption of Proposition \ref{instantonFuetercorrespondenceuniqueness} to HYM situation gives the claim.
\end{proof}

\begin{rmk}
It is conceivable that $\alpha$ and $\hat{\hat{\alpha}}$ are canonically isomorphic without the extra twist, but we do not pursue this because the analogue in the K3 setting (\cf \cite{Mukaidualitypaper}) is already  difficult.
\end{rmk}

\end{document}